\documentclass[aap,preprint]{imsart}

\usepackage[authoryear]{natbib}
\usepackage[colorlinks,citecolor=blue,urlcolor=blue]{hyperref}
\usepackage{amsmath,amssymb,mathrsfs,amsthm,amsfonts,bbm, dsfont}
\usepackage{times}
\usepackage[shortlabels]{enumitem}
\usepackage{tabularx}
\usepackage{float, graphicx, subfigure}
\usepackage{multirow}
\usepackage{rotating}
\usepackage{verbatim}
\usepackage{color,soul}
\usepackage{blkarray, bigstrut}
\usepackage{mathtools}
\usepackage[normalem]{ulem}
\usepackage[makeroom]{cancel}

\allowdisplaybreaks

\startlocaldefs

\numberwithin{equation}{section}

\newtheorem{lemma}{Lemma}[section]
\newtheorem{theorem}{Theorem}%[section]
\newtheorem{definition}{Definition}[section]

\newtheorem{corollary}{Corollary}[section]

\newtheorem{proposition}{Proposition}[section]
\newtheorem{remark}{Remark}[section]

\newtheorem{example}{Example}[section]

\newtheorem{assumption}{Assumption}

\newcommand{\RR}{{\mathbb R}}
\newcommand{\ZZ}{{\mathbb Z}}
\newcommand{\NN}{{\mathbb N}}

\newcommand{\af}{\alpha}
\newcommand{\lm}{\lambda}
\newcommand{\ep}{\epsilon}

\newcommand{\ra}{\rightarrow}
\newcommand{\deq}{\stackrel{\rm d}{=}}
\newcommand{\qandq}{\quad\mbox{and}\quad}

\newcommand{\qforq}{\quad\mbox{for }}

\newcommand{\qasq}{\quad\mbox{as }}

\newcommand{\qforallq}{\quad\mbox{for all }}
\def\tinf{\rightarrow\infty}

\def\Ra{\Rightarrow}

\newcommand{\sB}{{\mathscr B}}
\newcommand{\bes}{\begin{equation*}}
\newcommand{\ees}{\end{equation*}}
\newcommand{\bequ}{\begin{equation}}
\newcommand{\eeq}{\end{equation}}
\newcommand{\bsplit}{\begin{split}}
\newcommand{\esplit}{\end{split}}
\newcommand{\bea}{\begin{eqnarray}}
\newcommand{\eea}{\end{eqnarray}}
\newcommand{\beas}{\begin{eqnarray*}}
\newcommand{\eeas}{\end{eqnarray*}}
\newcommand{\benum}{\begin{enumerate}}
\newcommand{\eenum}{\end{enumerate}}
\newcommand{\RNum}[1]{\uppercase\expandafter{\romannumeral #1\relax}}
\newcommand{\Prob}[1]{\mathbb{P}\left(#1\right)}
\newcommand{\E}{\mathbb{E}}
\newcommand{\R}{\mathbb{R}}

\newcommand{\of}[1]{\left(#1\right)}
\newcommand{\off}[1]{\left[#1\right]}
\newcommand{\offf}[1]{\left\{#1\right\}}
\newcommand{\arr}{\rightarrow}
\newcommand{\Arr}{\Rightarrow}

\endlocaldefs

\begin{document}

\begin{frontmatter}
\title{Asymptotic Optimality of the Binomial-Exhaustive Policy for Polling Systems with Large Switchover Times}
\runtitle{Asymptotically Optimal Control for Polling Systems}

\begin{aug}
	%%%%%%%%%%%%%%%%%%%%%%%%%%%%%%%%%%%%%%%%%%%%%%
	%%Only one address is permitted per author. %%
	%%Only division, organization and e-mail is %%
	%%included in the address.                  %%
	%%Additional information can be included in %%
	%%the Acknowledgments section if necessary. %%
	%%%%%%%%%%%%%%%%%%%%%%%%%%%%%%%%%%%%%%%%%%%%%%
	\author[\empty]{\fnms{Yue} \snm{Hu}\textsuperscript{$\ast$}\ead[label=e1]{yhu22@gsb.columbia.edu}},
	\author[\empty]{\fnms{Jing} \snm{Dong}\textsuperscript{$\dagger$}\ead[label=e2]{jing.dong@gsb.columbia.edu}}
	\and
	\author[\empty]{\fnms{Ohad} \snm{Perry}\textsuperscript{$\ddagger$}\ead[label=e3]{ohad.perry@northwestern.edu}}
	%%%%%%%%%%%%%%%%%%%%%%%%%%%%%%%%%%%%%%%%%%%%%%
	%% Addresses                                %%
	%%%%%%%%%%%%%%%%%%%%%%%%%%%%%%%%%%%%%%%%%%%%%%
	\address[]{\textsuperscript{$\ast$}Decision, Risk, and Operations, Columbia Business School,
		\printead{e1}}
	
	\address[]{\textsuperscript{$\dagger$}Decision, Risk, and Operations, Columbia Business School,
		\printead{e2}}
	
	\address[]{\textsuperscript{$\ddagger$}Department of Industrial Engineering and Management Sciences, Northwestern University,
		\printead{e3}}
	
	%\author{{Yue Hu}\thanks{Decision, Risk, and Operations, Columbia Business School, 3022 Broadway, New York, NY 10027 (\texttt{yhu22@gsb.columbia.edu})} ~~~~
	%{Jing Dong}\thanks{Decision, Risk, and Operations, Columbia Business School, 3022 Broadway, New York, NY 10027 (\texttt{jing.dong@gsb.columbia.edu})} ~~~~
	%{Ohad Perry}\thanks{Department of Industrial Engineering and Management Sciences, Northwestern University,
	%2145 Sheridan Rd., Evanston, IL 60208 (\texttt{ohad.perry@northwestern.edu})}}
	
\end{aug}

\begin{abstract}
We study an optimal-control problem of polling systems with large switchover times, when a holding cost is incurred on the queues.
In particular, we consider a stochastic network with a single server that switches between several buffers (queues) according to a pre-specified
order, assuming that the switchover times between the queues are large relative to the processing times of individual jobs.
Due to its complexity, computing an optimal control for such a system is prohibitive, and so we instead search for an asymptotically optimal control.
To this end, we first solve an optimal control problem for a deterministic relaxation (namely, for a fluid model),
that is represented as a hybrid dynamical system.
We then ``translate'' the solution to that fluid problem to a binomial-exhaustive policy for the underlying stochastic system, and
prove that this policy is asymptotically optimal in a large-switchover-time scaling regime,
provided a certain uniform integrability (UI) condition holds.
Finally, we demonstrate that the aforementioned UI condition holds in the following cases:
(i) the holding cost has (at most) linear growth, and all service times have finite second moments;
(ii) the holding cost grows at most at a polynomial rate (of any degree), and the service-time distributions possess finite moment generating functions.
%\yue{Our proofs that the UI condition holds in these two cases may be of an independent interest.}
\end{abstract}



\end{frontmatter}

\section{Introduction} \label{secIntro}

A polling system is a queueing network in which a single server attends multiple queues according to a pre-specified routing mechanism.
This class of models has been extensively studied since the 1950's, starting with
\cite{mack1957efficiency}, and have been employed in numerous application settings, such as
computer-communication \citep{bux1981local}, production \citep{federgruen1996stochastic,olsen1999practical}, inventory-control \citep{winands2011stochastic},
transportation \citep{altman1992elevator,van2006bounds}, and healthcare \citep{cicin2001application,vlasiou2009two}.
We refer to \cite{levy1990polling,takagi1997queueing,vishnevskii2006mathematical,boon2011open,borst2018polling}
for comprehensive reviews of the relevant literature.

Exact analysis of polling systems is in general prohibitively hard; \cite{resing1993polling} argues that,
unless the switching policy has a certain branching property, the system is not amenable to exact analysis.
However, even for those ``branching-type'' policies, results are typically expressed via multi-dimensional transforms that can be hard to analyze.
Thus, despite being among the most extensively studied class of stochastic networks \citep{boon2011applications},
%(it is estimated in \cite{boon2011applications} that over a thousand papers have been dedicated to these systems),
little is known about how to optimally control polling systems, except for special cases, such as the two-queue system in \cite{hofri1987optimal},
a symmetric cost structure in \cite{levy1990dominance,liu1992optimal}, or a limited control problem that is solved for a
subset of the queues in \cite{duenyas1996heuristic,van1997polling,matveev2016global}.
As will be seen below, due to the dimensionality and the switching dynamics of the queue process,
finding an optimal control is  a difficult problem even for deterministic polling systems.

\paragraph*{Scalings of the Switchover Times}
To achieve analytical simplification, it is sometimes assumed that the server's switchover times are instantaneous.
This assumption is reasonable to make when those switchover times are sufficiently small relative to the service times,
{\em and} the total traffic intensity is not too close to $1$.
(If the system is nearly critical, then even small deviations from the ``ideal'' modeling assumptions can have substantial negative impacts on the performance;
see the discussion in \cite[Section 9]{perry2016chattering}.)
However, switchover times are often quite large, and sometimes can even be considered to be of a {\em larger order of magnitude}
than the service durations, see, e.g., \cite{nahmias1984stochastic, federgruen1994approximating,olsen2001limit,winands2011stochastic}.
In such cases, one can turn the analytical disadvantage of having switchover times into an advantage by taking limits as those switchover
times increase without bound. (This approach is analogous to the one in which the switchover times are assumed to be instantaneous, which is in turn tantamount
to taking limits as the switchover times decrease to $0$.)
This large-switchover-times asymptotic approach was taken in \cite{olsen2001limit, van1999delay, winands2007polling, winands2011branching}
to approximate stationary performance measures, and was identified as an important future direction in \cite{boon2011open,borst2018polling}.
The same approach is taken here to solve an optimal-control problem in an appropriate asymptotic sense that will be
explained below.

\subsection{Optimal Control of Stochastic Networks} \label{secFCPscheme}

A standard approach in the stochastic-network literature to solving optimal-control problems follows an asymptotic scheme that was first
proposed by \cite{harrison1988brownian}. This approach can be roughly summarized as follows:
(I) formulate and solve a Brownian control problem (BCP) inspired by a heavy-traffic approximation for the system;
(II) ``translate'' the resulting optimal Brownian control to a control for the stochastic system;
and (III) prove that the control for the system is asymptotically optimal in an appropriate sense.

In this paper, we follow the main line or reasoning of the above scheme, but with important differences.
First, instead of a BCP, we solve a fluid-control problem (FCP) related to the stochastic control problem.
Moreover, an important step in solving the BCP in Harrison's scheme involves solving an {\em equivalent workload formulation}
that is rigorously achieved for the controlled stochastic system
by showing that {\em state-space collapse} (SSC) holds asymptotically, namely, that the limit diffusion process
is confined to a subspace having a lower dimension than that of the prelimit; see, e.g., \cite{harrison1997dynamic}.
Under fluid scaling, SSC corresponds to sliding motion of the fluid limit on a lower-dimensional manifold,
as explained in \cite[Section 1]{perry2016chattering}.
For a specific example, see \cite{atar2011asymptotic}, which considers the problem of asymptotically minimizing long-run average costs
in an overloaded many-server fluid regime. (There is customer abandonment with rate $\theta_i$ in queue $i$, keeping the system stable despite being overloaded.)
The proposed $c\mu/\theta$ priority rule induces asymptotic SSC in the stationary fluid model,
because all the queues that receive service, except the ones with the smallest $c\mu_i/\theta_i$ parameters, are asymptotically null in stationarity;
see Equation (18) in this reference.
However, SSC does not occur in our setting, because all the queues increase at fixed rates (their respective arrival rates)
during the switchover times, which are non-negligible in the fluid time scale.
In particular, under the large-switchover-times asymptotic, there is no reduction in the dimensionality of the limiting process,
implying that the dynamics of the fluid limits are necessarily discontinuous. (In fact, the fluid limits may not
even exhibit continuous dependence on their initial condition, and so do not adhere to the classical definition of well-posed dynamical systems.)

We remark that asymptotic SSC can occur in polling systems if the switchover times are small relative to the time scaling used to derive the limiting process.
Such SSC following from the averaging principles are proved in \cite{coffman1995polling} for polling systems with zero switchover times,
and in \cite{coffman1998polling} for systems with positive switchover times that do not scale in the limit (i.e., are negligible asymptotically).

\paragraph*{The Optimal-Control Problem}
We consider the (asymptotically) optimal-control problem for a polling system in which the server moves
among the different queues in a fixed order that is specified by a table,
with the objective of minimizing a long-run average holding cost on the queue process.
In this setting, a dynamic control is a state-dependent server-routing policy which determines
when the server should switch away from its current queue to the next queue in the table.

Since solving the optimal-control problem is prohibitively hard, we seek a control that is optimal in an appropriate asymptotic sense.
To this end, we consider a sequence of systems under a functional weak law of large numbers (FWLLN) scaling, and analyze the resulting fluid limits
as solutions to a hybrid dynamical system (HDS).\footnote{We use the acronym HDS for both singular and plural forms (system and systems).}
We then identify an optimal fluid control for the HDS, which we ``translate'' to a control for the underlying polling system.
In particular, the control we propose is the well-known binomial-exhaustive policy whose specific control parameters are taken directly from the optimal
fluid control; see Sections \ref{secModel} and \ref{SecTranslatingPRC} for more details.
Finally, we prove that the binomial-exhaustive policy (with the fluid-optimal control parameters)
is asymptotically optimal under the fluid scaling, in that it asymptotically achieves a lower bound on the long-run average cost.

The proof of asymptotic optimality requires that the sequence of cumulative holding costs incurred over a table cycle in stationarity is uniformly integrable (UI).
We demonstrate that the UI condition holds in two important cases:
(i) when the holding cost is linear and the service times have finite second moments; and
(ii) when the holding cost grows at most at a polynomial rate (of any degree), and all service times have finite moment generating functions (m.g.f.).
%In either case, the proof of UI builds on showing the existence of finite moments for the queue process.
%The proof that the UI condition holds in these two cases is of independent interest, due to the extensive literature on numerical algorithms to compute moments for the queue process, even when those moments are not known to exist (as in the case of the binomial-exhaustive policy).
%\yue{In either case, the proof of UI builds on showing the existence of finite moments for the queue process, which is formally established in (Yue: cite it as a working paper here).}
%The working paper on the moments of the queues is of independent interest, due to the extensive literature on numerical algorithms to compute moments for the queue process, even when those moments are not known to exist (as in the case of the binomial-exhaustive policy).}

\subsection{Conventions About Notation}
%We list here general mathematical notation that is used throughout the paper.
%\yue{\sout{Model-specific notation is summarized in Appendix \ref{appNotation}.}}

All the random variables and processes are defined on a single probability space $(\Omega, \mathcal F, P)$.
We write $\E$ to denote the expectation operator, and $\E_\mathbb{P}$ when we want to emphasize that the expectation is with respect
to a specific probability measure $\mathbb P$.
We let $\RR$, $\ZZ$ and $\NN$ denote the sets of real numbers, integers and strictly positive integers, respectively,
$\ZZ_+ := \NN \cup \{0\}$, and $\RR_+ := [0,\infty)$. For $k \in \NN$, we let
$\RR^k$ denote the space of $k$-dimensional vectors with real components, and denote these vectors with bold letters and numbers; in particular,
we write $\boldsymbol 1 := (1,\dots, 1)$ for the vector of $1$'s.
We let $D^k$ denote the space of right-continuous $\RR^k$-valued functions (on arbitrary finite time intervals) with limits everywhere,
endowed with the usual Skorokhod $J_1$ topology; see Chapter 11 of \cite{whitt2002stochastic}.
We let $D := D^1$.
We use $C^k$ (and $C := C^1$) to denote the subspace of $D^k$ of continuous functions.
It is well-known that the $J_1$ topology relativized to $C^k$ coincides with the uniform topology on $C^k$, which is induced by the norm
\begin{equation*}
||x||_t := \sup_{0 \leq u \leq t} \|x(u)\|,
\end{equation*}
where $||x||$ denotes the usual Euclidean norm of $x \in \R^k$.
We use ``$\Ra$" to denote weak convergence of random variables in $\RR^k$, and of stochastic processes over compact time intervals.

For $f : \R^k \arr [0, \infty)$, $g : \R^k \arr [0, \infty)$ and $a \in \R_+^k \cup \{\infty\}$
we write $f(x) = O(g(x))$ as $x \arr a$ if $\limsup_{x \arr a} f(x)/g(x) < \infty$, and $f(x) = o(g(x))$ if  $\lim_{x \arr a} f(x)/g(x) = 0$.

Given a sequence of random variables $\{X^n : n \geq 1\}$ and a sequence of non-negative real numbers $\{a^n : n \geq 1\}$,
we write $X^n = O_p(a^n)$ if $||X^n||/a^n$ is stochastically bounded, i.e., for any $\ep >0$, there exist finite $M, N \in \NN$ such that $\Prob{||X^n||/a^n > M} < \ep$
for all $n \geq N$. We write $X^n = o_p(a^n)$ if $||X^n||/a^n$ converges to zero in probability,
and $X^n = \Theta_p(a^n)$ if $X^n$ is $O_p(a^n)$ but not $o_p(a^n)$.
We write that a sequence of stochastic processes $\{X^n : n \ge 1\}$ is $O_p(a^n)$, $o_p(a^n)$, and $\Theta_p(a^n)$
if the corresponding property holds for $\|X^n\|_t$ for any $t\in(0,\infty)$.

For $x,y \in \RR$, we write $x \wedge y := \max\{x, y\}$, $x \vee y := \min\{x, y\}$, and $x^+ := \max\{x, 0\}$.
For a function $x \in D$, $x(a-)$ denotes the left-hand limit at $a$, i.e., $x(a-) := \lim_{t \uparrow a} x(t)$ is the left-hand limit at the point $a$.
For a vector $\mathbf v \in \RR^\ell$, $\ell \in \NN$, we use $\mbox{dim}(\mathbf v)$ to denote the dimension of $\mathbf v$; namely, $\mbox{dim}(\mathbf v) = \ell$.

We use a ``bar" to denote fluid-scale quantities: $\bar X^n := X^n/n$ for a sequence of random variables $\{X^n : n \geq 1\}$,
and $\bar X^n(t) := X^n(nt)/n$, $t \ge 0$, if the $\{X^n : n \ge 1\}$ is a sequence of processes.

\subsection{Organization of the Paper}
The rest of the paper is organized as follows. In Section \ref{secModel} we introduce the model, the main results, and
a roadmap for our approach to proving those result.
In Section \ref{secFluidModel} we consider a deterministic relaxation (a fluid model) to the optimal control problem,
which is characterized as the set of solutions to an HDS. It is also shown that the fluid model is related to the sequence of stochastic systems
via functional weak laws.
In Section \ref{Sec:FluidControl} we analyze the FCP, propose an optimal fluid control, and establish important qualitative
properties of the fluid model under this control.
In Section \ref{secControlsForTheStochasticSystem} we relate the proposed fluid control to the binomial-exhaustive policty.
In Section \ref{secAsympOpt} we prove that the binomial-exhaustive policy is asymptotically optimal in our setting.
Section \ref{secProofsofMainResults} is dedicated to the proofs of the main theorems.
We conclude in Section \ref{secSummary}.
Complementary proofs appear in the appendix.

\section{Problem Formulation and Main Results} \label{secModel}
We consider a polling system with $K$ queues numbered $1,\dots, K$.
Customers (or jobs) arrive at queue $k \in \mathcal{K} := \{1,...,K\}$
according to a Poisson process with rate $\lambda_k > 0$, and wait for their turn to be served in a buffer with infinite capacity (so that no customers are blocked).
We refer to customers who arrive to queue $k$ as class-$k$ customers.
The service times for class-$k$ customers are independent and identically distributed (i.i.d.)\ random variables with mean $1/\mu_k < \infty$.
At this point, we do not impose any other assumptions on the service time distributions (other than assuming that they all have finite means),
but further assumptions on the existence of higher moments will be needed to prove the aforementioned UI condition.
We denote by $S_k$ a generic random variable that has the service time distribution of class-$k$ customers.

A single server visits the queues periodically according to a fixed order specified by a {\em table}.
In particular, the table consists of $I$ stages, $I \ge K$, and the queue to be served at each stage is defined by a polling function
$p: \mathcal I \arr \mathcal K$, for $\mathcal I := \{1,\dots, I\}$,
where $p(i)$ is the queue attended (polled) by the server at stage $i$, and $(p(i), i \in \mathcal{I})$ is the table.
Note that a queue may appear more than once in a table, in which case that queue is attended by the server in two or more nonconsecutive stages.
We refer to each such attendance as a {\em visit} (of the server to the queue).
We refer to the starting time of a visit as a polling epoch and the ending time of a visit as a departure epoch (of the server from the queue).
A {\em table cycle} is the time elapsed between two consecutive polling epochs of stage $1$ (the first visit to queue $p(1)$) in $\mathcal I$.
The table is said to be {\em cyclic} if all the queues appear in the table exactly once (so that
the server visits each queue exactly once in a table cycle), in which case $\mathcal K = \mathcal I$.

For $i \in \mathcal I$, we assume that the switchover time of the server from stage $i$ to stage $i+1$ is a random variable $V_i$
with mean $s_i := \E[V_i] < \infty$, that is independent of all other random variables and processes in the system.
We let $s := \sum_{i\in \mathcal{I}} s_i$ denote the total expected switchover time incurred within a table cycle, and assume that $s > 0$.

For a given table, the {\em switching policy} (the control) is the set of rules specifying when the server should switch from each stage to the next.
Note that if a queue is visited more than once in the table, then the control may prescribe a different switching rule for each visit.
In addition, we allow the switching policy to induce an {\em augmented table}
in which the queues appear in a periodic pattern that is an $L$ multiple of the pattern of the basic table, for some integer $L \ge 2$.
We refer to a switching policy inducing an $L$-cycle augmented table as an $L$-cycle control, and denote
the set of stages in that augmented table by $\mathcal{I}^L := \{1,...,IL\}$.
We refer to the original table $(p(i), i \in \mathcal I)$ as the {\em basic table} and to a corresponding control, whose switching rules are repeated
after $I$ stages, as a one-cycle control.

A {\em server cycle} is the time elapsed between two consecutive polling epochs of stage $1$ in $\mathcal I^L$.
Thus, a server cycle in an $L$-cycle control consists of $L$ table cycles (and the server cycle is equal to the table cycle under a one-cycle control).
Under an $L$-cycle control,
an $\ell$th table cycle is time elapsed between stage $1 + (\ell - 1)I$ and stage $1+ \ell I$ in the augmented table,
for $1 \leq \ell \leq L$. (That is, the time it takes the server to complete the $\ell$th basic table within the augmented table.)
Further, we let the polling function $p : \mathcal{I}^L \arr \mathcal{K}$ map a stage in the augmented table
to the queue being visited at that stage.
The corresponding augmented table is given by $(p(i), i \in \mathcal{I}^L)$.
The expected total switchover time in a corresponding server cycle is then $s L$.

For concreteness, consider a polling system with three queues ($K=3$) visited according to the basic table $(1,2,3,2,3)$.
The basic table contains five stages ($I=5$): queue $2$ is visited in stages $2$ and $4$,
and queue $3$ is visited in stages $3$ and $5$. Hence, $p(1) = 1$, $p(2) = p(4) = 2$ and $p(3) = p(5) = 3$.
Under a one-cycle control ($L$=1), the basic and $L$-cycle augmented table, as well as table and server cycles, are all equivalent notions.
In contrast, under a two-cycle control ($L=2$), the augmented table is given by $(1,2,3,2,3; \, 1,2,3,2,3$), so that queue $1$ is visited twice,
and queues $2$ and $3$ are each visited four times during a server cycle, which now consists of two table cycles.
Of course, by a two-cycle control we mean that the switching rule of queue $p(i)$ is different than the rule of queue $p(i+5)$,
for at least one $i \in \{1,\dots, 5\}$.

\begin{remark} \label{remTable}
{\em The term ``basic table'' typically suggests that the order at which the server visits the queues has no repeated pattern.
It may therefore seem artificial to consider augmented tables with $L \ge 2$ consecutive repetitions of the basic table.
However, one cannot rule out at the outset the possibility that an $L$-cycle control, for some $L > 1$,
is better (reduces the cost) than a one-cycle control. Further, by considering $L$-cycle controls we can prove in some important special cases
that the asymptotically optimal one-cycle control is the overall asymptotically optimal control.}
\end{remark}

Let $\rho_k := \lm_k/\mu_k$ denote the traffic intensity corresponding to queue $k$, and let
\begin{equation*}
\rho := \sum_{k \in \mathcal{K}} \lambda_k / \mu_k.
\end{equation*}
We assume that $\rho < 1$, so that the system can be stabilized in the sense that there exist service policies
under which the queue process admits a stationary distribution \citep{fricker1994monotonicity,boon2011applications}.
We note that the system is stable under the binomial-exhaustive policy if and only if $\rho < 1$.

Let $U^{(0)}:=0$, and for $m \ge 1$, let $U^{(m-1)}$ denote the beginning of the $m$th server cycle
(end of the $(m-1)$st server cycle), where without loss of generality, we take time $0$ to be a polling epoch of stage $1$.
Let $A^{(m)}_{i}$ and $D^{(m)}_{i}$, $i \in \mathcal{I}^L$, denote the polling and departure epochs of stage $i$ during the $m$th server cycle.
Then, $B^{(m)}_{i} := D^{(m)}_{i} - A^{(m)}_{i}$ is the {\em busy time} at stage $i$ in the $m$th server cycle.
Lastly, let $T^{(m)}$ be the length of the $m$th server cycle, i.e.,
$T^{(m)}=U^{(m)} - U^{(m-1)}$ and $T^{(m)}=\sum_{i \in \mathcal{I}^L}(B^{(m)}_{i}+V_i^{(m)})$, where $V_i^{(m)} \stackrel{d}{=} V_i$.
Under a given switching policy $\pi$, we denote by $Q_{\pi, k}(t)$ the number of customers in queue $k$ at time $t$, $k \in \mathcal{K}$,
and let $Q_\pi(t) := ( Q_{\pi, k}(t), k \in \mathcal{K} )$, $t \geq 0$.

\paragraph*{The Optimal Control Problem}
Let $\psi: \RR_+^K \arr \RR_+$ denote the holding cost, so that $\psi(Q(t))$ is the cost incurred at time $t$ when
the state of the queue is $Q(t)$.
We assume that $\psi$ is non-negative, non-decreasing and continuous.
Our goal is to find an asymptotically optimal control $\pi$ within a family $\Pi$ of {\em admissible controls}
(see Definition \ref{def:MarkovPolicy} below and Section \ref{secControlsForTheStochasticSystem}),
that minimizes the following expected long-run average costs
\begin{equation} \label{objectiveExp}
\liminf_{t \rightarrow \infty}  \frac{1}{t} \,\, \E\off{\int_{0}^{t} \psi \of{Q_\pi(u)} du} \quad \text{and} \quad 	\limsup_{t \rightarrow \infty}  \frac{1}{t} \,\, \E\off{\int_{0}^{t} \psi \of{Q_\pi(u)} du} .
\end{equation}
%Recall that $U^{(m)}$ is the beginning of the $(m+1)$st server cycle, namely, the polling epoch of the first queue in the augmented table.
Let
\bequ \label{Q-DTMC}
\tilde Q_\pi(m) := Q_\pi(U^{(m)}), \quad m \ge 0.
\eeq
The service policies we consider are state-dependent controls that may depend on the value of $\tilde Q_\pi(m)$, such that
the process $\tilde Q_\pi := \{\tilde Q_\pi(m) : m \ge 0\}$ is a discrete-time Markov chain (DTMC) (see Lemma \ref{lmMarkovControl} below),
and is therefore regenerative whenever it is positive recurrent (as must be the case under an optimal policy). We refer to this DTMC as {\em the embedded DTMC},
and remark that, under the asymptotically optimal control we propose, namely, under the binomial-exhaustive policy,
the embedded DTMC is ergodic; see \cite{fricker1994monotonicity}.

\paragraph*{The Family of Admissible Controls} \label{secAdmissible control}
We say that a switching control is non-idling if the server does not idle while attending a non-empty queue, and in addition, it
switches immediately to the next queue in the table if it empties the attended buffer.
On the other hand, if the server finds a buffer empty upon its polling epoch, we allow it to wait for work to arrive.
This latter event is asymptotically null, because the server always finds a queue upon arrival to a buffer in the fluid limits,
and thus has no impact on our asymptotic analysis.
Let $\{\mathcal F_t : t \ge 0\}$ denote the $\sigma$-algebra generated by the queue process.

\begin{definition} [admissible control] \label{def:MarkovPolicy}
A switching control is admissible if

(i) The policy is non-idling.

(ii) For $i \in \mathcal{I}^L$, $m \geq 1$, the number of customers served during the busy time at the $i$th stage in the $m$th server cycle
conditional on $Q(A_i^{(m)})$ is independent of $\mathcal F_{A_i^{(m)}}$.

(iii) The policy is non-anticipative. %In particular, the switching rules depend on $\mathcal F_t$, and not on any future information.
\end{definition}

It is significant that the set of admissible controls contains a wide range of controls studied in the literature.
For example, the family of branching-type controls, which includes the exhaustive, gated, binomial-exhaustive, binomial-gated, and Bernoulli-type policies \citep{resing1993polling,levy1988optimization,levy1989analysis}, limited-type policies \citep{boxma1986models,szpankowski1987ultimate}, and base-stock policies \citep{federgruen1996stochastic} are all admissible. More generally, all the policies studied in \cite{fricker1994monotonicity} are admissible.
In particular, the policies considered in \cite{fricker1994monotonicity} satisfy the three requirements in Definition \ref{def:MarkovPolicy}, in addition to
a certain stochastic-monotonicity condition that we do not impose; see Section 2 in this reference.

\paragraph*{The Binomial-Exhaustive Policy}
Let $Q_{p(i)}(A_i^{(m)})$ denote the number of customers in queue $p(i)$ upon its polling epoch in the $m$th server cycle,
$i \in \mathcal{I}^L$, $m \geq 1$. The binomial-exhaustive policy, which was proposed in \cite{levy1988optimization},
is fully characterized by two parameters: an integer $L$, that specifies the number of table cycles contained in a server cycle,
and a vector $\mathbf r=(r_1, \dots, r_{IL}) \in [0,1]^{IL}$, whose component $r_i$ is the ``success probability'' of the binomial
random variable corresponding to stage $i \in \mathcal I^L$.
Note that, if $\sum_{ \{i \in \mathcal{I}^L: p(i) = k \} } r_i = 0$,
then queue $k$ explodes since it never gets served. We therefore consider $\mathbf r$ to be an element in the set
\bequ \label{eq:r_condition_of_sum}
\mathcal{R} := \offf{\mathbf r \in  [0,1]^{IL} : \sum_{ \{i \in \mathcal{I}^L: p(i) = k \} } r_i > 0 \quad \text{for all } k \in \mathcal{K} }.
\eeq

\begin{definition} [binomial-exhaustive policy] \label{def:Binomial}
For $(L, \mathbf r) \in \NN \times \mathcal R$, $i \in \mathcal{I}^L$ and $m \geq 1$,
conditional on the event $\{Q_{p(i)}(A_i^{(m)}) = N\}$, $N \in \ZZ_+$,
the number of customers that the server leaves behind at the departure epoch of queue $p(i)$ is
$N - Y^{(m)}_i(N, r_i)$, where $Y^{(m)}_i(N, r_i)$ is a binomial
random variable with parameters $N$ and $r_i$, which is independent of all other random variables and processes.
\end{definition}
The binomial-exhaustive policy can equivalently be described as one in which the server performs an independent Bernoulli trial for each customer in the queue
at stage $i$, having ``success probability'' $r_i$.
If the outcome of that trial is a ``success,'' then the server serves that customer as well as all of the new arrivals during the service duration
of that customer.
Thus, the server attends the queue polled at stage $i$ in the $m$th server cycle for $Y^{(m)}_i(N, r_i)$ busy periods of an $M/G/1$
queue having arrival rate $\lm_{p(i)}$ and service rate $\mu_{p(i)}$.
%We also mention that the binomial-exhaustive policy is a branching-type control, as defined in \citep{resing1993polling}.

\subsection{The Large-Switchover-Time Asymptotic Regime} \label{secLargeSwitchoverTimeScaling}
To carry out our asymptotic analysis, we consider a sequence of systems indexed by $n \geq 1$,
and append a superscript $n$ to all random variables and processes that scale with $n$.
Let $V^n_i$ denote the switchover time from stage $i$ in system $n$.
Under the large-switchover-time scaling, we keep $\lm_k$ and $\mu_k$ fixed (they do not scale with $n$),
and impose the following assumptions on the sequence of switchover times.

\begin{assumption} \label{assum1}
$\bar V_i^n := V_i^n/n \Arr  s_{i} $ as $n \arr \infty$. Further,
$\E\off{V_i^n} = n s_i$\, for all $i \in \mathcal{I}$.
\end{assumption}
We make two remarks: First, we allow $V^n_i = o_p(n)$ for some, but not all, $i \in \mathcal I$, so that $s_i = 0$, but $s > 0$.
Second, the latter part of Assumption \ref{assum1} can be easily relaxed to $\E\off{V_i^n}/n \ra s_i$ as $n\tinf$. However, this relaxation
comes at the expense of more cumbersome notation in some proofs, and has no practical significance (for the actual stochastic system under consideration).

Under the large-switchover-time scaling, the server spends $\Theta_p(n)$ time switching, so that the queues at polling epochs are also of order $\Theta_p(n)$,
namely, the queue process is strictly positive in fluid scale.
Recall that fluid-scaled quantities (random variables, processes and parameters) are denoted with a bar, e.g., $\bar Q^n_{\pi^n}(t) := Q^n_{\pi^n}(nt)/n$.
Let
\begin{equation*} \label{eqLongRunCost}
\bar C^n_{\pi^n}(t) := \frac{1}{t}  \, \int_{0}^{t} \psi \of{\bar Q_{\pi^n}^n(u)} du  ,
\end{equation*}
%Under the fluid scaling, the objective is to minimize
%\begin{equation*}
%\liminf_{t \arr \infty} \, \E\off{ \bar C^n_{\pi^n}(t) } \qandq \limsup_{t \arr \infty} \, \E\off{ \bar C^n_{\pi^n}(t) } ,
%\end{equation*}
where $\pi^n$ is the control employed in the $n$th system (and is allowed to depend on $n$).

We say that a sequence of controls $\boldsymbol{\tilde \pi_*} := \{\tilde \pi_*^n : n \ge 1\}$ is asymptotically optimal if
\begin{equation} \label{asyOpt}
\limsup_{n\tinf} \, \limsup_{t\tinf} \,  \E\off{ \bar C^n_{\tilde \pi^n_*}(t) }  \leq
\liminf_{n\tinf} \, \liminf_{t\tinf} \,  \E\off{ \bar C^n_{\pi^n}(t) }   ,
\end{equation}
for any other sequence of admissible controls $\boldsymbol \pi := \{\pi^n : n \ge 1\}$.

%Under the fluid scaling, the objective is to minimize
%\begin{equation} \label{eqLongRunCost}
%\bar C^n_{\pi^n} := \lim_{t \rightarrow \infty}  \frac{1}{t} \int_{0}^{t} \psi \of{\bar Q_{\pi^n}^n(u)} du,
%\end{equation}
%where $\pi^n$ is the control employed in the $n$th system (and is allowed to depend on $n$).
%
%We say that a sequence of controls $\boldsymbol{\tilde \pi_*} := \{\tilde \pi_*^n : n \ge 1\}$ is asymptotically optimal if
%\begin{equation} \label{asyOpt}
% \limsup_{n\tinf} \bar C^n_{\tilde \pi^n_*} \le \liminf_{n\tinf} \, \bar C^n_{\pi^n},
%\end{equation}
%for any other sequence of admissible controls $\boldsymbol \pi := \{\pi^n : n \ge 1\}$.

\begin{remark} \label{RemOnControl}
{\em %It is helpful to keep in mind that there is a fundamental difference between the two sequences of controls in the two sides of the inequality \eqref{asyOpt}:
Since we seek an effective control for a given stocahstic system,
the sequence $\boldsymbol{\tilde \pi}_*$ of asymptotically optimal controls should be considered to be a single control whose parameters may depend on $n$.
For example, if a threshold-type control is exercised, then the control parameters (the thresholds) must increase linearly with $n$
in order to appear in the fluid limits.
Hence, there is no abuse of terminology in saying that a control (as opposed to a sequence of controls) is asymptotically optimal.
On the other hand, the elements of $\boldsymbol \pi$ are allowed to change arbitrarily with $n$.}
\end{remark}

\subsection{Summary of Main Results} \label{secSumMain}
Our main result establishes that the binomial-exhaustive policy, with properly selected parameters $(L_*, \mathbf r_*)$,
is asymptotically optimal (among the set of admissible controls) for a large family of cost functions $\psi$.
The specific control parameters $(L_*, \mathbf r_*)$ are computed by solving a corresponding FCP, as will be explained below,
and are referred to as the {\em optimal (control) parameters}.
We thus denote the sequence of binomial-exhaustive policies by $\boldsymbol \pi_*$. Note that the same control parameters $(L_*, \mathbf{r}_*)$ are used for all $n \ge 1$;
in particular, the same control is considered for all the systems along the sequence. This property of the asymptotically optimal control we propose
is attractive, because applying the control in a given system can be done directly, without any engineering considerations which are often needed
in order to determine the size of the control parameters for a specific system.

To formally state our main result, let $\bar T^n := T^n/n$, where $T^n$ is the length of the {\em stationary} server cycle in the $n$th system,
which is finite w.p.1 when the embedded DTMC is positive recurrent.
For each $n \geq 1$ and control $\pi^n$, let
\begin{equation} \label{Psi}
\bar \Psi^n_{\pi^n} := \int_{0}^{\bar T^n} \psi (\bar Q^n_{\pi^n}(u)) du
\end{equation}
denote the cumulative fluid-scaled cost under $\pi^n$ over a stationary server cycle, namely, when $\bar Q^n_{\pi^n}(0)$
is distributed according to a stationary distribution of the embedded DTMC.
Let $(L_*, \mathbf r_*)$ be the optimal FCP parameters, and $c_*$ be the optimal objective value of the FCP. In addition, let $\pi_*^n$ be the binomial-exhaustive policy with these parameters (which are fixed along the sequence).
The following theorem is the main result of the paper.

\begin{theorem} \label{thmMain}
If $\{\bar \Psi^n_{\pi_*^n} : n \geq 1\}$ is UI, then for any sequence of admissible controls $\boldsymbol \pi$,
\begin{equation} \label{eq:eqMain}
\liminf_{n\tinf} \, \liminf_{t \arr \infty} \, \bar C^n_{\pi^n}(t) \ge \lim_{n\tinf} \lim_{t \arr \infty} \bar C^n_{\pi^n_*}(t) = c_* \quad w.p.1.
\end{equation}
\end{theorem}

\begin{proof}
The result follows from Theorems \ref{thm:AsympLB} and \ref{thm:AsympOpt} in Section \ref{secAsympOpt}.
\end{proof}

The next corollary is a simple consequence of Theorem \ref{thmMain}.

\begin{corollary} \label{coro:AsyOpt}
If $\{\bar \Psi^n_{\pi_*^n} : n \geq 1\}$ is UI, then $\boldsymbol{\pi}_*$ satisfies \eqref{asyOpt}, i.e., it is asymptotically optimal.
\end{corollary}

\begin{proof}
It follows from \eqref{eq:eqMain} by applying Fatou's Lemma twice that
\begin{equation*}
\liminf_{n\tinf} \, \liminf_{t \arr \infty} \, \E\off{ \bar C^n_{\pi^n}(t) } \ge \E\off{ \lim_{n\tinf} \lim_{t \arr \infty} \bar C^n_{\pi^n_*}(t) } = c_* .
\end{equation*}
Moreover, for each $n \geq 1$, the embedded DTMC $\tilde Q_{\pi_*^n}$ is ergodic by \citep[Proposition 1]{fricker1994monotonicity},
so that $\bar C^n_{\pi^n_*}(t)$ converges to a deterministic finite value w.p.1 as $t \arr \infty$ (see  \eqref{eq:11} in Section \ref{secProofThm4}
for a characterization of this constant).
Since $\bar C^n_{\pi^n_*}(t)$ is continuous in $t$, it holds that $\E\off{\sup_{t\geq0} \bar C^n_{\pi^n_*}(t)}< \infty$.
Thus,
\begin{equation*}
\lim_{n\tinf} \lim_{t \arr \infty} \E\off{  \bar C^n_{\pi^n_*}(t) } =\lim_{n\tinf} \E\off{ \lim_{t \arr \infty} \bar C^n_{\pi^n_*}(t) } = c_* ,
\end{equation*}
where the first equality follows from the dominated convergence theorem, and the second equality follows from \eqref{eq:eqMain}
and the fact that $\lim\limits_{t \arr \infty} \bar C^n_{\pi^n_*}(t)$ is a constant w.p.1.
\end{proof}

To apply Theorem \ref{thmMain} (and Corollary \ref{coro:AsyOpt}) we must (i) compute the fluid-optimal control parameters $(L_*, \mathbf r_*)$, which are also the
parameters of the binomial-exhaustive policy $\pi^n_*$ for all $n \ge 1$, and (ii) establish that the UI condition holds.
We now discuss these two conditions, starting with the latter.

\paragraph*{The UI Condition in Theorem \ref{thmMain}}
Theorem \ref{th:MomentCheck} below provides sufficient conditions for $\{\bar \Psi^n_{\pi_*^n} : n \geq 1\}$ to be UI, whenever the next assumption holds.

\begin{assumption} \label{assum2}
The following two conditions hold for all $i \in \mathcal I$.

(i) $\E\off{e^{t V_i^n}} < \infty$ for all $t \geq 0$ and $n \geq 1$.

(ii) $\E\off{(\bar V_i^n)^\ell} \arr s_i^\ell$ as $n \arr \infty$ for all $\ell \geq 2$.
\end{assumption}

Recall that $S_k$ denotes a generic random variable having the service time distribution of the class-$k$ customers, $k \in \mathcal K$.

\begin{theorem} \label{th:MomentCheck}
For $p \ge 1$, let $\psi(x) = O(\|x\|^p)$.
Under Assumption \ref{assum2}, $\{\bar \Psi^n_{\pi^n_*} : n \ge 1\}$ is UI if either of the following two conditions holds for all $k \in \mathcal K$.
\begin{enumerate}[(i)]
\item $p > 1$, and for some $\ep > 0$, $\E\off{e^{t S_k}} < \infty$ for all $t \in (-\epsilon, \epsilon)$.
\item $p=1$, and $\E\off{S_k^2} < \infty$.
\end{enumerate}
\end{theorem}
\begin{proof}
See Section \ref{ap:MomentCheck}.
\end{proof}
\noindent As an immediate corollary to Theorems \ref{thmMain} and \ref{th:MomentCheck}, we obtain that $\boldsymbol \pi_*$ is asymptotically optimal under either one
of the assertions in Theorem \ref{th:MomentCheck}.

We remark that the condition that the second moments of the service times are finite when $p=1$ is also necessary in order for the desired UI to hold; see Theorem 3 in \cite{Hu20Momentarxiv}.
%This claim follows from the proof of \yue{Lemma 4.1 in our working paper; see, in particular, Equations (4.7) -- (4.9) in the proof of this lemma.}
%Lemma \ref{lem:SecondMoment} in Appendix \ref{ap:SecondMoment}; see, in particular, Equations \eqref{eq:35}--\eqref{eq:34} in the proof of this lemma.
Thus, for cost functions that grow at most at a linear rate,
$\boldsymbol \pi_*$ is asymptotically optimal if and only if $\E\off{S_k^2} < \infty$ for all $k \in \mathcal K$.

\paragraph*{The Optimal Control Parameters}
Solving the FCP in order to compute the optimal control parameters $(L_*, \mathbf r_*)$ is not always feasible,
because it requires optimizing over the table structure (within the infinite set of all possible augmented tables) simultaneously
with optimizing the parameters. Nevertheless, in addition to solving the FCP on a case-by-case basis, it can also be solved in certain general settings.
%{\color{blue} such as the setting of Proposition \ref{prop:OneQueueMultiplePolls} below.} %in Appendix \ref{ap:FCPExamples}.
The most important case for which we can solve the FCP is that of cyclic basic tables, when the cost function $\psi$ is separable convex,
namely, is of the form $\psi(x) = \sum_{k=1}^K \psi_k(x_k)$, $x = (x_1, \dots, x_K)$, 
where $\psi_k : \RR_+ \ra \RR_+$ is convex for each $k \in \mathcal K$. %and for $x = (x_1, \dots, x_K)$.
In this setting, we prove that the exhaustive policy, under which the server empties the queues in all visits, is fluid optimal;
%In particular, the fluid-optimal control parameters in this case are $(L_*, \mathbf r_*) = (1, \mathbf 1)$, where $\mathbf 1$ is the unit vector in $\RR^K_+$;
see Proposition \ref{LEM:LINEARCYCLIC}. See also Corollary \ref{Cor:1} for the corresponding asymptotic-optimality result.

We also consider a restricted optimal-control problem for cases in which the FCP cannot be solved.
In the restricted problem, we optimize the control parameters for a finite set of values of $L$ (including the case $L=1$).
Unlike the FCP, the {\em restricted FCP} (RFCP) can always be solved, and the corresponding binomial-exhaustive policy is then asymptotically optimal
under the same conditions in Theorem \ref{th:MomentCheck}, although among a smaller family of admissible controls;
see Definition \ref{def:LCyclic} in Section \ref{secPiL} and Theorem \ref{th:AsympOptRelaxed} in Section \ref{secAsyOptRestrict}.

\subsection{Roadmap to the Proof of Theorem \ref{thmMain}} \label{Sec:Roadmap}

We now describe the main steps in the proof of Theorem \ref{thmMain}.
We emphasize that the description here is provided for overview, and is not meant to be fully rigorous. In addition,
the proof scheme outlined below is for the general case, in which we optimize over all possible augmented tables.
The aforementioned restricted problem follows a similar procedure, except that the
family of admissible controls is smaller for that latter problem.
%In particular, for the general problem we consider $L$-cycle controls for all $L \in \NN$, whereas in the restricted problem,
%we consider $L$-cycle controls, for all $L$ in some finite subset of $\NN$.

\paragraph*{(I) {\bf Formalizing an FCP (Sections \ref{secFluidModel} and \ref{sec:computePE})}}
To formalize an FCP corresponding to the control problem for the stochastic system, we consider a fluid model of the original stochastic system.
Specifically, we consider a deterministic polling system that has the same basic table as the stochastic system,
in which the arrival and service processes are replaced by deterministic continuous processes with the same rates $\lm_k$ and $\mu_k$, $k \in \mathcal K$.
In that deterministic counterpart, the queue process $Q$ is replaced by a {\em fluid model} $q := \{q(t) : t \ge 0\}$,
whose dynamics are determined by its initial condition and the switching policy.

We then seek a fluid control that minimizes the long-run average cost.
Let $q_\phi(t)$ denote the value of the fluid queue at time $t$ under control $\phi$, for $\phi$ in some appropriate
family of admissible fluid controls.
To apply the asymptotic-optimization framework, we want $q_\phi$ to be related to $Q_\pi$ via a FWLLN.
Next, %under any admissible control $\pi$, $Q_\pi(t)$ converges weakly to a stationary process and,
as will be shown in the proof of Theorem \ref{thm:AsympLB},
all possible fluid limits for $Q_\pi$ in stationarity are ``almost periodic'' in the sense that each such limit is arbitrarily close to a
{\em periodic equilibrium} (PE).\footnote{We use the acronym PE for both singular and plural forms, i.e., periodic equilibrium and periodic equilibria.}
(A fluid model $q_e$ is a PE if $q_e(t + \tau) = q_e(t)$ for all $t \ge 0$, for some $\tau > 0$; see Definition \ref{def:PE}.)
Thus, optimizing the long-run average cost is equivalent to first optimizing over all possible PE,
and then finding a control that guarantees convergence of the fluid model to the optimal PE.
In particular, we can take the set of admissible fluid controls, denoted by $\Phi$, to be the set of all controls
under which the fluid model converges to a PE, so that the FCP reduces to
\bes
\inf_{\phi \in \Phi} \, \lim_{t\tinf} {1 \over t} \int_0^t \psi(q_\phi(u)) du = {1 \over \tau_*} \int_{0}^{\tau_*} \psi(q_*(u)) du,
\ees
where $q_*$ is the optimal PE and $\tau_*$ is its period.

\paragraph*{(II) {\bf Solving the FCP (Section \ref{Sec:PRC})}}
From the description of step (I), solving the FCP consists of two components: first, we need to identify an optimal PE $q_*$,
and second, we need to design a control $\phi_* \in \Phi$ such that $q_{\phi_*}$ converges in an appropriate sense (see Definition \ref{defConvPE}) to $q_*$.
These two components are interconnected, because we have quite some flexibility in how we characterize the optimal PE.
In particular, the orbit of any PE $q_e$ is a loop (a closed curve)
in $\RR^K_+$, which %(since the dynamics of the fluid model are deterministic between switching epochs)
is fully characterized by specifying the server's departure epochs during that server cycle, together with a single point on the PE, because the
dynamics of the fluid model are deterministic between switching epochs.
A fluid control $\phi$ is then a switching rule that produces the desired trajectory $q_e$ whenever the initial point is on that PE's trajectory,
and is in $\Phi$ if it guarantees the desired convergence. (The main difficulty in establishing that a control $\phi_*$ is optimal
is in establishing that it is an element of $\Phi$.)

The fluid control $\phi_*$ we propose prescribes reducing queue $p(i)$ by a fixed proportion $r_i$ of its size at the polling epoch.
Specifically, letting the value of the fluid queue polled at stage $i$ be $q_{p(i)}(a_i^{(m)})$ at the polling epoch,
the server will switch away from that queue when its value reaches $(1-r_i) q_{p(i)}(a_i^{(m)})$, $i \in \mathcal{I}^L$, $m \geq 1$.
We refer to this control as stage-based proportion reduction (SB-PR), and to the SB-PR control with the optimal parameters $(L_*, \mathbf{r}_*)$
as the {\em optimal SB-PR}.

\paragraph*{(III) {\bf Proving asymptotic optimality (Sections \ref{secControlsForTheStochasticSystem} and \ref{secAsympOpt})}}
The translation step of the optimal SB-PR to the binomial-exhaustive policy is straightforward, and was discussed above and further detailed in Section \ref{secControlsForTheStochasticSystem}.
To prove that the binomial-exhaustive policy with parameters $(L_*, \mathbf r_*)$ is asymptotically optimal, we first show
(in Theorem \ref{thm:AsympLB}) that the limiting holding cost of any sequence of admissible controls is lower bounded by the optimal fluid cost.
We then prove that under the conditions in Theorem \ref{th:MomentCheck},
the binomial-exhaustive policy with the optimal SB-PR parameters achieves the lower bound asymptotically; see Theorem \ref{thm:AsympOpt}.

\section{The Fluid Model} \label{secFluidModel}
To formulate the FCP, we start by constructing a fluid model for the polling system.
To this end, we consider a deterministic polling system having the exact same system's topology and basic table as the stochastic system,
but in which arrivals and service completions occur continuously and deterministically at rates $\lambda_k$ and $\mu_k$, $k \in \mathcal{K}$, respectively.
Let $q(t)$ denote the fluid content at time $t$, and for $i \in \mathcal{I}^L$ and $m \geq 1$,
let $a^{(m)}_{i}$, $d^{(m)}_{i}$ and $b^{(m)}_{i}$ denote the polling epoch, departure epoch, and the busy time of stage $i$ during the $m$th server cycle.
Let $u^{(m-1)}$ be the time at which the $m$th server cycle begins, and $\tau^{(m)}$ be the length of the $m$th server cycle.
For the following, we write $q$ instead of $q_\phi$ to simplify the notations whenever the control is fixed, and refer to $q$ as the ``queue'' or ``fluid content'' interchangeably.

Let $\mathbf k= (i\in \mathcal{I}^L: p(i)=k)$ denote the vector of ordered stages at which queue $k$ is visited in a server cycle, so that
queue $k$ is visited a total of $\mbox{dim}(\mathbf  k)$ times over a server cycle.
Then the fluid queue over the first server cycle satisfies
\begin{equation*}
q_k(t) = q_k(0) + \lm_k t - \mu_k \sum_{j=1}^{dim(\mathbf  k)} \int_{0}^t \mathbf 1_{\left[a_{k_j}^{(1)}, \, d_{ k_j}^{(1)} \right)}(s) ds,
\quad k \in \mathcal{K}, \,\, t \in [u^{(0)}, u^{(1)}).
\end{equation*}
Since the fluid model is time-invariant, it can be described inductively via its dynamics over one server cycle;
in particular, the dynamics of $q_k$ over the time interval $[u^{(m-1)}, u^{(m)})$, namely, during the $m$th server cycle, can be described by
\begin{equation} \label{fluid implicit}
q_k(t) = q_k(u^{(m-1)}) +  \lm_k(t-u^{(m-1)}) - \mu_k \sum_{j=1}^{dim(\mathbf  k)} \int_{u^{(m-1)}}^t \mathbf  1_{\left[a_{ k_j}^{(m)}, \, d_{ k_j}^{(m)} \right)}(s) ds,
\end{equation}
for $t \in [u^{(m-1)}, u^{(m)})$, $k \in \mathcal{K}$, $m \geq 1$.

\subsection{The Fluid Model as a Hybrid Dynamical System} \label{secHDS}
%The explicit-looking representation of $q$ in \eqref{fluid implicit} is misleading, because the values of
Note that the values of $b^{(m)}_{i}$, $a^{(m)}_{i}$ and $d^{(m)}_{i}$ are determined by the state of $q$ and the control,
and are therefore not available a-priori (those values must be computed on the fly).
It is therefore more useful to represent $q$ as a solution to a differential equation.
To achieve such a representation, let $z(t)$ denote the location of the server at time $t$:
we write $z(t)=i$ if the server is actively serving queue $p(i)$ at time $t$, and $z(t)=\ominus_i$ if
the server is switching from stage $i$ to stage $i + 1$ at time $t$ (with $i + 1 := 1$ for $i = IL$). We let
\bequ \label{Zspace}
\mathcal{Z}:=\{1,\dots,IL,\ominus_1, \dots, \ominus_{IL}\}
\eeq
denote the state space of the server-location process $z$.

If a control depends only on the state of the queue process $q$ and the location of the server,
then we should keep track of the state of the process $(q,z)$ in order to determine the values of the switching times.
However, since $q$ is a ``surrogate'' for the stochastic process $Q$, and since we consider controls under which $\{\tilde Q(m) : m \ge 0\}$
in \eqref{Q-DTMC} is a DTMC, we also allow the control to depend on the value of the fluid queue at the last polling epoch prior to $t$, i.e., on $q(a(t))$,
where
\bes
a(t):=\max \, \{a_i^{(m)} \leq t : i\in \mathcal{I}^L, m \ge 1\}.
\ees
Thus, we consider the process
\bequ \label{x}
x(t) := (q(t), q(a(t)), z(t)), \quad t \ge 0,
\eeq
taking values in $\RR_+^K \times \RR_+^K \times \mathcal{Z}$.
Note that $x$ in \eqref{x} is a hybrid of the fluid-content process $q$, which has a continuous state space, and the server-location process $z$,
which has a finite state space,
and is therefore an HDS. (In fact, $x$ is a slight generalization of standard HDS due to the additional processes $q(a(t))$.)
Then $x$ is a solution to the following state equations.
\begin{equation} \label{HDSgeneral}
\begin{split}
\dot q(t) & = f(z(t)) \\
z(t) & = g(q(t), q(a(t-)), z(t-)), \\
a(t) &= h(q(t), q(a(t-)), z(t-)),
\end{split}
\end{equation}
where $f:  \mathcal Z \arr \R^K$,\, $g: \RR_+^K \times \RR_+^K \times \mathcal Z \arr \mathcal Z$, and $h: \RR_+^K \times \RR_+^K \times \mathcal Z \arr \R^K_+$
are the functions specified below.

First, the function $f$ determines the dynamics of the queues, which change at the polling and departure epochs of each stage.
Thus, for each $k \in \mathcal K$,
$f$ is defined via
\begin{equation*} \label{eq:f}
f_k(q(t), z(t)) =
\begin{cases}
\lambda_k -  \mu_k  \quad &\text{if } z(t) = i \,\, \text{and} \,\, p(i) = k \\
\lambda_k \quad &\text{otherwise}.
\end{cases}
\end{equation*}
The functions $g$ and $h$ are determined by the control; to characterize these function, we define
a {\em service function} $\phi_i : \R^K_+ \arr \R_+$ mapping the queue length at the polling epoch of stage $i$ to the immediate busy time of the server;
\begin{equation} \label{phi}
\phi_i ( q ( a_i^{(m)}  ) ) := b_{i}^{(m)}, \quad i \in \mathcal{I}^L, \,\, m \ge 1.
\end{equation}
The non-idling property we impose implies that
\begin{equation*}
b_{i}^{(m)} \leq q (a_i^{(m)})/(\mu_{p(i)} - \lambda_{p(i)}), \quad i \in \mathcal{I}^L, \,\, m \geq 1.
\end{equation*}
Indeed, the expression on the right-hand side of the above inequality is the time at which the fluid queue that is attended by the server hits state $0$
if the server keeps processing work continuously.

Now, the function $g$ characterizing the location of the server as follows: \\
(i) If $z(t-) = i$ and $q_{p(i)}(t) = q_{p(i)}(a(t-)) - (\mu_{p(i)}-\lambda_{p(i)}) \phi_i \of{q(a(t-))}$, define
\begin{equation*}
\begin{split}
j_s &:= \min \{j \geq i : s_{(j \text{ mod } IL)} > 0 \} \\
j_\phi &:= \min \{j > i : \phi_{(j \text{ mod } IL)} \of{q(a(t-))} > 0 \} .
\end{split}
\end{equation*}

(a) If $j_s < j_\phi$, then $g(q(t), q(a(t-)), z(t-)) = \ominus_{(j_s \text{ mod } IL)}$.

(b) Otherwise, $g(q(t), q(a(t-)), z(t-)) = (j_\phi \text{ mod } IL)$.

\noindent (ii) If $z(t-) = \ominus_i$ and $q_{p(i+1)}(t) = q_{p(i+1)}(a(t-)) + \lambda_{p(i+1)} \of{ \phi_i \of{q(a(t-))} + s_i }$, define
\begin{equation*}
\begin{split}
j_s &:= \min \{j > i : s_{(j \text{ mod } IL)} > 0 \} \\
j_\phi &:= \min \{j > i : \phi_{(j \text{ mod } IL)} \of{q(a(t-))} > 0 \} .
\end{split}
\end{equation*}

(a) If $j_s < j_\phi$, then $g(q(t), q(a(t-)), z(t-)) = \ominus_{(j_s \text{ mod } IL)}$.

(b) Otherwise, $g(q(t), q(a(t-)), z(t-)) = (j_\phi \text{ mod } IL)$.

\noindent (iii) Otherwise, $g(q(t), q(a(t-)), z(t-)) = z(t-)$.

Lastly, the function $h$ updates the most recent polling epoch according to
\begin{equation*}  \label{eq:h}
\begin{split}
&h(q(t), q(a(t-)), z(t-)) = \\
&
\begin{cases}
t \quad &\text{if } z(t-) = \ominus_i \text{ and }q_{p(i+1)}(t) = q_{p(i+1)}(a(t-)) + \lambda_{p(i+1)} \of{ \phi_i \of{q(a(t-))} + s_i }\\
a(t-) \quad &\text{otherwise} .
\end{cases}
\end{split}
\end{equation*}

\subsection{Qualitative Behavior of the HDS}
Our qualitative analysis of the HDS relies on fundamental concepts defined in this section.

\begin{definition} [PE] \label{def:PE}
A solution $x_e$ to the HDS \eqref{HDSgeneral} is a PE if there exists $\tau > 0$ such that $x_e(t + \tau) = x_e(t)$ for all $t \geq 0$.
The smallest such $\tau$ is called the period.
\end{definition}
Note that a solution $x_e$ is a PE if and only if the {\em orbit} of $q_e$, namely, the image of $q_e$ in $\RR_+^K$, is a loop.
Thus, we will henceforth refer to the queue component $q_e$ as a PE.

\begin{definition} [$L$-cycle PE] \label{def:LCyclePE}
A solution $x_e$ to the HDS \eqref{HDSgeneral} is an $L$-cycle PE if $x_e(t + \tau_L) = x_e(t)$ for all $t \geq 0$,
where $\tau_L$ is its cycle length spanning $L$ table cycles.
\end{definition}
Clearly, the cycle length $\tau_L$ of a PE is an integer product of the period of that PE.
It follows from basic flow-balance equations that the cycle length of $L$-cycle PE satisfies
\begin{equation} \label{eq:57}
\tau_L = s L / (1-\rho).
\end{equation}
To see this, observe that the server must be working a fraction $\rho$ of the time, and is therefore switching between stages for a fraction $1-\rho$ of the time.
Since the total switchover time over $L$ table cycles is $s L$, it holds that $\tau_L (1-\rho) = s L$, from which \eqref{eq:57} follows.

\paragraph*{Stable PE}
The purpose of the fluid-optimal control is to steer every possible trajectory $q$ to a desired PE $q_*$.
It is significant that convergence of trajectories to a PE cannot occur in the Lyapunov sense,
i.e., it does not hold that $\|q(t)-q_e(t)\| \ra 0$ as $t \tinf$ for a trajectory $q$ that converges to the PE $q_e$.
Instead, convergence of $q$ to the PE $q_e$ is said to hold if the orbit of $q$ in $\RR_+^K$ ``spirals'' towards the closed orbit of $q_e$.
Recall that, without loss of generality, $u^{(0)} = 0$, namely, the beginning of the first server cycle is time $0$.
Similarly, we take $u_e^{(0)}=0$ for a PE $x_e$.

\begin{definition} [convergence to a PE] \label{defConvPE}
A solution $x$ to the HDS \eqref{HDSgeneral} is said to converge to a PE $x_e$ if\,
%$|| q(u^{(m)} + s) - q_e (s) ||_t \ra 0$
$|| q(u^{(m)} + \cdot) - q_e (\cdot) ||_t \ra 0$
as $m \tinf$, for all $t > 0$.
\end{definition}

A PE $q_e$ may be of several types; if any other trajectory in some neighborhood of $q_e$
converges to it, then $q_e$ is called a {\em stable limit cycle}. (It is unstable if the trajectories in its neighborhood are ``spirling'' away from it,
and semi-stable if some trajectories in its neighborhood converge, while other are repelled.)
For our optimality result, we require a stronger stability property to hold.
\begin{definition} [global limit cycle]
A PE $q_e$ of the HDS is said to be a global limit cycle if all the trajectories of the HDS converge to $q_e$.
\end{definition}

In ending we remark that determining the number of limit cycles of a dynamical system is in general a hard problem,
even in the classical setting of dynamical systems with continuous vector fields.
(For planar systems with a polynomial vector field of degree greater than $1$, this is part of Hilbert's 16th open problem, which is still unsolved.)
Further, HDS of the form \eqref{fluid implicit} can exhibit chaotic behavior, and in particular, possess infinitely many PE, {\em none of which is a limit cycle},
even when the continuous-state process $q$ is of a dimension as low as $3$; see \cite{chase1993periodicity}.
In contrast, the fluid model (and limit) under our proposed SB-PR control will be shown to possess a global limit cycle (which is necessarily unique).

\subsection{Fluid Limits and Their Relation to the Fluid Model} \label{sec:FluidLimit}

Whereas the fluid model is derived for deterministic polling systems, the {\em fluid limits},
namely, the subsequential limits of the sequence of fluid-scaled queue processes, may not be deterministic under an arbitrary sequence of controls.
A FWLLN holds, and the resulting fluid limit is deterministic, under an extra regularity condition; see Proposition \ref{thm:FWLLN} below.
Since the deterministic fluid model is the basis for solving the FCP and deriving the asymptotically optimal control, it is significant
that the FWLLN holds for the binomial-exhaustive policy.

Consider the stochastic polling system, and let $Z(t)$ denote the location of the server at time $t$, defined on the same state space $\mathcal{Z}$ in \eqref{Zspace};
that is, $Z$ is the stochastic counterpart of the server-location process $z$ in the fluid model.
For $t \ge 0$, let
\begin{equation} \label{eq: A(t)}
A(t) := \max\{A_i^{(m)} \leq t : i \in \mathcal{I}^L, \, m \geq 1 \}.
\end{equation}
We define the state-process (of the stochastic system)
\begin{equation*}
X(t) := ( Q(t), Q(A(t)), Z(t) ), \quad t \geq 0,
\end{equation*}
where we removed $\pi$ from the notation to simply it.

Let $\mathcal P_k := \{\mathcal P_k(t) : t \ge 0\}$ denote the Poisson arrival process to buffer $k$, and
let $\mathcal S_k := \{\mathcal S_k(t) : t \ge 0\}$ denote the {\em potential} service process in buffer $k$, namely, $\mathcal S_k(t)$ would be
the number of class-$k$ service completions by time $t$ if the server were to process work from queue $k$ continuously during $[0,t)$.
In particular,
\bes
\mathcal S_k(t) := \sup \left\{m \geq 1 : \sum_{j=1}^m S_k^{(j)} \le t\right\},
\ees
where $\{S_k^{(j)} : j \ge 1\}$ is a sequence of i.i.d.\ random variables distributed like $S_k$.
Then for $k \in \mathcal K$, %the sample paths of $Q_k$ satisfy
\begin{equation} \label{eq:samplepath1}
Q_k(t) = Q_k(0) + {\mathcal P}_k(t) - {\mathcal S}_k \of{ \sum_{m =1}^\infty \sum_{\ell=1}^{dim(\mathbf k)}
\int_{0}^{t} \mathbf  1_{ \left[ A_{ k_\ell}^{(m)},  D_{ k_\ell}^{(m)} \right) } (u) du }, \quad t \geq 0 .
\end{equation}

Now, consider the sequence of stochastic systems under the large-switchover-time scaling.
For the $n$th system, let $A_{i}^{(m), n}$ and $D_{i}^{(m), n}$ denote, respectively, the polling and departure epoch of stage $i$ in the $m$th server cycle,
$i \in \mathcal{I}^L$, $m \geq 1$.
The corresponding fluid-scaled server-switching epochs (arrival and departure epochs to and from the queues) are given by $\bar A_{i}^{(m), n} := A_{i}^{(m), n}/n$ and $\bar D_{i}^{(m), n} := D_{i}^{(m), n}/n$.
Analogously to \eqref{eq: A(t)}, we denote the most recent polling epoch prior to time $t$ in system $n$ via
\begin{equation*}
A^n(t) := \max\{A_i^{(m), n} \leq nt : i \in \mathcal{I}^L, \, m \geq 1 \}  , \quad t \geq 0 .
\end{equation*}
The fluid-scaled state-process is given by
\begin{equation*} \label{Xn}
\bar X^n(t) := ( \bar Q^n(t), \bar Q^n(\bar A^n(t)), Z(nt) ), \quad t \geq 0,
\end{equation*}
where $\bar Q^n(t) := Q(nt)/n$ and $\bar A^n(t) := A^n(t)/n$ (there is no spacial scaling of the process $Z(nt)$).

For $k \in \mathcal{K}$ and $n \geq 1$, define $\mathcal S^n_k(t) := \mathcal S(n t)$, $\mathcal P^n_k(t) := \mathcal P_k(nt)$,
$\bar{\mathcal S}_k^n(t) := \mathcal S_k^n(nt)/n$ and $\bar{\mathcal P}^n_k(t) := \mathcal P(nt)/n$.
Then the representation \eqref{eq:samplepath1} for the queue in the $n$th system becomes
\begin{equation} \label{eq:samplepath}
Q_k^{n}(t) = Q_k^{n}(0) + {\mathcal P}_k^{n}(t) - {\mathcal S}_k^{n} \of{ \sum_{m =1}^\infty \sum_{\ell=1}^{dim(\mathbf k)}
\int_{0}^{t} \mathbf  1_{ \left[ \bar A_{ k_\ell}^{(m), n}, \bar D_{ k_\ell}^{(m), n} \right) } (u) du }, \quad t \geq 0 .
\end{equation}

\begin{lemma} [tightness] \label{lem:Tightness}
If $\{\bar Q^n(0) : n \geq 1\}$ is tight in $\R_+^K$, then $\{\bar{Q}^n : n \geq 1\}$ is $C$-tight in $D^K$,
and the sample paths of its subsequential limits are of the form \eqref{fluid implicit}.
\end{lemma}

It is significant that, for $i \in \mathcal I$ and $m \ge 1$, the time epochs $u^{(m)}$, $a_{i}^{(m)}$ and $d_{i}^{(m)}$ of a subsequential limit of $\bar Q^n$
may be random variables, in which case that limit $q$ is stochastic.
However, Lemma \ref{lem:Tightness} states that, even in this case,
the evolution of a stochastic limit $q$ between any two consecutive server-switching epochs is deterministic,
and is characterized in \eqref{fluid implicit}.

\begin{proof}[Proof of Lemma \ref{lem:Tightness}]
Fix $T > 0$. Due to the scaling of the switchover times in Assumption \ref{assum1},
the number of server switchings in system $n$ over the time interval $[0, nT)$
is finite w.p.1 as $n\tinf$. Hence, the sequence of fluid-scaled server-switching epochs is tight in $[0, T)$.
In particular, any subsequence of the sequences $\{\bar A_{ k_\ell}^{(m), n} : n \ge 1\}$ and $\{\bar D_{ k_\ell}^{(m), n} : n \ge 1\}$
in \eqref{eq:samplepath} has a further converging sub-subsequence (for all $m$ and $ k_\ell$ for which there are infinitely many
elements of these sequences in $[0, T)$).
Now, the indicator functions in the time-changed service process in \eqref{eq:samplepath} are fixed at the value $0$ or at $1$
between any two consecutive server-switching epochs,
so that $\bar Q^n_k$ is a continuous mapping of its primitives between any two such switching epochs.
It follows from \cite[Theorem 13.6.4]{whitt2002stochastic} that
any subsequence of $\{\bar Q^n_k : n \ge 1\}$ for which all the server-switching epochs in $[0, T)$ converge, converges in $D^K$
to $q_k$ in \eqref{fluid implicit} as $n\tinf$.
\end{proof}

It follows immediately from the proof of Lemma \ref{lem:Tightness} that if the sequences of fluid-scaled server-switching epochs converge in $[0, T)$
for all $T > 0$, then $\bar Q^n_k \Ra q_k$ in $D^K$ as $n\tinf$, for $q_k$ in \eqref{fluid implicit}.
In fact, since the dynamics of the queues are deterministic between any two server-switching epochs, convergence of the server departure times
implies that the server arrival times also converge.
We therefore have the following FWLLN.

\begin{proposition} [FWLLN] \label{thm:FWLLN}
Assume that $\bar Q^n(0) \Ra q(0)$ in $\RR^K_+$ as $n\tinf$.
If $\bar D^{(m),n}_i \Ra d^{(m)}_i$ in $\RR_+^K$ for all $m \ge 1$ and $i \in \mathcal I^L$, then
$ \bar Q^n \Rightarrow q $ in $D^K$ as $n\tinf$, where each element $q_k$, $k \in \mathcal{K}$, of the vector process $q$ satisfies \eqref{fluid implicit}.
\end{proposition}
Note that if $q(0)$ and $d^{(m)}_i$ are deterministic for all $m \ge 1$ and $i \in \mathcal I^L$, then the fluid limit $q$
is the unique solution to an HDS of the form \eqref{fluid implicit}.

\section{The Fluid Control Problems} \label{Sec:FluidControl}
%Analogously to \eqref{eq:1}, we consider the problem of minimizing
%\bes
%C_\phi := \lim_{t\tinf} \frac{1}{t} \int_{0}^{t} \psi(q_\phi(u)) du,
%\ees
%for a control $\phi$ in an appropriate set of controls.
In this section we formally define the FCP, whose solution is an optimal fluid control for the family of all augmented tables,
and the restricted problem, namely the RFCP, whose solution is an optimal fluid control for a finite set of augmented tables.

\paragraph*{The FCP}
For the FCP, we consider the set $\Phi$ of controls for which the following holds for each control $\phi \in \Phi$:

\noindent(i) There exists a unique solution $q^\gamma_\phi := \{q^\gamma_\phi(t) : t \ge 0\}$ to the HDS \eqref{HDSgeneral}
under $\phi$ for any initial condition $\gamma \in \RR^K_+$.

\noindent(ii) Any solution $q^\gamma_\phi$ converges to a limit cycle as $t\tinf$.

For $\gamma \in \RR_+^K$, let
\bes
C_\phi(\gamma) := \inf_{\phi \in \Phi} \, \lim_{t\tinf}{1 \over t} \int_{0}^{t} \psi(q_\phi^\gamma(u))du.
\ees
\begin{definition}[fluid optimal control]
We say that $\phi_*$ is fluid-optimal if $C_{\phi_*}(\gamma) \le C_{\phi}(\gamma)$ for all $\phi \in \Phi$ uniformly in $\gamma$.
\end{definition}
The following lemma, whose proof appears in Section \ref{secProofsThm&Lem}, motivates searching for an ``optimal PE'',
namely, a PE that achieves the lowest possible time-average cost over its cycle length among all possible PE,
and then devising a control ensuring that that PE is a global limit cycle.
\begin{lemma} \label{lem:LongRunCost}
For $\phi \in \Phi$ and $\gamma \in \RR_+^K$, let $q^\gamma_\phi$ denote the unique solution to the HDS when control $\phi$ is exercised and when $q^\gamma_\phi(0) = \gamma$.
Let $q_e^\gamma$ denote the limit cycle to which $q^\gamma_\phi$ converges, and $\tau_e^\gamma$ denote its cycle length. Then
\begin{equation*}
\lim_{t\tinf} {1 \over t} \int_{0}^{t} \psi(q^\gamma_\phi(u)) du = {1 \over \tau_e^\gamma}\int_{0}^{\tau_e^\gamma} \psi(q_e^\gamma(u)) du.
\end{equation*}
\end{lemma}
Due to Lemma \ref{lem:LongRunCost}, the FCP is concerned with finding a control $\phi_*$ that achieves the optimal long-run average $c_*$, where
\bequ \label{FCP}
\bsplit
c_* := \inf_{\phi \in \Phi}C_\phi(\gamma) := \inf_{\phi \in \Phi} \, \lim_{t\tinf}{1 \over t} \int_{0}^{t} \psi(q_\phi^\gamma(u))du, \qforallq \gamma \in \RR_+^K.
\end{split}
\eeq
In turn, to solve the FCP, we seek a control $\phi_* \in \Phi$ under which there exists
a {\em global limit cycle} $q_*$, such that
\begin{equation} \label{optPE}
{1 \over \tau_*} \int_{0}^{\tau_*} \psi(q_*(u)) du \le {1 \over \tau_e} \int_{0}^{\tau_e} \psi(q_e(u)) du
\end{equation}
holds for any other PE $q_e$ (whose cycle length is $\tau_e$).
Note that both $\tau_*$ and $\tau_e$ in \eqref{optPE} are allowed to have any possible value of $\tau_L$ in \eqref{eq:57},
so that we are effectively optimizing the PE over all possible augmented tables.
%We thus refer to a solution to the FCP \eqref{FCP} as a {\em globally optimal solution}.

\paragraph*{Solving the FCP}
We start by identifying closed curves in $\RR^K_+$ which are possible solution to the HDS (namely, they can be obtained as a PE
under some control). We refer to each such closed curve $q_e$ as a {\em PE-candidate}, and treat it as a mapping from $[0, \tau_e]$ to $\R^K_+$
(where $q_e(0) = q_e(\tau_e)$).
We then optimize over all possible PE-candidates in order to find an {\em optimal PE-candidate} $q_*$ for which \eqref{optPE} holds.
Finally, we design an optimal control $\phi_* \in \Phi$ under which the optimal PE-candidate $q_*$ is a global limit cycle for the HDS,
so that \eqref{FCP} holds for any solution $q^\gamma_\phi$ to \eqref{HDSgeneral} with initial condition $\gamma \in \RR_+^K$.

We emphasize two points: (i) We do not rule out the possibility that, in general, the infimum $c_*$ is not achievable via a PE-candidate, namely,
that there exists no PE-candidate whose time-average cost over the cycle length is $c_*$.
(However, we are unaware of such pathological examples; we do not study this problem due to its impracticability, as explained in the next point.)
(ii) Computing a PE-candidate for which $c_*$ is attained is not always practically feasible, due to the need to optimize the table structure
among all the possible augmented tables. (Hence, proving that a given problem is well-posed may also be impractical.)

As was mentioned in Section \ref{secSumMain},
solving the FCP is possible for specific systems or in specific settings.
%; see for example the setting in Proposition \ref{prop:OneQueueMultiplePolls}.
The most important case for which the FCP can be solved is when the cost function is separable convex (including linear),
and the basic table is cyclic; see Proposition \ref{LEM:LINEARCYCLIC} and Corollary \ref{Cor:1} for the corresponding asymptotic-optimality result.

\begin{remark}[On the set $\Phi$]{\em
It is significant that the set of fluid limits is larger than the set of possible fluid models under $\Phi$. In particular,
fluid limits under a sequence of admissible controls can be non-stable, in the sense that they do not converge to a limit cycle, and can also be stochastic.
Thus, $\Phi$ is smaller than the set of possible controls for the fluid limits.
However, Theorem \ref{thm:AsympLB} in Section \ref{secAsympOpt} proves that $c_*$ in \eqref{FCP} is a lower bound
on the achievable costs asymptotically (as $n\tinf$), so that, it is sufficient to search for control in $\Phi$.
%is ``sufficiently large.''
}\end{remark}

\paragraph*{The RFCP}
When solving the FCP in \eqref{FCP} is not feasible, one can instead optimize among
all $L$-cycle PE for $L$ in some finite subset $\mathcal N \subset \NN$, e.g., $L \in \mathcal N=\{1, \dots, M\}$, where $M \ge 1$ is a finite integer.
To this end, we consider the RFCP, whose goal is to find $c_\mathcal N$, where
\begin{equation}  \label{FCPrelax}
c_\mathcal N := \min_{L \in \mathcal N} \, \inf_{\phi \in \Phi^L}C_\phi(\gamma) :=\min_{L \in \mathcal N} \,  \inf_{\phi \in \Phi^L} \, \lim_{t\tinf}{1 \over t} \int_{0}^{t} \psi(q_\phi^\gamma(u))du, \qforallq \gamma \in \RR_+^K,
\end{equation}
where $\Phi^L \subset \Phi$ is the set of all the controls under which any solution to the HDS \eqref{HDSgeneral} converges to an $L$-cycle limit cycle.
Correspondingly, for each $L \in \mathcal N$ we seek an optimal $L$-cycle PE-candidate $q_*^L$ such that the inequality
\begin{equation*}
{1 \over \tau_L} \int_{0}^{\tau_L} \psi(q^L_*(u)) du \le  {1 \over \tau_L} \int_{0}^{\tau_L} \psi(q^L_e(u)) du,
\end{equation*}
holds for any other $L$-cycle PE $q^L_e$. The solution to the RFCP is then
\bes
q_\mathcal N := \min_{L \in \mathcal N} q^L_*.
\ees
The procedure for solving the RFCP is similar to that of solving the FCP:
We start by computing an optimal PE-candidate for each $L \in \mathcal N$, and take the one with the lowest time-average cost over the cycle length
to be the optimal PE-candidate for the RFCP.
Letting $L_{\mathcal N}$ denote the number of table cycles contained in the cycle length of $q_\mathcal N$,
we then design a control $\phi_{\mathcal N}$ under which $q_\mathcal N$ is a global limit cycle for the HDS.
Unlike the FCP \eqref{FCP}, solving the RFCP is always feasible, because computing
an optimal PE-candidate $q_*^L$ for any fixed $L$, and therefore computing $q_\mathcal N$, is straightforward.

In ending we remark that $L=1$ should always be an element of $\mathcal N$, not only because it corresponds to the basic table, but also
because the period of an $L$-cycle PE can be smaller than $\tau_L$, i.e., the period might be $\tau_{L_2} < \tau_L$, with $L$ being divisible by $L_2$.
In particular, an optimal $L$-cycle PE with $L > 1$ may have period $\tau_1$.

\subsection{Computing an Optimal PE-Candidate} \label{sec:computePE}
We now discuss the first step in solving the FCP and RFCP, namely, characterizing an optimal PE-candidate.

\subsubsection{Optimal PE-Candidates for the FCP} \label{secFCPformal}
Let $\mathcal Q$ denote the set of all PE-candidates (of all possible cycle lengths $\tau_L$, $L \ge 1$).
When a solution to the FCP \eqref{FCP} exists, an optimal PE-candidate for this FCP solves the optimization problem
\begin{equation} \label{eq:FluidOpt}
\begin{split}
\min_{q_e \in \mathcal Q} \quad & \frac{1}{\tau_L} \int_{0}^{\tau_L} \psi(q_e(u)) \, du. \\
\end{split}
\end{equation}

Let $q_{exh}$ denote the one-cycle PE under the exhaustive policy in which the server empties the queue it attends and then switches to the next queue in the table
(the existence of such a PE is established in Lemma \ref{lem:bijective} below).
Recall that $\psi$ is separable convex if $\psi(x) = \sum_{k \in \mathcal{K}} \psi_k(x_k)$ for $x \in \RR_+^K$,
and $\psi_k$ is convex for each $k \in \mathcal{K}$.
\begin{proposition} \label{LEM:LINEARCYCLIC}
If the basic table is cyclic and $\psi$ is separable convex, then $q_{exh}$ is a solution to \eqref{eq:FluidOpt}.
\end{proposition}

\begin{proof}
See Appendix \ref{pf:linearcyclic}.
\end{proof}

\iffalse
A more general result is stated in the proposition below, which we bring without a proof, due to its length.
The proof can be found in \cite[Appendix B]{Hu20arxiv}.
\begin{proposition} \label{prop:OneQueueMultiplePolls}
If the basic table has exactly one queue that is polled more than once,
and if $\psi$ is separable convex, then $q_{exh}$ is a solution to \eqref{eq:FluidOpt}.
\end{proposition}
\fi

Whereas $q_{exh}$ is not a solution to \eqref{eq:FluidOpt} in general, as we show below,
it is easy to see that each queue must be emptied at least once in a PE-candidate
that solves \eqref{eq:FluidOpt}.
In particular, for $q_{*,k}$ denoting the $k$th component process of a solution $q_*$ to \eqref{eq:FluidOpt}, and $\tau_*$ denoting the period (or cycle length) of $q_*$,
it must hold that $q_{*,k}(t_k) = 0$, for some $t_k \in [0, \tau_*)$, $k \in \mathcal K$.
To see this, observe that a PE-candidate is completely determined by its initial condition and the busy times $(b_i, i \in \mathcal{I}^L)$.
Consider a PE-candidate $q_e^{(1)}$ in which the $k$th queue, denoted by $q_{e,k}^{(1)}$, is such that $q_{e,k}^{(1)}(t) > 0$ for all $t \in [0, \tau_e)$, where $\tau_e$
is the period of $q_{e,k}^{(1)}$. We can construct a PE-candidate $q_e^{(2)}$ that has lower cost than $q_{e}^{(1)}$ by taking
$$q_{e,k}^{(2)}(0) := q_{e,k}^{(1)}(0) - \min_{t \in [0, \tau_{*})}q_{e,k}^{(1)}(t), \quad q_{e,\ell}^{(2)}(0) := q_{e,\ell}^{(1)}(0) \, \text{ for }  \, \ell \neq k, \ell \in \mathcal K, $$
and giving $q_{e}^{(2)}$ the same busy times $(b_i, i \in \mathcal{I}^L)$ of $q_{e,k}^{(1)}$. Then $q_e^{(2)}(t) < q_e^{(1)}(t)$ for all $t \in [0, \tau_e)$,
and the same inequality holds for the corresponding costs, because $\psi$ is nondecreasing.

\subsubsection{Optimal PE-Candidates for the RFCP} \label{secOptimalPErelax}
Let $\mathcal Q^L$ denote the set of all PE-candidates having cycle length $\tau_L$.
To solve the RFCP in \eqref{FCPrelax}, we solve for the optimal $L$-cycle PE-candidate for each $L \in \mathcal N$, taking the
one that gives the overall minimal cost as the solution. To this end, we consider the following optimization problem.
\begin{equation} \label{eq:FluidOpt2}
\begin{split}
\min_{q_e^L \in \mathcal Q^L} \quad  &\frac{1}{\tau_L} \int_{0}^{\tau_L} \psi(q_e^L(u)) \, du, \quad \text{for some (fixed) } L \ge 1.
\end{split}
\end{equation}
Unlike \eqref{eq:FluidOpt}, problem \eqref{eq:FluidOpt2} always admits solution.

\begin{lemma} \label{lem:ExistenceOfSolution}
For any fixed $L \in \NN$, the optimization problem \eqref{eq:FluidOpt2} admits a solution $q^L_e$.
\end{lemma}
The proof of the lemma builds on Lemmas \ref{lem:bijective} and \ref{lem:PRC} which are stated below, and is therefore relegated to Section \ref{secProofsThm&Lem}.

An analogous result to Proposition \ref{LEM:LINEARCYCLIC} holds for general cost functions when $L=1$, due to the aforementioned fact that,
in an optimal PE-candidate, each queue must be exhausted at least once. We therefore have:
\begin{proposition} \label{cor:5.3}
If the basic table is cyclic, then $q_{exh}$ is a solution to \eqref{eq:FluidOpt2} with $L=1$.
\end{proposition}

Finally, to demonstrate that $q_{exh}$ is not an optimal PE-candidate in general, consider a system with three queues
and basic (non-cyclic) table $(1,2,3,2,3)$. %Note that the basic table is not cyclic because queues 2 and 3 are each visited twice within a table cycle.
We take $\lambda_k=2$, $\mu_k = 8$, $s_k = 2$ for $k = 1,2,3$, and $\psi$ to be linear with $c_1 = c_2 = 1$, and consider $\mathcal N = \{1\}$.
(We remark that the solution remains unchanged when we optimize over larger values of $L$;
we conjecture that the optimal one-cycle PE-candidate also solves \eqref{eq:FluidOpt}.)
If $c_3 > 4$, then it is optimal to {\em not exhaust} $q_2$ at stage $2$. Moreover, the proportion of fluid processed at stage $2$
is decreasing to $0$ as $c_3$ increases.
It is easy to explain why $q_2$ is not exhausted in one of its visits. Specifically,
as the holding cost of $q_3$ increases, it becomes more and more advantageous to keep this queue smaller at the expense of
making $q_2$ larger. This can be achieved while keeping $q_2$ (and its corresponding holding cost) bounded,
because $q_2$ is visited twice, so the server has an opportunity to exhaust it in a server cycle.

\subsection{The SB-PR Control} \label{Sec:PRC}
Consider an $L$-cycle PE-candidate $q^L_e$, and let $r_i$ denote the proportion by which the queue polled in stage $i \in \mathcal I^L$ is reduced.
In particular, with $a_i$ and $d_i$ denoting, respectively, the polling epoch and departure epoch of stage $i$,
\begin{equation} \label{r_i}
r_i :=
\begin{cases}
{q_{e,p(i)}^L(a_i) - q_{e,p(i)}^L(d_i) \over q^L_{e,p(i)}(a_i)} \quad &\text{if } q^L_{e,p(i)}(a_i) > 0 \\
0 & \text{otherwise}
\end{cases}, \quad
i \in \mathcal I^L .
\end{equation}
Clearly, one can always represent a PE-candidate via parameters $(L, \mathbf r)$, where $\mathbf r := (r_i, i \in \mathcal I^L)$ is a vector
whose component $r_i$ is defined in \eqref{r_i}.
The following lemma shows that the reverse is also true; its proof is deferred to Section \ref{secProofsThm&Lem}.
For a given system, recall $\mathcal R$ in \eqref{eq:r_condition_of_sum} and that $\mathcal{Q}$ is the set of all PE-candidates.

\begin{lemma} \label{lem:bijective}
For any $(L, \mathbf r) \in \NN \times \mathcal R$,
there exists a unique $L$-cycle PE-candidate $q_e^L$ such that \eqref{r_i} is satisfied.
In particular, the function $q_e^L \mapsto (L,\mathbf r)$ is a bijection between $\mathcal{Q}$ and $\NN \times \mathcal R$.
\end{lemma}
Lemma \ref{lem:bijective} motivates our proposed SB-PR control, which will be shown to be fluid optimal in Theorem \ref{corSBPRopt} below.

\begin{definition} [SB-PR control] \label{def:SB-PR}
Let $(L, \mathbf r) \in \NN \times \mathcal R$.
The SB-PR control with parameters $(L, \mathbf r)$ has the service function
\begin{equation*}
\phi_i(q)=r_{i}q_{p(i)}/ (\mu_{p(i)} - \lambda_{p(i)}), \quad i \in \mathcal{I}^L,
\end{equation*}
for $\phi_i$ in \eqref{phi}. In particular, at each stage $i$, the server reduces the polled queue to a proportion $1-r_i$ of its value at the
polling epoch of this stage.
\end{definition}

Let $(L_*, \mathbf r_*)$ denote the SB-PR control parameters corresponding to a solution to \eqref{eq:FluidOpt} (and optimal PE-candidate for the FCP),
and for $\mathcal N \subset \NN$, let $(L_\mathcal N, \mathbf r_\mathcal N)$ denote the SB-PR control parameters corresponding to a solution to \eqref{eq:FluidOpt2}
(an optimal PE-candidate for the RFCP).
%The following theorem establishes the optimality of SB-PR.

\begin{theorem}[optimality of SB-PR] \label{corSBPRopt}
SB-PR with parameters $(L_*, \mathbf r_*)$ is a solution to the FCP \eqref{FCP}.
Similarly, SB-PR with parameters $(L_\mathcal N, \mathbf r_\mathcal N)$ is a solution to the RFCP \eqref{FCPrelax}.
\end{theorem}

The proof of Theorem \ref{corSBPRopt} relies on the following lemma, which establishes, in particular, that any PE-candidate is a bona-fide PE
under the corresponding SB-PR control, and that this PE is a global limit cycle.
%The proof of the Lemma \ref{lem:PRC} appears in Section \ref{secProofsThm&Lem} below.

\begin{lemma}[global stability of SB-PR]  \label{lem:PRC}
Let $q_e^L$ be an $L$-cycle PE-candidate, and let $\mathbf r$ be the corresponding vector of ratios defined for $q_e^L$ via \eqref{r_i}.
Then $q_e^L$ is a global limit cycle for the HDS \eqref{HDSgeneral} under SB-PR with parameters $(L,\mathbf r)$.
\end{lemma}

\begin{proof}[Proof of Lemma \ref{lem:PRC}]
For the HDS under SB-PR with parameters $(L, \mathbf r) \in \NN \times \mathcal R$,
define the operator $\Gamma_i : \RR_+^K \arr \RR_+^K$, $i \in \mathcal{I}^L$,
mapping the queue length at the polling epoch of stage $i$ to that at the polling epoch of stage $i+1$.
Note that during the busy time of stage $i$, queue $p(i)$ decreases at rate $\mu_{p(i)}-\lambda_{p(i)}$,
and any other queue $k \neq p(i)$ increases at rate $\lambda_k$.
If $q$ is the queue length at the polling epoch of stage $i$, then the busy time at stage $i$ lasts for $r_i q_{p(i)}/(\mu_{p(i)}-\lambda_{p(i)})$ units of time,
which is the time it takes to reduce queue $p(i)$ to $(1-r_i) q_{p(i)}$.
During the switchover time from stage $i$ to stage $i+1$, each queue $k \in \mathcal K$ increases at rate $\lambda_k$, and the switching takes $s_i$ unit of time.

For $i \in \mathcal I^L$, let
\[\Gamma_i(q) := \mathcal{A}_i q + \mathcal{B}_i,\]
where $\mathcal{A}_i$ is the $K \times K$ square matrix and $\mathcal{B}_i \in \RR^K$ are given by
\begin{equation*}
\mathcal{A}_i :=
\begin{blockarray}{*{7}{c} l}
\begin{block}{*{7}{>{$\footnotesize}c<{$}} l}
&  &  & $p(i)$th column & &  & & \\
\end{block}
\begin{block}{[*{7}{c}]>{$\footnotesize}l<{$}}
1 & 0 & \cdots & \lambda_1 \frac{r_i}{\mu_{p(i)} - \lambda_{p(i)}} & \cdots & 0 & 0   \\
0 & 1 & \cdots & \lambda_2 \frac{r_i}{\mu_{p(i)} - \lambda_{p(i)}} & \cdots & 0  & 0 \\
\vdots & \vdots & \ddots & \vdots & \ddots & \vdots & \vdots \\
0 & 0 & \cdots & 1 - r_i & \cdots & 0  & 0 & $p(i)$th row,  \\
\vdots & \vdots & \ddots & \vdots & \ddots & \vdots & \vdots \\
0 & 0 & \cdots & \lambda_{K-1} \frac{r_i}{\mu_{p(i)} - \lambda_{p(i)}} & \cdots & 1 & 0   \\
0 & 0 & \cdots & \lambda_K \frac{r_i}{\mu_{p(i)} - \lambda_{p(i)}} & \cdots & 0 & 1   \\
\end{block}
\end{blockarray}
\quad
\mathcal{B}_i :=
\begin{bmatrix}
\lambda_1 s_i    \\
\vdots \\
\lambda_{p(i)-1} s_i   \\
\lambda_{p(i)} s_i  \\
\lambda_{p(i)+1} s_i      \\
\vdots \\
\lambda_{K} s_i      \\
\end{bmatrix} ,
\end{equation*}
so that
\begin{equation*}
\mathcal{A}_i q :=
\begin{bmatrix}
q_1 + \lambda_1 \frac{r_i q_{p(i)}}{\mu_{p(i)} - \lambda_{p(i)}}     \\
\vdots \\
q_{p(i)-1} + \lambda_{p(i)-1} \frac{r_i q_{p(i)}}{\mu_{p(i)} - \lambda_{p(i)}}     \\
(1-r_i) q_{p(i)} \\
q_{p(i)+1} + \lambda_{p(i)+1} \frac{r_i q_{p(i)}}{\mu_{p(i)} - \lambda_{p(i)}}     \\
\vdots \\
q_{K} + \lambda_{K} \frac{r_i q_{p(i)}}{\mu_{p(i)} - \lambda_{p(i)}}     \\
\end{bmatrix}.
\end{equation*}

Let $\Gamma':= \Gamma_{IL} \circ ... \circ \Gamma_1$ be the composition operator over one server cycle, namely, the operator mapping
the value of the queue at the beginning of a server cycle to its value at the beginning of the subsequent server cycle. Then
\[
\Gamma'(q)=\mathcal{A}' q + \mathcal{B}'
\]
for
\[
\mathcal{A}' := \mathcal{A}_{IL}\cdots\mathcal{A}_1 \quad  \mbox{and} \quad
\mathcal{B}' := \sum_{i=1}^{IL-1}\left(\prod_{j=i+1}^{IL}\mathcal A_{j}\right) \mathcal B_i + \mathcal B_{IL}.
\]

Since each of the operators $\Gamma_i$, $i \in \mathcal{I}^L$, is affine and {\em positively invariant}, the same is true for $\Gamma'$.
(An affine operator is positively invariant if it maps $\RR^K_+$ into itself; see \cite[p.10]{matveev2016global}.)
By Lemma 5.1 in \cite{matveev2016global}, if $\varrho(\mathcal{A'}) < 1$, where $\varrho(\mathcal{A'})$ denotes the spectral radius of the matrix $\mathcal A'$,
then the positively invariant affine operator $\Gamma'(q)$ is a contraction mapping in $\RR^K_+$.

Hence, we next show that $\varrho(\mathcal{A}') < 1$.
To this end, observe that $\mathcal{A}'$ does not depend on the switchover times, so that
if the switchover times in the system are changed, but the arrival and service rates are kept fixed, then the matrix $\mathcal A'$ remains unchanged.
In particular, the matrix $\mathcal A'$ does not change if the switchover times in the system under consideration
are modified to $s_i = 0$ for all $i \in \mathcal I^L$, with all other parameters remaining unchanged.

Consider an auxiliary system that has the same parameters as the system under consideration,
except that $(s_i, i \in \mathcal I^L) = \mathbf 0$, and denote its queue process by $q^{a} := \{q^{a} (t) : t \ge 0\}$.
Let $W(t) := \sum_{k \in \mathcal{K}} q^a_k(t) / \mu_k$ denote the total workload in this auxiliary system at time $t$.
Since all the switchover times are null, the server is busy at all times in the set $W_+ := \{t : W(t) > 0\}$, so that
\begin{equation*}
\dot{W}(t) = \sum_{\ell \in \mathcal{K}, \ell \neq k} \frac{\lambda_\ell}{\mu_\ell} + \frac{\lambda_k - \mu_k}{\mu_k} = \rho - 1 < 0, \quad t \in W_+.
\end{equation*}
Now, $\mathcal{A}'$ is a non-negative square matrix, and so by the Perron-Frobenius theorem (e.g., \cite[Chapter 8.3]{meyer2000matrix}),
it has a maximal eigenvalue which is strictly positive. This implies that $\varrho(\mathcal{A}')> 0$, and that
the eigenvector $v$ associated with $\varrho(\mathcal{A}')$ has strictly positive components.
Hence, the eigenvector $v$ is a legitimate state for the queue process $q^{a}$.

Take $q^{a}(0) = v$.
Then, at the end of the first server cycle, we have
$q^{a}(u^{(1)}) = \mathcal{A}' v = \varrho(\mathcal{A}') v$, with the second equality holding because $v$ and $\varrho(\mathcal A')$
are the associated eigenvector and eigenvalue of $\mathcal A'$.
In addition, the workload in the system changes from $W(0) = \sum_{k \in \mathcal{K}} v_k / \mu_k$
to $W(u^{(1)}) =\sum_{k \in \mathcal{K}} \varrho(\mathcal{A}') v_k/\mu_k$.
Since the workload process $W$ is strictly decreasing, it holds that
\begin{equation*}
\varrho(\mathcal{A}') \sum_{k \in \mathcal{K}} \frac{v_k}{\mu_k} < \sum_{k \in \mathcal{K}} \frac{v_k}{\mu_k},
\end{equation*}
so that $\varrho(\mathcal A') < 1$, from which it follows that $\Gamma'$ is a contraction mapping in $\RR^K$. In turn, under SB-PR
(with any control parameters $(L,\mathbf r) \in \NN \times \mathcal R$),
there exists a global limit cycle for the HDS\footnote{It is easily seen that the global limit cycle
under SB-PR for a system with zero switchover times is trivial, namely, a fixed point; in particular, the limit cycle for this system is the origin.}
\eqref{HDSgeneral} if (and only if) $\rho < 1$.
\end{proof}

As a consequence of Lemmas \ref{lem:LongRunCost} and \ref{lem:PRC}, we also have the following corollary, which in turn, implies the statement of Theorem \ref{corSBPRopt}.

\begin{corollary} \label{cor:LongRunCostSBPR}
Let $q_e^L$ be an $L$-cycle PE-candidate with ratios $\mathbf r$ in \eqref{r_i}.
Then, under SB-PR with parameters $(L, \mathbf r)$, it holds that
\begin{equation*}
\lim_{t\tinf} {1 \over t} \int_{0}^{t} \psi(q(u)) du = {1 \over \tau_L}\int_{0}^{\tau_L} \psi(q_e^L(u)) du,
\end{equation*}
for any solution $q$ to the HDS \eqref{HDSgeneral}.
\end{corollary}

\begin{proof}[Proof of Theorem \ref{corSBPRopt}]
The proof follows immediately from Corollary \ref{cor:LongRunCostSBPR} by taking the SB-PR control parameters to be
$(L_*, \mathbf r_*)$ for the FCP, or $(L_\mathcal N, \mathbf r_\mathcal N)$ for the RFCP.
\end{proof}

\subsection{Proofs of Lemmas \ref{lem:LongRunCost}, \ref{lem:ExistenceOfSolution} and \ref{lem:bijective}} \label{secProofsThm&Lem}

\begin{proof}[Proof of Lemma \ref{lem:LongRunCost}]
	Let $L_e$ denote the number of table cycles contained in the period of the limit cycle $q_{e}^\gamma$. Let $v^{(m-1)}$ denote the beginning epoch of the $((m-1)L_e +1)$th table cycle, $m \geq 1$.
	Define $\tilde T^{(m)} := v^{(m)} - v^{(m-1)}$. By construction,
	$\tilde T^{(m)}$ contains exactly $L_e$ table cycles.
	Since $q_{e}^\gamma$ is the limit cycle for $q_\phi^\gamma$,
	it follows that	for any fixed $\ep > 0$, there exists $N_\ep \ge 1$, such that, for all $m \ge N_\ep$,
	\bequ \label{bddPE}
	\| q_\phi^\gamma(v^{(m-1)} + \cdot) - q_{e}^\gamma(\cdot)\|_t < \ep \qforallq t > 0 , \quad  |\tilde T^{(m)} - \tau_e^\gamma| < \ep ,
	\eeq
	and
	\begin{equation*}
	\left|\int_{v^{(m-1)}}^{v^{(m)}} q_\phi^\gamma(s) ds - \int_{0}^{\tau_e^\gamma} q_{e}^\gamma(s) ds \right| < \ep .
	\end{equation*}
	Since the PE $q_e^\gamma$ is bounded (componentwise), \eqref{bddPE} implies that $q_\phi^\gamma$ is also bounded.
	Due to the continuity of $\psi$, $q_e^\gamma$ and $q_\phi^\gamma$, the composites $\psi \circ q_e^\gamma$ and $\psi \circ q_\phi^\gamma$
	are uniformly continuous over any compact time interval.
	It follows that for any $\ep >0$, there exists $M_\ep \ge N_\ep$, such that
	%for all $m \ge M_\ep$,
	\bequ \label{bddIntegral}
	\left|\int_{v^{(m-1)}}^{v^{(m)}} \psi(q_\phi^\gamma(s)) ds - \int_{0}^{\tau_e^\gamma} \psi(q_e^\gamma(s)) ds \right| < \ep, \qforallq m \ge M_\ep.
	\eeq
	Let $M(t) := \max\{m \ge 1 : v^{(m)} \le t\}$. Then
	\bes
	{1 \over t} \int_{0}^{t} \psi(q_\phi^\gamma(s)) ds = {1 \over t} \sum_{m=1}^{M(t)} \int_{v^{(m-1)}}^{v^{(m)}} \psi(q_\phi^\gamma(s)) ds + {1 \over t} \int_{v^{(M(t))}}^t \psi(q_\phi^\gamma(s)) ds.
	\ees
	Since $0 \le t - v^{(M(t))} \le \tilde T^{(M(t)+1)}$ and $\tilde T^{(M(t)+1)}$ is bounded by virtue of \eqref{bddPE},
	the second term on the right-hand side of the equality above converges to $0$ as $t\tinf$. Now, for all $t$ large enough, it holds that $M(t) > M_\ep$,
	so that
	\bes
	{1 \over t} \sum_{m=1}^{M(t)} \int_{v^{(m-1)}}^{v^{(m)}} \psi(q_\phi^\gamma(s)) ds =
	{1 \over t} \sum_{m=1}^{M_\ep - 1} \int_{v^{(m-1)}}^{v^{(m)}} \psi(q_\phi^\gamma(s)) ds + {1 \over t} \sum_{m=M_\ep}^{M(t)} \int_{v^{(m-1)}}^{v^{(m)}} \psi(q_\phi^\gamma(s)) ds.
	\ees
	For fixed $\ep > 0$, $M_\ep$ is fixed, so that the first term in the right-hand side of the equality converges to $0$ as $t\tinf$.
	Applying \eqref{bddIntegral}
	for the second term gives that for $t$ large enough, we get that
	\begin{equation} \label{eq:1}
	\begin{split}
	{1 \over t} \sum_{m=M_\ep}^{M(t)} \int_{v^{(m-1)}}^{v^{(m)}} \psi(q_\phi^\gamma(s)) ds
	& \le
	{M(t) \over t} {1 \over M(t)} \sum_{m=M_\ep}^{M(t)} \left(\int_0^{\tau_e^\gamma} \psi(q_e^\gamma(s)) ds + \ep \right) \\
	& = {M(t) \over t} {M(t) - M_\ep \over M(t)} \left(\int_0^{\tau_e^\gamma} \psi(q_e^\gamma(s)) ds + \ep \right) \\
	& \leq \of{\frac{1}{\tau_e^\gamma-\ep} + o(1) }  \left(\int_0^{\tau_e^\gamma} \psi(q_e^\gamma(s)) ds + \ep \right) \\
	& \ra {1 \over \tau_e^\gamma - \ep}  \of{\int_0^{\tau_e^\gamma} \psi(q_e^\gamma(s)) ds + \ep} \qasq t\tinf.
	\end{split}
	\end{equation}
%	\begin{align} \label{eq:1}
%	{1 \over t} \sum_{m=M_\ep}^{M(t)} \int_{v^{(m-1)}}^{v^{(m)}} \psi(q_\phi^\gamma(s)) ds
%	& \le
%	{M(t) \over t} {1 \over M(t)} \sum_{m=M_\ep}^{M(t)} \left(\int_0^{\tau_e^\gamma} \psi(q_e^\gamma(s)) ds + \ep \right) \\
%	& = {M(t) \over t} {M(t) - M_\ep \over M(t)} \left(\int_0^{\tau_e^\gamma} \psi(q_e^\gamma(s)) ds + \ep \right) \\
%	& \leq \of{\frac{1}{\tau_e^\gamma-\ep} + o(1) }  \left(\int_0^{\tau_e^\gamma} \psi(q_e^\gamma(s)) ds + \ep \right) \\
%	& \ra {1 \over \tau_e^\gamma - \ep}  \of{\int_0^{\tau_e^\gamma} \psi(q_e^\gamma(s)) ds + \ep} \qasq t\tinf.
%	\end{align}
	In the second inequality above, we have used the fact that
	\bes
	\bsplit
	{t \over M(t)} & = {1 \over M(t)} \left(\sum_{m=1}^{M_\ep-1} \tilde T^{(m)} + \sum_{m=M_\ep}^{M(t)}  \tilde T^{(m)} + \of{t-v^{(M(t))}} \right) \\
	& = {1 \over M(t)} \sum_{m=M_\ep}^{M(t)} \tilde  T^{(m)} + o(1) \\
	& \ge {M(t) - M_\ep \over M(t)} (\tau_e^\gamma - \ep) + o(1) \ra \tau_e^\gamma - \ep \qasq t\tinf.
\end{split}
\ees
It follows from \eqref{eq:1} that
\begin{equation*}
\begin{split}
\limsup_{t \rightarrow \infty} {1 \over t} \sum_{m=M_\ep}^{M(t)} \int_{v^{(m-1)}}^{v^{(m)}} \psi(q_\phi^\gamma(s)) ds
& \le
{1 \over \tau_e^\gamma - \ep}  \of{\int_0^{\tau_e^\gamma} \psi(q_e^\gamma(s)) ds + \ep} .
\end{split}
\end{equation*}
We can similarly show that
\bes
\liminf_{t\tinf} {1 \over t} \sum_{m=M_\ep}^{M(t)} \int_{v^{(m-1)}}^{v^{(m)}} \psi(q_\phi^\gamma(s)) ds \ge {1 \over \tau_e^\gamma + \ep} \of{\int_0^{\tau_e^\gamma} \psi(q_e^\gamma(s)) ds - \ep},
\ees
and so the statement follows by taking $\ep \ra 0$.
\end{proof}

\begin{proof}[Proof of Lemma \ref{lem:ExistenceOfSolution}]
Let $e_{i,j}$ denote the time elapsed between the departure epoch of stage $i$ and the polling epoch of stage $j$ in $q^L_e$, $i,j \in \mathcal{I}^L$,
namely,
\begin{equation*} \label{eq:V}
e_{i,j} :=
\begin{cases}
s_{i} + \sum_{\ell = i + 1}^{j-1} (s_{\ell} + b_{\ell})  \quad &\text{if } i < j \\
s_{i} + \sum_{\ell = i + 1}^{IL} (s_{\ell} + b_{\ell}) + \sum_{\ell = 1}^{i-1} (s_{\ell} + b_{\ell})  \quad &\text{if } i \geq j ,
\end{cases}
\end{equation*}
with $\sum_{\ell = \ell_1}^{\ell_2} (s_{\ell} + b_{\ell}) := 0$ for $\ell_1 > \ell_2$.
Then an $L$-cycle PE-candidate $q_e^L$ necessarily satisfies the following systems of equations
\begin{equation}  \label{eq:PRCRecursion}
q^L_{k} (a_{k_j}) (1- r_{ k_j})  + \lambda_k e_{ k_j, k_{j+1}} = q^L_{k} (a_{ k_{j+1}}) ,
\quad j =1,..., dim(\mathbf k), \,\, k \in \mathcal{K},
\end{equation}
where $ k_{dim(\mathbf k) +1 } := k_1$.
Since the $L$-cycle PE-candidate parameterized by $\mathbf r$ is unique by virtue of Lemma \ref{lem:bijective},
the linear system \eqref{eq:PRCRecursion} admits a unique solution.
Hence, solving \eqref{eq:PRCRecursion} at all possible value of $\mathbf r \in \mathcal{R}$ for the corresponding PE
gives the entire constraint set of \eqref{eq:FluidOpt2}, because for given $(L, \mathbf r)$, $q_e^L$ is determined
by the solution to \eqref{eq:PRCRecursion}, $(q^L_{e,k} (a_{ k_j}),  j =1,..., dim(\mathbf k), k \in \mathcal{K})$.
Thus, \eqref{eq:FluidOpt2} can be reformulated equivalently as follows.
\begin{equation*} \label{FCPrelax2}
\begin{split}
\min_{\mathbf r \in \mathcal{R}} \quad & \frac{1}{\tau_L} \int_{0}^{\tau_L} \psi(q_e^L(u)) \, du \\
s.t. \quad   & q^L_{e,k} (a_{ k_j}) (1- r_{ k_j})  + \lambda_k e_{ k_j,  k_{j+1}} = q^L_{e,k} (a_{ k_{j+1}}) ,
\quad  j =1,..., dim(\mathbf k), \, k \in \mathcal{K} \\
&q_e^L \text{ is determined by } \of{q^L_{e,k} (a_{ k_j}),  j =1,..., dim(\mathbf k), k \in \mathcal{K}}.
\end{split}
\end{equation*}
Now, as was explained in Section \ref{secFCPformal}, each queue in an optimal PE-candidate must be emptied at least once within a server cycle,
and so the vector $\mathbf r$ corresponding to an optimal PE-candidate is an element of the set
\begin{equation*}
\mathcal{R}' := \offf{\mathbf r \in  [0,1]^{IL} : \sum_{ \{i \in \mathcal{I}^L: p(i) = k \} } r_i \geq 1 \quad \text{for all } k \in \mathcal{K} } .
\end{equation*}
Note that $\mathcal R'$ is a compact subset of the (non-compact) set $\mathcal R$ in \eqref{eq:r_condition_of_sum}.

Finally, since \eqref{eq:PRCRecursion} is a system of linear equations for a given $\mathbf r$, its unique solution
$(q^L_{e,k} (a_{ k_j}),  j =1,..., dim(\mathbf k), k \in \mathcal{K})$, is continuous in $\mathbf r$.
It follows that, for a given $\ep >0$, there exists a $\delta >0$, such that for all $\mathbf r_1, \mathbf r_2 \in \mathcal{R}'$ and their corresponding PE-candidates $q^{L,(1)}_{e}, q^{L,(2)}_{e}$, if $||\mathbf r_1 - \mathbf r_2|| < \delta$, then $|\psi(q^{L,(1)}_{e}(u)) - \psi(q^{L,(2)}_{e}(u))| < \ep$
for all $u \in [0,\tau_L)$, so that
\bes
\left| \frac{1}{\tau_L} \int_{0}^{\tau_L} \psi(q^{L,(1)}_{e}(u))du - \frac{1}{\tau_L} \int_{0}^{\tau_L} \psi(q^{L,(2)}_{e}(u))du \right| < \ep.
\ees
Thus, we established an equivalent formulation for problem \eqref{eq:FluidOpt2},
in which the objective function is continuous over the compact constraint set $\mathcal{R}'$.
It follows from Weierstrass theorem that a global minimum exists.
\end{proof}

\begin{proof}[Proof of Lemma \ref{lem:bijective}]
It follows from the proof of Lemma \ref{lem:PRC} that under SB-PR with parameters $(L, \mathbf r) \in \NN \times \mathcal R$,
the HDS converges to a global limit cycle. Thus, an $L$-cycle PE $q_e^L$ that satisfies \eqref{eq:r_condition_of_sum} exists.
Moreover, this PE is a global limit cycle, and is therefore the unique PE characterized via $(L, \mathbf r)$.
The statement of the lemma follows, because a PE is a PE-candidate by definition.
\end{proof}

%\begin{proof}[Proof of Theorem \ref{corSBPRopt}]
%The proof follows immediately from Corollary \ref{cor:LongRunCostSBPR} by taking the SB-PR control parameters to be
%$(L_*, \mathbf r_*)$ for the FCP, or $(L_\mathcal N, \mathbf r_\mathcal N)$ for the RFCP.
%\end{proof}

\section{Translating SB-PR to the Stochastic System} \label{secControlsForTheStochasticSystem}

As discussed in Section \ref{secIntro}, we translate
SB-PR with control parameters $(L,\mathbf r)$ in the deterministic system to the binomial-exhaustive policy with the same control parameters in the stochastic system.
To show that the binomial-exhaustive policy with the optimal fluid-control parameters
is asymptotically optimal, %(among an appropriate set of controls, depending on whether we solve the FCP or the RFCP),
we first establish general results for admissible policies.
Recall that, for a control $\pi$, $\tilde Q_\pi$ is the embedded process defined via \eqref{Q-DTMC}.
The proof of the following lemma follows from the proof of \cite[Proposition 1]{fricker1994monotonicity}, and is thus omitted.

\begin{lemma} \label{lmMarkovControl}
$\tilde Q_{\pi}$ is a homogeneous, aperiodic DTMC for any admissible control $\pi$.
\end{lemma}
We remark that the controls considered in \cite{fricker1994monotonicity} are assumed to satisfy a certain stochastic monotonicity property,
in addition to the conditions in our definition of admissible controls. Thus, the set of controls in this reference is smaller than ours.
However, that extra stochastic-monotonicity property does not determine the Markov property of the embedded process $\tilde Q_\pi$;
see the proof of \cite[Proposition 1]{fricker1994monotonicity}.

\begin{definition}
We say that a control $\pi$ is {\em stable} if $\tilde{Q}_\pi$ is absorbed in a positive recurrent class, regardless of its initial distribution.
\end{definition}

It follows from Lemma \ref{lmMarkovControl} that, for a stable control $\pi$,
\begin{equation*} \label{statDTMC}
Q_{\pi}(m) \Ra \tilde Q_{\pi}(\infty) \qasq m\tinf,
\end{equation*}
where $\tilde Q_{\pi}(\infty)$ is a random variable distributed according to a stationary distribution of the DTMC $\tilde Q_\pi$.
By flow-balance arguments, see, e.g., \cite{boon2011applications},
the length of a stationary server-cycle over an $L$-cycle augmented table $T_L$ has mean
\begin{equation*} \label{eq:ExpectedT}
\E\off{T_L} = L s/(1-\rho), \qforq L \ge 1.
\end{equation*}

Clearly, only stable controls are relevant for our (asymptotic) control-optimization problem.
However, we note that the stability region of a given control, namely, the set of values of the service and arrival rates for which the system is stable,
can be hard to characterize; see \cite{takagi1988queuing}.
The most general characterization of the stability region we are aware of was developed
in \cite{fricker1994monotonicity} (under the aforementioned stochastic-monotonicity property).

\subsection{Sequences of Admissible Controls}

%For our asymptotic-optimality analysis, we consider sequence of controls.
We say that a sequence of controls $\boldsymbol \pi = \{\pi^n : n \ge 1\}$ is admissible if $\pi^n$ is an admissible
policy for each $n \ge 1$, and denote the family of all such sequences by $\Pi$.
For $n \ge 1$ and $U^{(0), n} := 0$, let $U^{(m),n}$ denote the beginning of the $(m+1)$st server cycle of the $n$th system, $m \ge 0$.
Then, for $\boldsymbol \pi \in \Pi$,
\begin{equation*}
\tilde Q_{\pi^n}^n(m) := \bar Q_{\pi^n}^n(\bar U^{(m),n}) , \quad m \geq 0,
\end{equation*}
is a DTMC for all $n \ge 1$ by Lemma \ref{lmMarkovControl}.
If, in addition, the control is stable for each $n \ge 1$, then there exists a stationary distribution
for each of the DTMCs in the sequence, and we say that $\boldsymbol \pi$ is stable.

For the queue process in stationarity, the server-cycle length $T^n_L$ (when the control is designed for an $L$-cycle augmented table) has mean
$n s L/(1-\rho)$, and for $\bar T^n_L := T^n_L/n$,
\begin{equation} \label{eq:ExpeCycle}
\E\off{\bar T^n_L} = s L / (1-\rho),
\end{equation}
which is equal to the equilibrium cycle length $\tau_L$ in any $L$-cycle PE of the fluid model.

In order for a sequence of controls $\boldsymbol \pi \in \Pi$ to be asymptotically optimal, it must be stable and
the sequence of corresponding stationary distributions $\{\tilde Q_{\pi^n}^n  (\infty) : n \geq 1\}$ must be tight in $\RR_+^K$.
%We will use this latter fact to show that the binomial-exhaustive policy (with appropriately chosen parameters) is asymptotically optimal.
However, we remark at the outset that, even if $\boldsymbol \pi$ is stable and $\{\tilde Q_{\pi^n}^n  (\infty) : n \geq 1\}$ is tight,
there is no guarantee that there exists a global limit cycle for any of the resulting fluid limits, because the limits as $n\tinf$ and as $t\tinf$ need not commute.

\subsubsection{$L$-Cyclic Controls for the Restricted Problem}  \label{secPiL}

As is the case for the unrestricted problem, for a sequence of controls to be asymptotically optimal with respect
to the restricted optimal-control problem, that sequence must be stable, and the corresponding sequence of stationary distributions must be tight.
The difference between the two versions of the optimal-control problem is that,
in the restricted problem, we have fixed values of table cycles $L$ which we target.
%For example, assume that we want to find
%an asymptotically optimal control corresponding to the basic table, namely, for $L=1$. In this case, we search for an admissible control for the stochastic system
%that has a fixed set of switching rules which are repeated in each table cycle. (If $L=2$, say, then the switching rules will be repeated every two table cycles.)

Let $q_e$ denote a PE for a fluid limit when the sequence of controls is a stable sequence $\boldsymbol \pi$, and when
$\tilde Q^n_{\pi^n}(0) \deq \tilde Q^n_{\pi^n}(\infty)$. The fact that a fluid limit for such a sequence exists follows from Lemma \ref{lem:Tightness}
because the sequence of initial distribution is stationary, and is assumed to be tight, for the reason described above.
From the asymptotic perspective, there is clearly no point in considering $L$-cycle controls which give rise in the limit to PE that have a period that
does not divide $\tau_L$.
(We always allow the period of an $L$-cycle PE to be smaller than the cycle length.)
%Thus, when searching for an $L$-cycle asymptotically optimal control, we should only consider sequences of controls that give rise to $L$-cycle PE.
Thus, when solving the restricted problem over a set $\mathcal N \subset \NN$, we should only consider sequences of admissible controls
that give rise to $L$-cycle PE for $L \in \mathcal N$, which motivates considering the following family of controls.

\begin{definition} \label{def:LCyclic}
A sequence of admissible controls $\boldsymbol \pi \in \Pi$ is said to be $L$-cyclic if
any fluid limit of $\{\bar Q^n_{\pi^n} : n \ge 1\}$ with initial condition $\bar Q_{\pi^n}^n(0) \stackrel{d}{=} \tilde Q_{\pi^n}^n(\infty)$, $n \ge 1$,
is an $L$-cycle PE.
\end{definition}
\noindent We denote the subset of $L$-cyclic controls by $\Pi^L$.

\subsection{SB-PR and the Corresponding Binomial-Exhaustive Policy}\label{SecTranslatingPRC}

A FWLLN for the binomial-exhaustive policy follows easily from Proposition \ref{thm:FWLLN}, as the next corollary shows.

\begin{corollary}[FWLLN under binomial-exhaustive] \label{cor:BinAdmissible}
Let $\{Q^n : n \ge 1\}$ denote a sequence of queues where, for each $n \ge 1$, the system operates
under the binomial-exhaustive policy with the same parameters $(L, \mathbf r) \in \NN \times \mathcal R$.
If $\bar Q^n(0) \Arr q(0)$ in $\R_+^K$, then $\bar Q^n \Rightarrow q$ in $D^K$,
where $q$ is the fluid queue process under SB-PR with parameters $(L, \mathbf r)$ and initial condition $q(0)$.
\end{corollary}

\begin{proof}
We verify that the condition in Proposition \ref{thm:FWLLN} holds under SB-PR, namely,
$\bar D^{(m),n}_{i} \Arr d^{(m)}_{i}$ in $\RR_+$ as $n \arr \infty$, for all $m \geq 1$ and $i \in \mathcal{I}^L$.	
To this end, consider the first departure epoch, which is also the first busy time of the server at stage $1$ in the first server cycle, i.e.,
time $\bar D^{(1),n}_{1} = \bar B^{(1),n}_{1}$, for $n \ge 1$.
Under the binomial-exhaustive policy, all the arrivals to queue $k \in \mathcal K$ during the service time of a customer from that same queue are served as well,
and so the {\em total service time} of each served customer and all the arrivals during his service time is distributed like a busy period in an $M/G/1$
queue that has arrival rate $\lm_k$ and service rate $\mu_k$.

For each stage $i \in \mathcal I^L$ and the corresponding queue $p(i)$, denote by
$\Theta_{p(i)}^{(\ell)}$ the busy period ``generated'' by the service of the $\ell$th served customer in this queue.
Let $\{Y_i^{(\ell)} : \ell \ge 1\}$ be a sequence of i.i.d.\ Bernoulli r.v.'s with success probability $r_i$.
We use $Y_i$ and $\Theta_{p(i)}$ to denote corresponding generic random variables.
Then
\begin{equation*}
\bar B^{(1),n}_{1} = \frac{1}{n} \sum_{\ell =1}^{Q^n_{p(1)}(0)} \Theta_{p(1)}^{(\ell)} Y_{1}^{(\ell)}
\Arr q_{p(1)}(0) \E\off{\Theta_{p(1)} Y_1} \qasq n \tinf,
\end{equation*}
and due to the independence of $\Theta_{p(1)}$ and $Y_1$,
\begin{equation} \label{barB}
\bar B^{(1),n}_{1} = \bar D^{(1),n}_{1} \Arr q_{p(1)}(0)\E\off{\Theta_{p(1)} } \E\off{Y_1} = r_1 q_{p(1)}(0)/(\mu_{p(1)}-\lambda_{p(1)}) = d_1^{(1)},
\end{equation}
where the weak convergence holds as $n\tinf$.
Furthermore, the length of queue $p(1)$ at the end of the busy time is given by
\begin{equation*}
\begin{split}
\bar Q^n_{p(1)} (\bar D^{(1),n}_{1}) = \bar Q^n_{p(1)} \of{0} - \frac{1}{n} \sum_{\ell =1}^{Q^n_{p(1)}(0)}  Y_1^{(\ell)}
\Arr q_{p(1)}(0) - q_{p(1)}(0) r_1
= q_{p(1)} (d_1^{(1)})  \quad \text{as } n \arr \infty.
\end{split}
\end{equation*}
It follows from the FWLLN for the Poisson process and \eqref{barB} that, for all $k \neq p(1)$,
$$\bar Q^n_k (\bar D^{(1),n}_{1}) \Ra q_k (d_1^{(1)}) = q_k(0) + \lambda_k b_1^{(1)} \qasq n \tinf,$$
and that
\begin{equation*} \label{eq:2}
\bar Q (\bar D^{(1),n}_{1} + \bar V_1^{(1),n}) \Arr q (d_1^{(1)} + s_1)  \quad \text{as } n \arr \infty.
\end{equation*}
Continuing with the same line of arguments gives
$\bar D^{(m),n}_{i} \Arr d_i^{(m)}$ as $n \arr \infty$, for all $m \geq 1$, $i \in \mathcal{I}^L$, as required.
\end{proof}

The FWLLN under the binomial-exhaustive policy remains to hold if the condition that the initial queue converges
is replaced with the condition that the initial distribution of the queue
is equal to its stationary distribution at the beginning of a server cycle.
In this case, the resulting fluid limit is the global limit cycle (the unique PE) under the corresponding SB-PR control.
This result, stated formally in the following lemma, will be employed in the proofs of our main theorems.

\begin{lemma} [interchange of limits] \label{lem:interchange}
Let $\{Q^n : n \ge 1\}$ denote a sequence of queues where, for each $n \ge 1$, the system operates
under the binomial-exhaustive policy with the same parameters $(L, \mathbf r) \in \NN \times \mathcal R$.
Then for any real-valued, continuous, and bounded function $f$ on $\R^K_+$,
\begin{equation} \label{eqInterchange}
\lim_{m \arr \infty} \lim_{n \arr \infty} \E\off{f\of{\tilde Q^n(m)}} = \lim_{n \arr \infty} \lim_{m \arr \infty} \E\off{f\of{\tilde Q^n(m)}} = f\of{q_e(a_1)} ,
\end{equation}
where $q_e$ is the PE under SB-PR with parameters $(L, \mathbf r)$.
In particular, if $\bar Q^n(0) \deq \tilde Q^n(\infty)$ for all $n \ge 1$, then $\bar Q^n \Ra q_e$ in $D^K$ as $n\tinf$.
\end{lemma}

\begin{proof}
The key to the proof is the fact that
%\bes
$\E\off{\bar Q^n(\bar A_1^n)} = q_e(a_1)$ %\qforallq n \geq 1,
%\ees
for all $n \ge 1$. This fact, which is established in Lemma \ref{LEM:ASYMPQ2}, %This equality
%which is proved in Lemma \ref{LEM:ASYMPQ2} in Appendix \ref{ap:AsympMoment},
implies that
$\sup_n \E\off{\bar Q^n(\bar A_1^n)} < \infty.$
It follows from Markov's inequality that $\{\bar Q^n(\bar A_1^n): n \geq 1\}$ is UI, and thus tight in $\RR^K_+$.
Since $\bar Q^n(0) = \bar Q^n(\bar A_1^n)$ by definition, $\{\bar Q^n(0) : n \ge 1\}$ is tight.
Further, $\{\bar Q^n(U^{(m),n}) : m \ge 0\}$ is a stationary sequence, so that, since $A^n_1 = U^{(0), n}$, we have convergence
along subsequences $\bar Q^n(U^{(m),n_k}) \Ra \bar Q(0)$ as $k \tinf$, for all $m \ge 0$.
Note that, conditional on $\bar Q(0)$, the fluid limit $\bar Q$ is deterministic, and
converges to the global limit cycle $q_e$ as $t \tinf$, regardless of the realized value of $\bar Q(0)$.

Assume, in order to arrive at a contradiction, that there exists a set $E \subsetneq \RR^K$, with $q_e(a_1) \notin E$,
such that $\Prob{E} > 0$, where $\mathbb{P}$ denotes the probability distribution of $\bar Q(0)$.
Due to the convergence of $\bar Q(t)$ to $q_e$ as $t \tinf$,
there exists an $m_0$, such that $\|\bar Q(\bar U^{(m)}) - q_e(u^{(1)})\| < \ep$ w.p.1 for all $m \ge m_0$ and for any $\ep > 0$.
It follows that, for all $m$ large enough, $\bar Q(\bar U^{(m)}) \notin E$.
Since this holds for all the trajectories $\bar Q$ with $\bar Q(0) \in E$, it follows that $E$ is a set of transient states,
contradicting the stationarity of $\{\bar Q(\bar U^{(m)}) : m \ge 0\}$. Thus, $\Prob{E} = 0$, and in turn, $\bar Q(0) = q_e(a_1)$ w.p.1.
This latter equality holds for all converging subsequences of $\{\bar Q^n(0) : n \ge 1\}$, and so it holds for the sequence itself,
namely, $\bar Q^n(0) \Ra q_e(a_1)$ as $n\tinf$, implying \eqref{eqInterchange}.
This, together with the FWLLN in Corollary \ref{cor:BinAdmissible} when $\bar Q^n(0) \deq \bar Q^n(\infty)$, $n \ge 1$, implies that
$\bar Q^n \Ra q_e$ in $D^K$ as $n\tinf$.
\end{proof}

\section{Asymptotic Optimality of Binomial-Exhaustive} \label{secAsympOpt}
In this section we consider the global optimal-control problem, which is the subject of Theorem \ref{thmMain} and the corresponding FCP,
and the restricted optimal-control problem. %in which we seek an asymptotically optimal control among a set of admissible controls
%corresponding to a finite number of augmented tables.
%The UI condition that both the global and the restricted asymptotically optimal control problems is verified under the conditions in
%Theorem \ref{th:MomentCheck}, whose proof appears in Section \ref{ap:MomentCheck}, with supporting results in Appendices \ref{ap:Moments} and \ref{ap:AsympMoment}.

\subsection{Asymptotic Optimality for the Global Problem}

Theorems \ref{thm:AsympLB} and \ref{thm:AsympOpt} below imply Theorem \ref{thmMain}.
Recall that $c_*$ is the optimal objective value of the FCP.

\begin{theorem} [asymptotic lower bound] \label{thm:AsympLB}
$ \liminf\limits_{n \arr \infty} \, \liminf\limits_{t \arr \infty} \, \bar C^n_{\pi^n}(t) \geq c_* $ w.p.1, for any $\boldsymbol \pi \in \Pi$.
\end{theorem}

\begin{proof}
See Section \ref{secAsymtOpt}.
\end{proof}

Recall $\Psi^n_{\pi_*^n}$ in \eqref{Psi}, and that $\pi_*^n$ is the binomial-exhaustive policy with the same parameters $(L_*, \mathbf r_*)$ for all $n \ge 1$, where $(L_*, \mathbf r_*)$ are the optimal FCP parameters.
%Note also that $\bar C^n_{\pi_n^*}$, defined via \eqref{eqLongRunCost} for $\boldsymbol \pi_*$, is a constant (see \eqref{eq:11} for a characterization of that constant).

\begin{theorem} [asymptotic optimality] \label{thm:AsympOpt}
If $\{\bar \Psi^n_{\pi_*^n} : n \geq 1\}$ is UI, then
$$\lim_{n \arr \infty} \lim_{t \arr \infty} \bar C^n_{\pi_*^n}(t) = c_* \quad w.p.1.$$
\end{theorem}

\begin{proof}
See Section \ref{secProofThm4}.
\end{proof}

The following is an immediate corollary to Theorems \ref{thm:AsympLB} and \ref{thm:AsympOpt} (alternatively, to Theorem \ref{thmMain}), Corollary \ref{coro:AsyOpt},
Proposition \ref{LEM:LINEARCYCLIC}, and Theorem \ref{th:MomentCheck}.

\begin{corollary} \label{Cor:1}
Assume that Assumption \ref{assum2} holds and that the basic table is cyclic. Then the exhaustive policy is asymptotically optimal
under either of the following:
\begin{enumerate}[(i)]
\item For some $p \ge 1$, $\psi(x) = O(||x||^p)$ and is separable convex, and in addition, there exists an $\ep > 0$ such that
$\E\off{e^{t S_k}} < \infty$ for all $t \in (-\epsilon, \epsilon)$ and for all $k \in \mathcal K$.
\item $\psi(x) = O(||x||)$, and in addition, $\E\off{S_k^2} < \infty$ for all $k \in \mathcal{K}$.
\end{enumerate}
\end{corollary}

\subsection{Asymptotic Optimality for the Restricted Problem} \label{secAsyOptRestrict}
Recall that $L_\mathcal{N}$ is the number of table cycles contained in one server cycle of $q_\mathcal{N}$, and that
$\mathbf r_\mathcal{N}$ is the vector of proportion reductions at each stage in $q_\mathcal{N}$.
Let $\boldsymbol \pi_{\mathcal{N}} := \{\pi_\mathcal{N}^{n} : n \geq 1\}$ denote the sequence of binomial-exhaustive policies with parameters
$(L_\mathcal{N}, \mathbf r_\mathcal{N})$.
We then have the following asymptotic optimality result for the restricted class of admissible controls.
The proof of this result follows similar lines of arguments to those in the proofs of Theorems \ref{thm:AsympLB} and \ref{thm:AsympOpt},
and is therefore omitted.

\begin{theorem}[asymptotic optimality for the restricted problem] \label{th:AsympOptRelaxed}
For all $\boldsymbol \pi \in \bigcup\limits_{L \in \mathcal{N}} \Pi^L$ it holds that
$$\liminf_{n \arr \infty} \, \liminf_{t \arr \infty} \, \bar C^n_{\pi^n}(t) \geq c_\mathcal{N} \quad w.p.1,$$
for $c_\mathcal N$ in \eqref{FCPrelax}.
If, in addition, $\{\bar \Psi^n_{\pi_\mathcal{N}^{n}} : n \geq 1\}$ is UI, then
$$\lim_{n\tinf} \lim_{t\tinf} \bar C^n_{\pi_\mathcal{N}^{n}}(t) = c_\mathcal{N} \quad w.p.1.$$
\end{theorem}

Following the same lines of arguments as in Corollary \ref{coro:AsyOpt}, Theorem \ref{th:AsympOptRelaxed} implies that if $\{\bar \Psi^n_{\pi_\mathcal{N}^{n}} : n \geq 1\}$ is UI, then $\boldsymbol \pi_{\mathcal{N}}$ is asymptotically optimal among the restricted class of admissible controls. In particular, \eqref{asyOpt} holds.
This, together with Theorem \ref{th:MomentCheck}, implies the following corollary.

\begin{corollary} \label{Cor:3}
Suppose that $\psi(x) = O(||x||^p)$, for some $p \ge 1$, and that Assumption \ref{assum2} holds.
If, for some $\ep > 0$, $\E\off{e^{t S_k}} < \infty$ for all $t \in (-\epsilon, \epsilon)$ and for all $k \in \mathcal K$,
then the binomial-exhaustive policy with parameters $(L_\mathcal{N}, \mathbf r_\mathcal{N})$ is asymptotically optimal among
$\bigcup\limits_{L \in \mathcal{N}} \Pi^L$.
\end{corollary}

\subsection{Summary of Established Asymptotic Optimality Results}

We summarize the conditions and results of Corollaries \ref{Cor:1} and \ref{Cor:3} in Table \ref{table:1}.
\begin{table}[H]
\small
\centering
\begin{tabular}{| c !{\vrule width 1pt}  c | c | c | }
\hline
& Corollary \ref{Cor:1} (i) & Corollary \ref{Cor:1} (ii) & Corollary \ref{Cor:3}  \\  \hline
Admissible controls & $\Pi$ & $\Pi$ &  $\bigcup_{L \in \mathcal{N}} \Pi^L$ \\  \hline
Cost function &  \vtop{\hbox{\strut polynomial growth}\hbox{\strut and separable convex}} & linear growth & polynomial growth \\  \hline
Basic table & cyclic & cyclic  &  general \\  \hline
Service time distributions   &  finite m.g.f.'s  & second moments &  finite m.g.f.'s  \\  \hline
Optimal control & \vtop{\hbox{\strut $(L, \mathbf r) = (1, \mathbf 1)$}\hbox{\strut (exhaustive)}} &
\vtop{\hbox{\strut $(L, \mathbf r) = (1, \mathbf 1)$}\hbox{\strut (exhaustive)}} &
\vtop{\hbox{\strut binomial-exhaustive with}\hbox{\strut parameters $(L_\mathcal N, \mathbf r_\mathcal N)$}}   \\  \hline
\end{tabular}
\caption{Established asymptotic optimality results}
\label{table:1}
\end{table}

\section{Proofs of the Main Results} \label{secProofsofMainResults}
In this section we prove Theorems \ref{thm:AsympLB} and \ref{thm:AsympOpt}, from which Theorem \ref{thmMain} follows,
and Theorem \ref{th:MomentCheck}.
Some technical results which are employed in the proofs are proved in the appendix.

\subsection{Proof of Theorem \ref{thm:AsympLB}} \label{secAsymtOpt}
To establish Theorem \ref{thm:AsympLB}, it is sufficient restrict attention to
sequences of admissible controls $\boldsymbol \pi \in \Pi$ under which the corresponding sequences of embedded stationary DTMC's
$\{\tilde{Q}_{\pi^n}^n(\infty) : n \geq 1\}$ are tight; the set of such controls $\boldsymbol \pi$ is not empty due to Lemma \ref{lem:interchange}.
Take $\tilde Q_{\pi^n}^n(0) \deq \tilde{Q}_{\pi^n}^n(\infty)$ for each $n \geq 1$. Then $\{\bar Q_{\pi^n}^n(0) : n \geq 1\}$ is tight,
so that $\{\bar Q_{\pi^n}^{n} : n \geq 1\}$ is $C$-tight in $D^K$ by Lemma \ref{lem:Tightness}.

To decrease the notational burden, we fix a sequence of admissible controls $\boldsymbol \pi$ and a corresponding
converging subsequence of $\{\bar Q_{\pi^{n}}^{n} : n \geq 1\}$, but we remove the subscript $\pi^n$ from the notation, and denote the converging subsequence
by a superscript $\ell$.
For example, $\bar Q^\ell := \bar Q^{n_\ell}$ denotes the fluid-scaled queue process in system $n_\ell$, $\ell \ge 1$, operating
under the control $\pi^{n_\ell}$ in the converging subsequence of $\{\bar Q_{\pi^{n_\ell}}^{n_\ell} : \ell \ge 1\}$.

Let $\bar Q$ denote the limit of $\{\bar Q^\ell : \ell \ge 1\}$,
and let $ \alpha^\ell$ denote the stationary distribution of the corresponding embedded DTMC $\{\tilde Q^\ell(m): m \geq 0\}$.
Since each process in the pre-limit is stationary, the limit $\{\tilde Q(m): m \geq 0\}$ of this subsequence of DTMCs
is also stationary; we denote the corresponding stationary distribution by $\alpha$.
For $r \geq 0$, let $\mathscr B(r)$ denote a ball in $\R_+^K$ with positive $\af$-measure, namely, $\alpha(\mathscr B(r)) \in (0,1]$,
and let $B^o = (B^o_1, \dots B^o_K)$ denote the center of this ball.
Note that we do not rule out the case where $r = 0$, which is tantamount to $\sB(r)$ being a point in $\R_+^K$
and the limiting distribution $\af$ having a point mass on $B^o$.

Due to the weak convergence of $\{\tilde Q^\ell: \ell \geq 1\}$ to $\tilde Q$, we have
\begin{equation} \label{eq:22}
\lim_{\ell\tinf} \Prob{\tilde Q ^\ell(0) \in \mathscr B(r) } = \Prob{\tilde Q(0) \in \mathscr B(r) },
\end{equation}
so that
\begin{equation*} \label{eq:16}
\lim_{\ell\tinf} \af^\ell(\sB(r)) = \af(\sB(r))>0.
\end{equation*}
It follows from \eqref{eq:22} that $\af^\ell(\sB(r)) > 0$ for all $\ell$ large enough,
so that $\{\tilde Q ^\ell(m): m \geq 0\}$ must return to $\sB(r)$ infinitely often for any such $\ell$.
Similarly, there are infinitely many $m$'s for which $\tilde Q(m) \in \sB(r)$.
Let
\[
N^\ell_{r} :=\inf\{m\geq 1: \tilde Q^\ell(m)\in\sB(r)\} \quad  \mbox{and} \quad
N_r :=\inf\{m\geq 1: \tilde Q(m)\in \sB(r)\}.
\]
Then for
\bequ \label{alphas}
\alpha_{r}^\ell(\cdot) := \Prob{\tilde Q^\ell(0)\in \cdot \, | \, \tilde Q^\ell(0) \in \sB(r)} \mbox{ and ~}
\alpha_{r}(\cdot) :=\Prob{\tilde Q(0)\in \cdot \, | \, \tilde Q(0)\in \sB(r)},
\eeq
we have
\begin{equation} \label{limBallMeasure}
\lim_{\ell\tinf} \E_{\alpha_{r}^\ell}[ N^\ell_{r} ] = \lim_{\ell\tinf} \frac{1}{\alpha^\ell\of{\sB \of{r}}} = {1 \over \af(\sB(r))} = \E_{\alpha_{r}}\off{ N_r}.
\end{equation}
Define the following {\em first return times} to $\sB(r)$
\begin{equation} \label{eq:return}
\begin{split}
\bar R^\ell_{r} &:=\inf\{\bar U^{(m),\ell}>0: \bar Q^\ell(\bar U^{(m),\ell}) \in \sB(r)\} \\
\bar R_r &:=\inf\{\bar U^{(m)}>0: \bar Q(\bar U^{(m)}) \in \sB(r)\}.
\end{split}
\end{equation}
The next lemma is proved in Section \ref{SecProofOfRn}.

\begin{lemma} \label{lem:1}
The subsequence $\{\bar R^\ell_{r} : \ell \ge 1\}$ in \eqref{eq:return} is UI and satisfies $\bar R^\ell_{r} \Arr \bar R_{r}$. Hence,
\bes
\E_{\alpha_{r}^\ell}\off{\bar R^\ell_{r} } \ra \E_{\alpha_{r}}\off{\bar R_r} \qasq \ell\tinf.
\ees
\end{lemma}

Observe that the trajectory of $\bar Q$ over one return time (from time $0$ to $\bar R_r$) is ``nearly periodic" for small $r$,
in the sense that both $\bar Q(0)$ and $\bar Q(\bar R_r)$ are in $\sB(r)$, although the return time $\bar R_r$ may increase as $r$ decreases.

The next lemma, whose proof is given in Section \ref{SecProofOfRn} below, provides an upper bound on the value of $\bar R_r$,
and formalizes the observation that $\bar Q$ is ``nearly periodic,'' by proving that it can be made arbitrarily close to a PE-candidate.
To emphasize the fact that that PE-candidate depends on the realization of $\bar Q$, and therefore on the sample point $\omega \in \Omega$
(where $\Omega$ is the underlying sample space), we make explicit the dependence on $\omega$ by adding it to the notation when needed.
For example, we write $\bar Q(\omega, \cdot)$ for the sample path $\{\bar Q(t) : t \ge 0\}$ and $\bar R_r(\omega)$
for the realization of the random variable $\bar R_r$ corresponding to $\omega$.

\begin{lemma}  \label{lem:NearlyPeriodic}
There exist constants $d_1, d_2>0$ such that the following hold.

(i) $\left|\bar R_{r} - \tau_{N_{r}} \right| \leq d_1 r$ w.p.1.

(ii) There exists a set $E \subseteq \Omega$, with $P(E) = 1$, such that, for each $\omega \in E$,
there exists an $N_{r}(\omega)$-cycle PE-candidate $q^\omega$ for which
\begin{equation} \label{eq:60}
|| \bar Q(\omega, \cdot) - q^\omega||_{ \bar R_{r}(\omega) \vee \tau_{N_{r}(\omega)}  } \leq d_2 r.
\end{equation}
\end{lemma}

Consider the set $E$ in Lemma \ref{lem:NearlyPeriodic}, and fix $\omega \in E$. Assume that
$\bar R_{r}(\omega) \geq \tau_{N_{r}(\omega)}$; similar arguments to those below hold for the case $\bar R_{r}(\omega) < \tau_{N_{r}(\omega)}$.

Clearly, \eqref{eq:60} implies that
\begin{equation} \label{eq:62}
\max_{ k \in \mathcal{K}} |\bar Q_k(\omega, t) - q_k^\omega(t)| \leq d_2 r  \quad \text{for all } t \in [0, \bar R_{r}(\omega)].
\end{equation}
Now, $||\bar Q(\omega, 0) - B^o|| \leq r$ because $\bar Q(\omega, 0) \in \sB(r)$, so that
\begin{equation*}
\max_{ k \in \mathcal{K}} \bar Q_k(\omega,0) \leq \max_{ k \in \mathcal{K}} B^o_k + r.
\end{equation*}
Then \eqref{eq:60} implies that
$$\max_{ k \in \mathcal{K}} q_k^\omega(0) \leq \max_{ k \in \mathcal{K}} B^o_k + r + d_2 r,$$ and in turn, for all $t \in [0, \bar R_{r}(\omega)]$,
\begin{equation} \label{eq:63}
\begin{split}
\max_{ k \in \mathcal{K}} q_k^\omega(t) &\leq \max_{ k \in \mathcal{K}} q_k^\omega(0) + \max_{ k \in \mathcal{K}} \{\lambda_{k}(1-\rho_k)\} \tau_{N_{r}(\omega)} \\
&\leq \max_{ k \in \mathcal{K}} B^o_k  + r + d_2 r + \max_{ k \in \mathcal{K}} \{\lambda_{k}(1-\rho_k)\} \tau_{N_{r}(\omega)}.
\end{split}
\end{equation}
Together with \eqref{eq:62}, \eqref{eq:63} implies that for all $t \in [\tau_{N_{r}(\omega)}, \bar R_{r}(\omega)]$,
\begin{equation} \label{eq:65}
\begin{split}
\frac{1}{\tau_{N_{r}}} \max_{ k \in \mathcal{K}}  \bar Q_k(\omega, t)
&\leq \frac{1}{\tau_{N_{r}(\omega)}}  \of{\max_{ k \in \mathcal{K}} B^o_k + r + 2d_2 r } + \max_{ k \in \mathcal{K}} \{\lambda_{k}(1-\rho_k)\} \\
& \leq \frac{1}{\tau_{1}}  \of{\max_{ k \in \mathcal{K}} B^o_k + r + 2d_2 r } +  \max_{ k \in \mathcal{K}} \{\lambda_{k}(1-\rho_k)\} \\
&= d_3 r + d_4,
\end{split}
\end{equation}
where $d_3$ and $d_4$ are the following constants (that do not depend on $\omega$)
\begin{equation*}
d_3 := \frac{1}{\tau_{1}}  (1+2d_2) \quad \text{and} \quad d_4 := \frac{1}{\tau_{1}} \max_{ k \in \mathcal{K}} B^o_k + \max_{ k \in \mathcal{K}} \{\lambda_{k}(1-\rho_k)\} .
\end{equation*}
Thus, for any $\omega \in E$, it holds that
\begin{equation} \label{eq:64}
\begin{split}
&\max_{ k \in \mathcal{K}} \left| \frac{1}{\bar R_{r}(\omega) } \int_{0}^{\bar R_{r}(\omega)}  \bar Q_{k}(\omega, u) du - \frac{1}{\tau_{N_{r}(\omega)}}
\int_{0}^{\tau_{N_{r}(\omega)}}  q_{k}^\omega(u)  du \right| \\
&= \max_{ k \in \mathcal{K}}  \left| \frac{1}{\bar R_{r}(\omega) } \of{\int_{0}^{\tau_{N_{r}}(\omega)}  \bar Q_{k}(\omega, u) du
	+ \int_{\tau_{N_{r}(\omega)}}^{\bar R_{r}(\omega)}  \bar Q_{k}(\omega, u) du }
- \frac{1}{\tau_{N_{r}(\omega)}}  \int_{0}^{\tau_{N_{r}(\omega)}}  q_{k}^\omega(u)  du \right| \\
& \leq \max_{ k \in \mathcal{K}} \of{\frac{1}{\tau_{N_{r}(\omega)}}  \int_{0}^{\tau_{N_{r}(\omega)}} | \bar Q_{k}(\omega, u) - q_{k}^\omega( u) | du
	+  \frac{1}{\tau_{N_{r}(\omega)}}  \int_{\tau_{N_{r}(\omega)}}^{\bar R_{r}(\omega)}  \bar Q_{k}(\omega,u) du } \\
& \leq  d_2 r + \of{d_3 r + d_4} d_1 r,
\end{split}
\end{equation}
where the last inequality follows from \eqref{eq:62}, statement (i) in Lemma \ref{lem:NearlyPeriodic}, and \eqref{eq:65}.
It follows from \eqref{eq:64} that, for any $\ep > 0$, there exists $r_\ep > 0$ (that does not depend on $\omega$), such that for all $r < r_\ep$,
\begin{equation} \label{eq:Bd}
\max_{ k \in \mathcal{K}} \left| \frac{1}{\bar R_{r}(\omega) } \int_{0}^{\bar R_{r}(\omega)}  \bar Q_{k}(\omega, u) du - \frac{1}{\tau_{N_{r}(\omega)}}
\int_{0}^{\tau_{N_{r}(\omega)}} q_{k}^\omega(u)  du \right| < \ep.
\end{equation}

Since $q^\omega$ is bounded, \eqref{eq:60} implies that $\bar Q(\omega, \cdot)$ is bounded as well.
Therefore, due to the continuity of $\psi$, $q^\omega$, and of the sample path $\bar Q(\omega, \cdot)$,
the composite functions $\psi \circ q^\omega$ and $\psi \circ \bar Q(\omega, \cdot)$ are both uniformly continuous over compact time intervals.
It therefore follows from \eqref{eq:Bd} that for any $\delta > 0$, there exists an $\epsilon >0$ and a corresponding $r_\ep >0$, such that
\begin{equation*}
\left| \frac{1}{\bar R_{r}(\omega)} \int_{0}^{\bar R_{r}(\omega)}  \psi( \bar Q(\omega, u) ) du
- \frac{1}{\tau_{N_{r}(\omega)}} \int_{0}^{\tau_{N_{r}(\omega)}} \psi( q^\omega(u) ) du \right| < \delta, \qforallq r \in (0, r_\ep),
\end{equation*}
so that
\begin{equation} \label{eq:73}
\frac{1}{\bar R_{r}(\omega)} \int_{0}^{\bar R_{r}(\omega)}  \psi( \bar Q(\omega, u) ) du > c^\omega - \delta \geq c_* - \delta,
\end{equation}
where
\begin{equation*}
c^\omega := \frac{1}{\tau_{N_{r}(\omega)}} \int_{0}^{\tau_{N_{r}(\omega)}} \psi( q^\omega(u) ) du
\end{equation*}
is the time-average holding cost of $q^\omega$, and is necessarily no smaller than $c_*$ by the definition of the latter term.
Hence, due to the regenerative structure of $\bar Q^\ell$ for all $\ell \ge 1$, we have
(considering the random elements, and thus dropping $\omega$ from the notation)
\begin{align*} \label{eq:19}
\liminf_{\ell \arr \infty} \lim_{t \arr \infty} \bar C^\ell (t)
&=  \liminf_{{\ell} \arr \infty} \lim_{t \arr \infty} \frac{1}{t} \int_{0}^t \psi \of{\bar Q^\ell(u)}du \\
&= \liminf_{{\ell} \arr \infty} \frac{\E_{\alpha_{r}^\ell} \off{\int_{0}^{\bar R^\ell_{r} } \psi\of{\bar Q^\ell(u)} du}}{\E_{\alpha_{r}^\ell} \off{\bar R^\ell_{r} }} \quad w.p.1 \mbox{ by renewal-reward theorem} \\
&\geq  \frac{\liminf\limits_{\ell \arr \infty} \E_{\alpha_{r}^\ell} \off{\int_{0}^{\bar R^\ell_{r} } \psi\of{\bar Q^\ell(u)} du}}{\limsup\limits_{\ell \arr \infty}\E_{\alpha_{r}^\ell} \off{\bar R^\ell_{r} }} \\
&=  \frac{\liminf\limits_{\ell \arr \infty} \E_{\alpha_{r}^\ell} \off{\int_{0}^{\bar R^\ell_{r} } \psi\of{\bar Q^\ell(u)} du}}{\E_{\alpha_{r}} \off{\bar R_r }} \quad \mbox{by Lemma \ref{lem:1}}\\
&\geq  \frac{ \E_{\alpha_{r}} \off{\liminf\limits_{\ell \arr \infty} \int_{0}^{\bar R^\ell_{r^\ell} } \psi\of{\bar Q^\ell(u)} du}}{\E_{\alpha_{r}} \off{\bar R_r }} \quad \mbox{by Fatou's lemma}\\
&=  \frac{ \E_{\alpha_{r}} \off{ \of{\frac{1}{\bar R_r} \int_{0}^{\bar R_r } \psi\of{\bar Q(u)} du} \bar R_r}}{\E_{\alpha_{r}} \off{\bar R_r}} \\
&>  \frac{ \E_{\alpha_{r}} \off{ (c_* - \delta) \bar R_r}}{\E_{\alpha_{r}} \off{\bar R_r }} \quad \text{on the event $E$ by } \eqref{eq:73} \\
&= c_* - \delta.
\end{align*}
Note that the second equality above holds regardless of whether $\E_{\alpha_{r}^\ell} \off{\int_{0}^{\bar R^\ell_{r} } \psi\of{\bar Q^\ell(u)} du} < \infty$
because $\psi$ is nonnegative; see, e.g., Theorem 2.2.1 and the corresponding remark on p.42 in \cite{tijms2003first}.
The result follows because $\delta$ is arbitrary \hfill \qed

\subsubsection{Proofs of the Auxiliary Results in the Proof of Theorem \ref{thm:AsympLB}} \label{SecProofOfRn}

\begin{proof} [Proof of Lemma \ref{lem:1}]
The weak convergence in the statement follows from the continuous mapping theorem applied to the first passage time \citep[Theorem 13.6.4]{whitt2002stochastic}.
To prove the convergence of the means,
let $\bar Q^{\ell}(0)$ be distributed according to $\alpha^{\ell}$,
and $\bar Q(0)$ be distributed according to $\alpha$, for $\alpha^\ell$ and $\alpha$ in \eqref{alphas}.

The length of the return time $\bar R^{\ell}_{r} $ consists of the total time the server spends serving each queue $k$, $k \in \mathcal{K}$, plus the total switchover time in $ N^{\ell}_{r}$ table cycles. Let $G_k^{\ell}$ denote the number of customers served at queue $k$ over the time interval $[0, \bar R^{\ell}_{r} ] $, and $\bar G_k^{\ell} := G_k^{\ell}/\ell$. It holds that
\begin{equation} \label{eq:82}
\bar Q^{\ell}_k(0) + \bar P_k^{\ell} \of{ \frac{1}{\ell} \sum_{\nu=1}^K \sum_{j=1}^{ G_\nu^{\ell}} S_\nu^{(j)} + \frac{1}{\ell}\sum_{\nu = 1}^{N^{\ell}_{r}}  \sum_{i=1}^{I} V_i^{(\nu),\ell}} - \bar G_k^{\ell} = \bar Q^{\ell}_k(\bar R^{\ell}_{r}) , \quad k \in \mathcal{K} ,
\end{equation}
where
$$\bar R^{\ell}_{r} = \frac{1}{\ell} \sum_{\nu=1}^K \sum_{j=1}^{ G_\nu^{\ell}} S_\nu^{(j)} + \frac{1}{\ell}\sum_{\nu = 1}^{N^{\ell}_{r}}  \sum_{i=1}^{I} V_i^{(\nu),\ell}.$$
Since $\bar Q$ is stationary, both $\bar Q^\ell (0)$ and $\bar Q^\ell(\bar R^\ell)$ are distributed according to $\af_r^\ell$, so that
$$\E_{\alpha_{r}^\ell} \off{\bar Q^{\ell}_k(0)} = \E_{\alpha_{r}^\ell} \off{\bar Q^{\ell}_k(\bar R^{\ell}_{r})}, \quad k \in \mathcal{K}.$$
Thus, taking expectations in \eqref{eq:82} and applying Wald's equation give
\begin{equation*} \label{eq:42}
\begin{split}
\E_{\alpha_{r}^\ell} \off{\bar G_k^{\ell} }
=\lambda_k \of{ \sum_{\nu=1}^K \frac{1}{\mu_\nu} \E_{\alpha_{r}^\ell} \off{\bar G_\nu^{\ell}} + \E_{\alpha_{r}^\ell} \off{N^{\ell}_{r} } s } , \quad k \in \mathcal{K}.
\end{split}
\end{equation*}
It follows that $\E_{\alpha_{r}^\ell} \off{\bar G_k^{\ell} }  =  \frac{\lambda_k s }{1-\rho} \E_{\alpha_{r}^\ell} \off{N^{\ell}_{r}}$, so that
\begin{equation} \label{eq:barRnk}
\E_{\alpha_{r}^\ell} \off{\bar R_r^{\ell} }  = \frac{s }{1-\rho} \E_{\alpha_{r}^\ell} \off{N^{\ell}_{r}}  , \quad k \in \mathcal{K}.
\end{equation}

Similar flow equation holds for the subsequential limit process $\bar Q$. Since the sample paths of $\bar Q$ are of the form \eqref{fluid implicit} by Lemma \ref{lem:Tightness}, the process $\bar Q$ satisfies
\begin{equation*}
\bar{Q}_k(t) = \bar{Q}_k(0) + \lm_k t - \mu_k \bar{\mathcal B}_k(t), \quad t \ge 0, \quad  k \in \mathcal{K},
\end{equation*}
where $\bar{\mathcal B}_k := \{\bar{\mathcal B}_k(t) : t \ge 0\}$ is of the form
\begin{equation} \label{eq:84}
\bar{\mathcal B}_k(t) = \int_{0}^{t} b_k(u)du ,
\end{equation}
for a piecewise-constant function $b_k : \RR_+ \ra \{0,1\}$.
Then, by definition of $\bar R_r$, we have
\begin{equation} \label{fluidLimOfTight}
\bar{Q}_k(\bar R_r) = \bar{Q}_k(0) + \lm_k \bar R_r - \mu_k \bar{\mathcal B}_k(\bar R_r), \quad  k \in \mathcal{K},
\end{equation}
where $ \bar R_r =  \sum_{k=1}^K \bar{\mathcal B}_k(\bar R_r) + \bar N_r s$.
As both $\bar{Q}_k(0)$ and $\bar{Q}_k(\bar R_r)$ are distributed according to $\alpha$,
it holds that $\E_{\alpha_{r}}\off{\bar{Q}_k(0)} = \E_{\alpha_{r}} \off{\bar{Q}_k(\bar R_r)}$, and therefore
\begin{equation*}
\lambda_k \E_{\alpha_{r}} \off{\sum_{\nu=1}^K \bar{\mathcal B}_\nu(\bar R_r) +  N_r s } = \mu_k \E_{\alpha_{r}} \off{\bar{\mathcal B}_k(\bar R_r)}, \quad k \in \mathcal{K}.
\end{equation*}
In turn, $\E_{\alpha_{r}} \off{\bar{\mathcal B}_k(\bar R_r)} = \frac{\rho_k s }{1-\rho} \E_{\alpha_{r}} \off{N_r}$, so that
\begin{equation} \label{eq:Rr}
\E_{\alpha_{r}} \off{\bar R_r} = \frac{s }{1-\rho} \E_{\alpha_{r}} \off{N_r}  , \quad k \in \mathcal{K} .
\end{equation}
Since $\E_{\alpha_{r}^\ell}\off{N^{\ell}_{r}} \arr \E_{\alpha_{r}}\off{ N_r}$ as $\ell \arr \infty$ by \eqref{limBallMeasure},
it follows from \eqref{eq:barRnk} and \eqref{eq:Rr} that $\E_{\alpha_{r}^\ell} \off{\bar R^{\ell}_{r} } \arr \E_{\alpha_{r}} \off{\bar R_r}$ as $\ell \arr \infty$,
and the result follows.
Finally, since $\bar R_r \ge 0$ and $\bar R^\ell_r \ge 0$ for all $\ell \ge 1$ w.p.1, the sequence $\{\bar R^\ell_r : \ell \ge 1\}$ is UI by Theorem 5.4 in \cite{billingsley1968convergence}.
\end{proof}

%\bigskip

\begin{proof} [Proof of Lemma \ref{lem:NearlyPeriodic}]
We prove the two assertions of the lemma separately.
\paragraph*{Proof of $\boldsymbol{(i)}$}
By \eqref{fluidLimOfTight} and the fact that $\|\bar Q(\bar R_{r}) - \bar Q(0)\| \leq 2r$, it holds that for each $k \in \mathcal{K}$, 	
\begin{equation}  \label{eq:119}
\begin{gathered}
-2r \,\, \leq \,\, -\mu_k \bar{\mathcal B}_k(\bar R_{r}) + \lambda_k \bar R_{r} \,\, \leq \,\, 2r,\\
\end{gathered}
\end{equation}
so that
\begin{equation}  \label{eq:66}
\sum_{k \in \mathcal{K}} (-2r/\mu_k +  \rho_k \bar R_{r}) \,\, \leq \,\, \sum_{k \in \mathcal{K}} \bar{\mathcal B}_k(\bar R_{r}) \,\, \leq \,\,
\sum_{k \in \mathcal{K}} (2r/\mu_k +  \rho_k \bar R_{r}) \quad w.p.1.
\end{equation}
Since $\bar R_{r}$ is the total length of the $N_{r}$ table cycles, it equals the total time the server spends switching,
which is equal to $s N_{r}$, and the total time it spends serving in each of the queues. Hence,
\begin{equation*}
\bar R_{r} = s N_{r} + \sum_{k \in \mathcal{K}} \bar{\mathcal B}_k (\bar R_{r}) .
\end{equation*}
It then follows from \eqref{eq:66} that
\begin{equation*}
\begin{split}
s N_{r} + \sum_{k \in \mathcal{K}}  \of{-2r/\mu_k +  \rho_k \bar R_{r} } \,\, & \leq
\,\, \bar R_{r}  \,\, \leq \,\, s N_{r} + \sum_{k \in \mathcal{K}} \of{2r/\mu_k +  \rho_k \bar R_{r} } \quad \text{so that }\\
- 2r \sum_{k \in \mathcal{K}} \frac{1}{\mu_k} \,\, & \leq \,\, (1-\rho) \bar R_{r} - s N_{r}  \,\, \leq \,\, 2r \sum_{k \in \mathcal{K}} \frac{1}{\mu_k},
\end{split}
\end{equation*}
and employing \eqref{eq:57} gives
\bequ  \label{eq:67}
- \frac{2r}{1-\rho} \sum_{k \in \mathcal{K}} \frac{1}{\mu_k} \,\, \leq
\,\,  \bar R_{r} - \tau_{N_{r}}  \,\, \leq  \,\, \frac{2r}{1-\rho} \sum_{k \in \mathcal{K}} \frac{1}{\mu_k} .
\eeq
Taking $d_1 := 2(1-\rho)^{-1} \sum_{k \in \mathcal{K}} 1/\mu_k $ proves the first part of the lemma.

\paragraph*{Proof of $\boldsymbol{(ii)}$}
We show that \eqref{eq:60} holds w.p.1, so that the event $E$ in the statement exists. To this end, we fix $\omega \in \Omega$,
and prove the result by constructing a $N_r(\omega)$-cycle PE-candidate $q^\omega$ such that \eqref{eq:60} holds for the sample path $\bar Q(\omega, \cdot)$.
To simplify the notation, the values of all the random elements (variables and processes) below are assumed to be realizations corresponding to
that fixed $\omega$, although we remove it from the notation (except for the PE-candidate $q^\omega$ we construct).

For the limiting process $\bar Q$,
let	$\bar B^{(m)}_{k_j}$, $j = 1,..., dim(\mathbf k)$, $m = 1,...,N_r$, denote the busy time spent serving queue $k$ at stage $ k_j$
in the $m$th server cycle. By definition,
\begin{equation*}
\bar{\mathcal B}_{k}(\bar R_{r}) = \sum_{m =1}^{N_{r}} \sum_{j=1}^{dim(\mathbf k)} \bar B^{(m)}_{ k_j} , \quad k \in \mathcal{K} ,
\end{equation*}
for $\bar{\mathcal B}_{k}$ in \eqref{eq:84}.
It follows from \eqref{eq:119} and \eqref{eq:67} that for $k \in \mathcal{K}$,
\begin{equation*}
- \frac{2r}{\mu_k} +  \rho_k \of{\tau_{N_{r}}   - \frac{2r}{1-\rho} \sum_{k \in \mathcal{K}} \frac{1}{\mu_k}}  \, \leq \, \bar{\mathcal B}_{k}(\bar R_{r}) \, \leq
\, \frac{2r}{\mu_k} +  \rho_k \of{\tau_{N_{r}}  + \frac{2r}{1-\rho} \sum_{k \in \mathcal{K}} \frac{1}{\mu_k}} ,
\end{equation*}
so that, for $\delta_{k} := \bar{\mathcal B}_{k}(\bar R_{r}) - \rho_k \tau_{N_{r}} $, it holds that
\begin{equation} \label{eq:85}
|\delta_{k}| \leq 2r \of{\frac{1}{\mu_k} + \frac{\rho_k}{1-\rho} \sum_{k \in \mathcal{K}} \frac{1}{\mu_k}}, \quad k \in \mathcal K.
\end{equation}

The proof proceeds by explicitly constructing $q^\omega$.
To this end, we first characterize a $N_r$-cycle closed-curve in $\RR^K$, denoted by $\bar Q'$,
whose trajectory is sufficiently close to the sample path of $\bar Q$ (corresponding to the sample point $\omega$).
However, that closed curve $\bar Q'$ is not necessarily a PE-candidate, as its components may achieve negative values.
We then show that only a small perturbation of the trajectory $\bar Q'$, such that the perturbed trajectory remains sufficiently close to $\bar Q$,
produces a bona-fide $N_r$-cycle PE-candidate $q^\omega$.

To construct $\bar Q'$, we first take $\bar Q'(0) := \bar Q(0)$, and then specify the busy times of $\bar Q'$, such that $\bar Q'(\tau_{N_r}) = \bar Q'(0)$.
(We treat $\bar Q'$ as a queue process, similarly to our treatment of the fluid models. Thus, by ``busy times'' of $\bar Q'_k$
we mean the times at which the $k$th component of $\bar Q'$ is decreasing.)
This can be easily achieved by solving {\em flow balance equations} which equate the ``inflow'' to $\bar Q'$ over the $N_r$ table cycles,
which occurs at a constant rate $\lm_k$ throughout, to the ``outflow'' over the $N_r$ table cycles, which occurs at constant rate
$- \mu_k$ only during the busy times.
Let $(\bar B_{ k_j}^{'(m)}, j = 1,..., dim(\mathbf k), m = 1,...,N_{r})$ denote those busy times of $\bar Q'_k$.

(1) For queue $k$ with $\delta_{k} < 0$, we take
$\bar B_{ k_1}^{'(1)} := \bar B_{ k_1}^{(1)} + |\delta_{k}|$;
and $\bar B_{k_j}^{'(m)} := \bar B_{ k_j}^{(m)}$, for $j = 1,..., dim(\mathbf k)$, $m = 1,..., N_{r}$, and $j + m > 2$.
Thus, except for its first busy time, all other busy times of $\bar Q'_k$ are equal to those of $\bar Q_k$.

(2) For queue $k$ with $\delta_{k} > 0$, we take $\bar B_{ k_{\hat j}}^{'(\hat{m})} := \bar B_{ k_{\hat j}}^{(\hat{m})} - \delta_{k}$
for some $\hat j \in \{1,..., dim(\mathbf k)\}$ and $\hat m \in \{1,...,N_{r}\}$ with $\bar B_{ k_{\hat j}}^{(\hat{m})} \geq \delta_{k}$.
(Such  ${\hat j}$ and $\hat{m}$ exist for sufficiently small $r$ due to \eqref{eq:66} and \eqref{eq:67}.)
and as $\bar B_{ k_j}^{'(m)} := \bar B_{ k_j}^{(m)}$ for $j = 1,..., dim(\mathbf k)$, $m = 1,...,N_{r}$, $j \neq \hat j$, and $m \neq \hat m$.
Thus, $\bar Q'_k$ and $\bar Q_k$ have the same busy times, except for one busy time, which is shorter for $\bar Q'_k$ by $\delta_k$.

(3) For queue $k$ with $\delta_{k} = 0$,  we take $\bar B_{ k_j}^{'(m)} := \bar B_{k_j}^{(m)}$ for all $j = 1,..., dim(\mathbf k)$ and $m = 1,...,N_{r}$.
In particular, $\bar Q'_k$ and $\bar Q_k$ have the same busy times.

Observe that the  busy times of $\bar Q'$ satisfy the flow balance at all queues, i.e.,
\begin{equation*}
\sum_{m =1}^{N_{r}} \sum_{j=1}^{dim(\mathbf k)} \bar B^{'(m)}_{ k_j} = \rho_k \tau_{N_{r}} \quad k \in \mathcal{K} .
\end{equation*}
so that $\bar Q'( \tau_{N_{r}}) = \bar Q'(0)$.
In addition, we will show that
\begin{equation} \label{eq:69}
||\bar Q - \bar Q'||_{ \bar R_{r} \vee \tau_{N_{r}}  } \leq  \of{\frac{2rK}{1-\rho} \sum_{k \in \mathcal{K}} \frac{1}{\mu_k} } \max_{k \in \mathcal{K}} \offf{\mu_k} .
\end{equation}
However, before proving \eqref{eq:69} we show that we can use this inequality to construct a PE-candidate $q^\omega$ as in the statement of the lemma.
To this end, let
\begin{equation} \label{eq:68}
q^\omega := \bar Q' + \Delta \qforq \Delta := \of{\frac{2rK}{1-\rho} \sum_{k \in \mathcal{K}} \frac{1}{\mu_k} } \max_{k \in \mathcal{K}} \offf{\mu_k}.
\end{equation}
Since $\bar Q \geq 0$, it follows from \eqref{eq:69} that $q^\omega \geq 0$, and since $\bar Q'$ is a closed curve in $\RR^K$, so is $q^\omega$.
Finally, $q^\omega$ clearly satisfies the fluid model equations \eqref{fluid implicit}, and is therefore a bona-fide PE-candidate.
Combining \eqref{eq:69} and \eqref{eq:68} gives
\begin{equation*}
|| \bar Q - q^\omega ||_{\bar R_{r} \vee \tau_{N_{r}}  } \leq \of{\frac{4rK}{1-\rho} \sum_{k \in \mathcal{K}} \frac{1}{\mu_k} } \max_{k \in \mathcal{K}} \offf{\mu_k},
\end{equation*}
so that \eqref{eq:60} follows by setting $d_2 := \of{\frac{4K}{1-\rho} \sum_{k \in \mathcal{K}} \frac{1}{\mu_k} } \max_{k \in \mathcal{K}} \offf{\mu_k}$.

To finish the proof of the lemma, it remains to justify \eqref{eq:69}.
To this end, note that $\bar Q'$ and $\bar Q$ follow identical trajectories from initialization until some busy time differs,
namely, when $\bar B^{'(m)}_{ \ell_j} \neq \bar B^{(m)}_{ \ell_j}$,
for some $\ell \in \mathcal{K}$, $j \in \{1,..., dim(\mathbf k)\}$, and $m \in \{1,...,N_{r}\}$.
By construction, $|\bar B^{'(m)}_{ \ell_j} - \bar B^{(m)}_{ \ell_j}| = |\delta_{\ell}|$.
Since queue $k$ in either process decreases at rate $\mu_k - \lambda_k$ during the busy times, and increases at rate $\lambda_k$ everywhere else,
it holds that
\begin{equation} \label{eq:3}
|| \bar Q_{k} - \bar Q_{k}' ||_{\bar D^{(m)}_{ \ell_j} \vee \bar D^{'(m)}_{ \ell_j} }
= |\delta_{\ell}| \mu_k , \quad k \in \mathcal{K},
\end{equation}
where $\bar D^{'(m)}_{ \ell_j}$ (alternatively, $\bar D^{(m)}_{ \ell_j}$)
is the departure epoch immediate after the busy time $\bar B^{'(m)}_{ \ell_j}$
(alternatively, $\bar B^{(m)}_{\ell_j}$) in $\bar Q'$ (alternatively, $\bar Q$).
Then \eqref{eq:3} implies that
\begin{equation} \label{eq:4}
|| \bar Q - \bar Q' ||_{\bar D^{(m)}_{ \ell_j} \vee \bar D^{'(m)}_{ \ell_j} }  \leq   |\delta_{\ell}| \max_{k \in \mathcal{K}} \offf{\mu_k} .
\end{equation}
After that departure epoch (i.e., $\bar D^{'(m)}_{ \ell_j}$ for $\bar Q'$, and $\bar D^{(m)}_{ \ell_j}$ for $\bar Q$),
the trajectories of $\bar Q'$ and $\bar Q$ increase and decrease at the same rate over the same time intervals, until another busy time differs, i.e.,
$\bar B^{'(\hat{m})}_{\hat{\ell}_{\hat{j}}} \neq \bar B^{(\hat{m})}_{\hat{\ell}_{\hat{j}}}$,
for some $\hat{\ell} \in \mathcal{K}$ ($\hat{\ell} \neq \ell$), $\hat{j} \in \{1,..., dim(\boldsymbol{\hat{\ell}})\}$, and $\hat{m} \in \{1,...,N_{r}\}$.
Following similar arguments as above, the second difference in the busy times can {\em further} enlarge the distance between $\bar Q'$ and $\bar Q$
(from time zero to the departure epoch after the busy time in consideration) component wise by a maximum of
$(|\delta_{\ell}| + |\delta_{\hat{\ell}}|) \max_{k \in \mathcal{K}} \offf{\mu_k}$.
In particular, define
\begin{equation} \label{eq:81}
\tilde Q' := \bar Q' - \of{\bar Q'(\bar D^{(m)}_{ \ell_j} \vee \bar D^{'(m)}_{\ell_j}) - \bar Q(\bar D^{(m)}_{ \ell_j} \vee \bar D^{'(m)}_{ \ell_j}) } .
\end{equation}
(Note that $\tilde Q'$ is the trajectory ``shifted" from $\bar Q'$, so that $\tilde Q'(\bar D^{(m)}_{ \ell_j} \vee \bar D^{'(m)}_{ \ell_j}) = \bar Q(\bar D^{(m)}_{ \ell_j} \vee \bar D^{'(m)}_{ \ell_j})$.)
It holds that
\begin{equation} \label{eq:83}
\sup_{t \, \in \, [ \bar D^{(m)}_{ \ell_j} \vee \bar D^{'(m)}_{ \ell_j} , \, \bar D^{(\hat{m})}_{\hat{\ell}_{\hat{j}}} \vee \bar D^{'(\hat{m})}_{\hat{\ell}_{\hat{j}}} ]} \, \max_{ k \in \mathcal{K}} \, | \bar Q_k(t) - \tilde Q'_k(t) |  \leq    (|\delta_{\ell}| + |\delta_{\hat{\ell}}|) \max_{k \in \mathcal{K}} \offf{\mu_k} .
\end{equation}
This fact in \eqref{eq:83}, together with \eqref{eq:4} and \eqref{eq:81}, gives
\begin{equation} \label{eq:88}
\sup_{t \, \in \, [ \bar D^{(m)}_{ \ell_j} \vee \bar D^{'(m)}_{ \ell_j} , \, \bar D^{(\hat{m})}_{\hat{\ell}_{\hat{j}}} \vee \bar D^{'(\hat{m})}_{\hat{\ell}_{\hat{j}}} ]} \, \max_{ k \in \mathcal{K}} \, | \bar Q_k(t) - \bar Q'_k(t) |  \leq    (2|\delta_{\ell}| + |\delta_{\hat{\ell}}|) \max_{k \in \mathcal{K}} \offf{\mu_k} .
\end{equation}
Since the right-hand side of \eqref{eq:88} is larger than that of \eqref{eq:4}, we have that
\begin{equation*}
|| \bar Q - \bar Q' ||_{\bar D^{(\hat{m})}_{\hat{\ell}_{\hat{j}}} \vee \bar D^{'(\hat{m})}_{\hat{\ell}_{\hat{j}}}  }
\leq  \of{2|\delta_{\ell}| + |\delta_{\hat{\ell}}|} \max_{k \in \mathcal{K}} \offf{\mu_k} .
\end{equation*}
The same arguments continue to the end of the $N_{r}$th server cycle. Therefore,
\begin{equation*}
\begin{split}
|| \bar Q - \bar Q' ||_{\bar R_{r} \vee \tau_{N_{r}}  }  &\leq \sum_{k \in \mathcal{K}} K |\delta_{k}| \max_{k \in \mathcal{K}} \offf{\mu_k} \\
&\leq K \sum_{k \in \mathcal{K}} 2r \of{\frac{1}{\mu_k} + \frac{\rho_k}{1-\rho} \sum_{k \in \mathcal{K}} \frac{1}{\mu_k}} \max_{k \in \mathcal{K}} \offf{\mu_k}
\quad \text{by } \eqref{eq:85} \\
&= \of{\frac{2rK}{1-\rho} \sum_{k \in \mathcal{K}} \frac{1}{\mu_k} } \max_{k \in \mathcal{K}} \offf{\mu_k}. \qedhere
\end{split}
\end{equation*}
\end{proof}

\subsection{Proof of Theorem \ref{thm:AsympOpt}} \label{secProofThm4}
Consider a sequence of systems operating under the sequence of binomial-exhaustive policies $\boldsymbol{\pi}_*$, each with parameters $(L_*, \mathbf r_*)$.
For each $n \geq 1$, let $\bar M^n(t) := \max\{m \geq 1 : \bar U^{(m),n} \leq t \}$. Then
%\begin{equation} \label{eq:11}
%\begin{split}
\begin{align}\label{eq:11}
\lim_{t\tinf} \bar C_{\pi_*^n}^n(t) &= \lim_{t\tinf} \frac{1}{t} \int_0^t \psi(\bar Q_{\pi_*^n}^n(u)) du \nonumber \\
& = \lim_{t\tinf} \frac{\sum_{m = 1}^{\bar M^n(t)} \int_{\bar U^{(m-1),n}}^{\bar U^{(m),n}} \psi(\bar Q_{\pi_*^n}^n(u)) du
	+ \int_{\bar U^{(\bar M^n(t)),n}}^{t}\psi(\bar Q_{\pi_*^n}^n(u)) du }{ \sum_{m = 1}^{\bar M^n(t)} \bar T^{(m),n} +  (t - \bar U^{(\bar M^n(t)),n}) } \nonumber\\
& = \lim_{t\tinf} \frac{ \frac{1}{\bar M^n(t)} \sum_{m = 1}^{\bar M^n(t)} \int_{\bar U^{(m-1),n}}^{\bar U^{(m),n}} \psi(\bar Q_{\pi_*^n}^n(u)) du
	+ \frac{1}{\bar M^n(t)}  \int_{\bar U^{(\bar M^n(t)),n}}^{t}\psi(\bar Q_{\pi_*^n}^n(u)) du }{ \frac{1}{\bar M^n(t)}  \sum_{m = 1}^{\bar M^n(t)} \bar T^{(m),n}
	+ \frac{1}{\bar M^n(t)}  (t - \bar U^{(\bar M^n(t)),n}) } \nonumber\\
& = \frac{\E_{\alpha^n}\off{\bar \Psi^n_{\pi^n_*}}}{\E_{\alpha^n}\off{\bar T_{L_*}^{n}}} \quad w.p.1,
%\end{split}
%\end{equation}
\end{align}
for $\bar \Psi^n_{\pi^n_*}$ in \eqref{Psi}, where the last equality follows because the embedded DTMC converges to its unique stationary distribution $\alpha^n$,
and because the second terms on both the numerator and denominator converge to $0$ w.p.1.
Indeed, $\bar Q^n_{\pi_*^n}(\infty)$ is bounded and $0 \le t - \bar U^{(\bar M^n(t)), n} \le \bar T_{L_*}^{(\bar M^n(t)+1), n} < \infty$ w.p.1
by virtue of \eqref{eq:ExpeCycle}.

Lemma \ref{lem:interchange}, the continuity of $\psi$ and the continuous mapping theorem imply that
\bes
\bar \Psi^n_{\pi^n_*} \Ra \int_0^{\tau_*}\psi(q_*(u)) du \quad \text{in $\RR$ as } n \tinf.
\ees
Thus, since $\E_{\alpha^n}\off{\bar T_{L_*}^{n}} = \tau_*$ for all $n\ge 1$, due to \eqref{eq:57} and \eqref{eq:ExpeCycle},
the assumed UI of $\{\bar \Psi^n_{\pi_*^n} : n \geq 1\}$ and \eqref{eq:11} give
\begin{equation*}
\begin{split}
\lim_{n \arr \infty} \lim_{t \arr \infty} \bar C_{\pi_*^n}^n(t)
= \lim_{n \arr \infty} \frac{1}{\tau_{*}} \E_{\alpha^n}\off{\bar \Psi^n_{\pi_*^n}}
= \frac{1}{\tau_{*}} \int_{0}^{\tau_{*}} \psi\of{q_{*}(u)} du = c_* \quad w.p.1.  \hfill \qed
\end{split}
\end{equation*}

\subsection{Proof of Theorem \ref{th:MomentCheck}} \label{ap:MomentCheck}
Throughout this section, we consider a sequence of systems operating under the binomial-exhaustive policy with control parameters $(L, \mathbf r) \in \NN \times \mathcal R$ for each system $n$, and its fluid limit $q$ (established in Corollary \ref{cor:BinAdmissible}).
We denote the unique PE (the global limit cycle) of that fluid limit by $q_e$.

Before proving Theorem \ref{th:MomentCheck},
we state two technical lemmas---Lemma \ref{LEM:ASYMPQ2} and Lemma \ref{LEM:ASYMPB} below.
These two lemmas are proved in Theorem 4 and Proposition 6.1 in \cite{Hu20Momentarxiv}, and are restated here for completeness.
%below---which are key to the proof of the theorem\yue{\sout{; the proofs of these two lemmas appear in Appendix \ref{ap:AsympMoment}}}.
Recall that, if $\bar Q^n$ is stationary for each $n \ge 1$, then $\bar Q^n \Ra q_e$ in $D^K$ as $n\tinf$ by Lemma \ref{lem:interchange}.
%Lemma \ref{LEM:ASYMPQ2} establishes results for the convergence of moments and cross moments of the stationary queue
%at the polling epochs, while Lemma \ref{LEM:ASYMPB} establishes results concerning the convergence of the moments for the sequence of busy times.
%Of course, we must first know that the aforementioned moments exist (under the conditions of Theorem \ref{th:MomentCheck}) for the stochastic system.
%Unfortunately, despite there being a relatively rich literature that is concerned with computational algorithms for moments of performance measures
%under different controls, including the binomial-exhaustive control, we are unaware of results that establish the existence of those moments.
%See, e.g., \cite{swartz1980polling,ferguson1986computation,sarkar1989expected,levy1989delay,konheim1994descendant,choudhury1996computing}.
%\yue{Thus, due to the independent interest, we prove the two lemmas in our working paper, where we close this theoretical gap in the literature by proving the required moment-existence theorems in the setting of Theorem \ref{th:MomentCheck}.}

\begin{lemma}
\label{LEM:ASYMPQ2}
Assume that, for all $n \ge 1$, $\bar Q^n(0) \deq \tilde Q^n(\infty)$, so that the process $Q^n$ is stationary. Then
\begin{enumerate}[(i)]
\item $\E\off{\bar Q_k^n (\bar A_i^n)} = q_{e,k}(a_i)$ for all $n \geq 1$, $k \in \mathcal{K}$, $i \in \mathcal{I}^L$.
\item  If
(a) $\E\off{S_k^2} < \infty$ for all $k \in \mathcal{K}$,
(b) $\E\off{(V_i^n)^2} < \infty$ for all $i \in \mathcal{I}^L$, $n \geq 1$,
and (c) $\E\off{(\bar V_i^n)^2} \arr s_i^2$ as $n \arr \infty$ for all $i \in \mathcal{I}^L$,
then $\E  \off{Q_k^n (A_i^n) Q_j^n (A_i^n)}  < \infty$ for all $n \geq 1$ and
$$\lim_{n\tinf} \E\off{\bar Q_k^n (\bar A_i^n) \bar Q_j^n (\bar A_i^n)  } = q_{e,k}(a_i) q_{e,j}(a_i),
\qforallq k, j \in \mathcal{K}, \text{ and } i \in \mathcal{I}^L.$$
\item If (a) for each $k\in \mathcal{K}$, there exists $\epsilon_k >0$ such that $\E\off{e^{t S_k}} < \infty$ for all $t \in (-\epsilon_k, \epsilon_k)$,
(b) $\E\off{e^{t V_i^n}} < \infty$ for all $t \in \RR_+$, $i \in \mathcal{I}^L$, $n \geq 1$, (c) $\E\off{(\bar V_i^n)^\ell} \arr s_i^\ell$ as $n \arr \infty$ for all $\ell \geq 3$, $i \in \mathcal{I}^L$,
then $\E\off{Q_k^n (A_i^n)^\ell} < \infty$ for all $n \geq 1$
and
$$\lim_{n\tinf} \E\off{\bar Q_k^n (\bar A_i^n)^\ell} = (q_{e,k}(a_i))^\ell \qforallq \ell \geq 3, \,\, k \in \mathcal{K}, \text{ and } i \in \mathcal{I}^L.$$
\end{enumerate}
\end{lemma}

Recall (see Section \ref{SecTranslatingPRC}) that, for each stage $i$,
$\Theta_{p(i)}^{(\ell)}$ denotes the busy period ``generated'' by the service of the $\ell$th served customer in queue $p(i)$,
which is the queue being polled at stage $i$.

\begin{lemma}  \label{LEM:ASYMPB}
Assume that $\bar Q^n(0) \deq \tilde Q^n(\infty)$, so that $Q^n$ is stationary, for all $n \ge 1$, and consider the corresponding sequence of busy times
$\{B_i^n : i \in \mathcal{I}^L , n \ge 1\}$ (over a generic stationary server cycle). Then
\begin{enumerate}[(i)]
\item
$\E\off{\bar B_i^n} = r_i q_{e, p(i)}(a_i) \E\off{\Theta_{p(i)}} $ for all $n \geq 1$, $i \in \mathcal{I}^L$.
\item  If
(a) $\E\off{S_k^2} < \infty$ for all $k \in \mathcal{K}$,
(b) $\E\off{(V_i^n)^2} < \infty$ for all $i \in \mathcal{I}^L$, $n \geq 1$,
and (c) $\E\off{(\bar V_i^n)^2} \arr s_i^2$ as $n \arr \infty$ for all $i \in \mathcal{I}^L$,
then $\E\off{\of{\bar B_{i}^n}^2} < \infty$ for all $n \geq 1$ and
\begin{equation*}
\lim_{n \arr \infty} \E\off{\of{\bar B_{i}^n}^2} \arr \of{r_i q_{e,p(i)}(a_i) \E\off{\Theta_{p(i)}}}^2, \quad \text{for all } i \in \mathcal{I}^L .
\end{equation*}
\item If (a) for each $k\in \mathcal{K}$, there exists $\epsilon_k >0$ such that $\E\off{e^{t S_k}} < \infty$ for all $t \in (-\epsilon_k, \epsilon_k)$,
(b) $\E\off{e^{t V_i^n}} < \infty$ for all $t \in \RR_+$, $i \in \mathcal{I}^L$, $n \geq 1$, (c) $\E\off{(\bar V_i^n)^\ell} \arr s_i^\ell$ as $n \arr \infty$ for all $\ell \geq 3$, $i \in \mathcal{I}^L$,
then $\E\off{\of{\bar B_{i}^n}^\ell} < \infty$ for all $n \geq 1$
and
$$\lim_{n\tinf} \E\off{\of{\bar B_{i}^n}^\ell} = \of{r_i q_{e,p(i)}(a_i) \E\off{\Theta_{p(i)}}}^\ell \qforallq \ell \geq 3, \,\, i \in \mathcal{I}^L.$$
\end{enumerate}
\end{lemma}

\begin{proof}[Proof of Theorem \ref{th:MomentCheck}]
We prove that the two assertions in the theorem hold for the binomial-exhaustive policy under any control parameters $(L, \mathbf r) \in \NN \times \mathcal R$,
and so, in particular, for $(L_*, \mathbf r_*)$.
\paragraph*{Proof of $\boldsymbol{(i)}$}
Since $\psi(x) = O(||x||^p)$ for some $p > 1$, there exist $x_0 \in \RR_+$ and $M \in \RR_+$ such that for all $x$ with
$\max_{ k \in \mathcal{K}} x_k \geq x_0$, we have
\begin{equation} \label{eq:89}
\psi(x) \leq M ||x||^p \leq
M (K (\max_{ k \in \mathcal{K}} x_k)^2)^\frac{p}{2}
\leq
M K^\frac{p}{2} \sum_{k \in \mathcal{K}} x_k^p .
\end{equation}
Let $M' := M K^\frac{p}{2}$, and $\mathbf x_0 := (x_0,...,x_0) \in \RR_+^K$.
Due to \eqref{eq:89} and the fact that $\psi$ is non-decreasing, $\bar \Psi^n_{(L,\mathbf r)}$ satisfies
\begin{align*}
\stepcounter{equation}\tag{\theequation}\label{eq:90}
\bar \Psi^n_{(L,\mathbf r)} & = \int_{0}^{\bar T^n_L} \psi(\bar Q^n(u)) du \\
& \leq \int_{0}^{\bar T^n_L}  \of{\psi(\mathbf x_0) \mathbf  1_{\{ \max_{ k \in \mathcal{K}} \bar Q_k^n(u) < x_0 \} }
+ \of{M' \sum_{k \in \mathcal{K}} (\bar Q_k^n(u))^p} \mathbf  1_{\{ \max_{ k \in \mathcal{K}} \bar Q_k^n(u) \geq x_0 \} } }  du \\
&\leq \bar T^n_L \psi(\mathbf x_0) + \int_{0}^{\bar T^n_L}   \of{M' \sum_{k \in \mathcal{K}} (\bar Q_k^n(u))^p} du.
\end{align*}

We start by showing that $\{\bar T^n_L \psi(\mathbf x_0): n \geq 1\}$ is UI. To do this, note that the steady-state cycle length $\bar T^n_L$ can be represented as
$\bar T^n_L = \sum_{i \in \mathcal{I}^L} \of{\bar B_i^n + \bar V_i^n}$.
Similar derivation to that of \eqref{barB} gives
\begin{equation} \label{eq:91}
\bar B_i^n \Arr r_i q_{e,p(i)}(a_i) \E\off{\Theta_{p(i)}} \quad \text{as } n \arr \infty.
\end{equation}
and, by Lemma \ref{LEM:ASYMPB} (i),
\begin{equation} \label{convBmeans}
\E\off{\bar B_i^n} = r_i q_{e,p(i)}(a_i) \E\off{\Theta_{p(i)}}  \quad \text{for all } n \geq 1.
\end{equation}
Since $\bar B^n_i \ge 0$ for all $n \ge 1$, the two convergence in \eqref{eq:91} and in \eqref{convBmeans} imply together that
the sequence $\{\bar B_i^n : n \geq 1\}$ is UI; e.g., \cite[Theorem 5.4]{billingsley1968convergence}.
Together with the fact that $\{\bar V_i^n : n \geq 1\}$ is UI by Assumption \ref{assum1},
we get that $\{\bar T^n_L : n \geq 1\}$, and thus $\{\bar T^n_L \psi(\mathbf x_0): n \geq 1\}$, is UI.

We next prove that the second term in the right-hand side of \eqref{eq:90} is UI.
To this end, let $\tilde B_i^n$ denote the busy time if the exhaustive policy is employed at stage $i$, when the initial queue length at the corresponding
polling epoch, the arrival process to the queue, and the service times of all customers served during $B^n_i$ remain unchanged, so that $\bar B^n_i \le \tilde B^n_i$ w.p.1
for all $n \ge 1$. Then
%\begin{equation}  \label{eq:UI2}
%\begin{split}
\begin{align}\label{eq:UI2}
M' \int_{0}^{\bar T^n_L} \sum_{k \in \mathcal{K}} (\bar Q_k^n(u))^p  \, du
&= M' \sum_{i \in \mathcal{I}^L} \int_{\bar A_i^n}^{\bar A_i^n + \bar B_i^n + \bar V_i^n}
\sum_{k \in \mathcal{K}} (\bar Q_k^n(u))^p \, du \nonumber \\
\quad &\leq M'  \sum_{i \in \mathcal{I}^L} \of{\bar B_i^n + \bar V_i^n} \sum_{k \in \mathcal{K}} \of{\bar Q_k^n(\bar A_i^n) + \bar{\mathcal P}_k^n(\bar B_i^n
+ \bar V_i^n)}^{p} \nonumber \\
\quad &\leq M'  \sum_{i \in \mathcal{I}^L} \of{\tilde B_i^n + \bar V_i^n} \sum_{k \in \mathcal{K}} \of{\bar Q_k^n(\bar A_i^n) + \bar{\mathcal P}_k^n(\tilde B_i^n
+ \bar V_i^n)}^{p},
%\end{split}
%\end{equation}
\end{align}
where the first inequality is due to the omission of the service process at stage $i$.

We next show that, for any $\ell \geq 1$,
\begin{equation} \label{eq:70}
\sup_{n \geq 1} \, \E\off{ \of{\tilde B_i^n + \bar V_i^n}^{\ell} } < \infty ,
\end{equation}
and
\begin{equation} \label{eq:71}
\sup_{n \geq 1} \, \E\off{\of{\bar Q_k^n(\bar A_i^n) + \bar{\mathcal P}_k^n(\tilde B_i^n + \bar V_i^n)}^{\ell}  } < \infty,
\end{equation}
from which it follows that, for any $\ep > 0$,
\begin{equation*} \label{eq:38}
\sup_{n \geq 1} \, \E\off{ \of{\tilde B_i^n + \bar V_i^n}^{1 +\epsilon} \of{\bar Q_k^n(\bar A_i^n)
+ \bar{\mathcal P}_k^n(\tilde B_i^n + \bar V_i^n)}^{p (1 +\epsilon)}  } < \infty
\end{equation*}
by virtue of H\"{o}lder's inequality, so that the sequence of bounds in \eqref{eq:UI2} is UI.

The inequality in \eqref{eq:70} follows directly from the fact that $\tilde B_i^n$ and $\bar V_i^n$ are independent, and both are uniformly bounded in $n$.
Indeed, $\sup_n \E [(\tilde B_i^n)^\ell] < \infty$ by Lemma \ref{LEM:ASYMPB} (i)--(iii),
and $\sup_n \E[(\bar V_i^n)^{\ell}] < \infty$ by Assumption \ref{assum2}.	
To prove \eqref{eq:71}, note that
%\begin{equation} \label{eq:40}
%\begin{split}
\begin{align} \label{eq:40}
\E\off{\of{\bar Q_k^n(\bar A_i^n) + \bar{\mathcal P}_k^n(\tilde B_i^n + \bar V_i^n)}^{\ell}  }
& = \E\off{  \sum_{j=0}^{\ell}  {\ell \choose j} \of{\bar Q_k^n(\bar A_i^n) }^{\ell-j} \of{\bar{\mathcal P}_k^n(\tilde B_i^n + \bar V_i^n)}^j } \nonumber \\
& \leq \sum_{j=0}^{\ell}  {\ell \choose j} \E\off{   \of{\bar Q_k^n(\bar A_i^n) }^{(\ell-j)\alpha} }^{\frac{1}{\alpha}}
\E\off{ \of{\bar{\mathcal P}_k^n(\tilde B_i^n + \bar V_i^n)}^{j \beta}   }^{\frac{1}{\beta}} ,
\end{align}
%\end{split}
%\end{equation}
where the equality holds the the Binomial Theorem, and the inequality follows from H\"{o}lder's inequality, for $\alpha > 1$ and $1/\alpha + 1/\beta = 1$.

Let ${\cdot \brace \cdot} $ denote the Stirling numbers of the second type, and recall that, for a Poisson random variable $Y$ with mean $\nu$, it holds that
$\mathbb E[Y^N] = \sum_{j=1}^N {N \brace j}\nu^j$, for $N \in \NN$. Thus,
\begin{equation} \label{eq:41}
\begin{split}
\E\off{ \of{\bar{\mathcal P}_k^n(\tilde B_i^n + \bar V_i^n)}^{j \beta}   }  &= \E\off{ \E\off{ \of{\bar{\mathcal P}_k^n(\tilde B_i^n + \bar V_i^n)}^{j \beta} \bigg| \tilde B_i^n + \bar V_i^n  }   } \\
&= \E\off{ \E\off{ \of{\frac{1}{n} \mathcal P_k^n( \tilde B_i^n n + \bar V_i^n n )}^{j \beta} \bigg| \tilde B_i^n + \bar V_i^n }   } \\
&= \E\off{  \frac{1}{n^{j \beta}} \sum_{m=0}^{j \beta} \of{\lambda_k ( \tilde B_i^n n + \bar V_i^n n )}^{m} {j\beta \brace m}   } \\
&= \E\off{ \sum_{m=0}^{j \beta} \frac{1}{n^{j\beta-m}} \lambda_k^m (\tilde B_i^n + \bar V_i^n)^{m} {j\beta \brace m} } \\
&= \sum_{m=0}^{j \beta} \frac{1}{n^{j\beta-m}} \lambda_k^m {j\beta \brace m}  \E\off{ (\tilde B_i^n + \bar V_i^n)^{m} } \\
&= \lambda_k^{j\beta}  \E\off{ (\tilde B_i^n + \bar V_i^n)^{j\beta} } + o(1).
\end{split}
\end{equation}
Plugging \eqref{eq:41} into \eqref{eq:40}, we get
\begin{equation*}
\begin{split}
&\E\off{\of{\bar Q_k^n(\bar A_i^n) + \bar{\mathcal P}_k^n(\tilde B_i^n + \bar V_i^n)}^{\ell}  }  \\
&\leq \sum_{j=0}^{\ell}  {\ell \choose j} \E\off{   \of{\bar Q_k^n(\bar A_i^n) }^{(\ell-j)\alpha} }^{\frac{1}{\alpha}} \of{\lambda_k^{j\beta}
\E\off{ (\tilde B_i^n + \bar V_i^n)^{j\beta} } + o(1)}^{\frac{1}{\beta}} ,
\end{split}
\end{equation*}
which is uniformly bounded in $n$ due to Lemma \ref{LEM:ASYMPQ2} (i)--(iii), Lemma \ref{LEM:ASYMPB} (i)--(iii), Assumption \ref{assum2},
and the independence of $\tilde B_i^n$ and $\bar V_i^n$. Thus, \eqref{eq:71} holds.

\paragraph*{Proof of $\boldsymbol{(ii)}$}
For $x \in \RR^K$, let $f(x) = \sum_{k=1}^{K} c_k x_k$.
Since $\psi(x) = O(f(x))$, there exist $x_0 \in \RR_+$ and $M \in \RR_+$, such that for all $x$ with
$\max_{ k \in \mathcal{K}} x_k \geq x_0$,
\begin{equation*}
\psi(x) \leq M f(x) =
M \sum_{k \in \mathcal{K}} c_k x_k.
\end{equation*}
As in \eqref{eq:90}, this implies that
\begin{equation} \label{eq:96}
\begin{split}
\bar \Psi^n_{(L,\mathbf r)} & = \int_{0}^{\bar T^n_L} \psi(\bar Q^n(u)) du
\leq \bar T^n_L \psi(\mathbf x_0) + \int_{0}^{\bar T^n_L}   \of{M \sum_{k \in \mathcal{K}} c_k \bar Q_k^n(u)} du. 	
\end{split}
\end{equation}			

Since $\{\bar T^n_L \psi(\mathbf x_0) : n \geq 1\}$ was shown to be UI in the proof of part (i) above,
we need only show that the sequence corresponding to the second term on the right-hand side of \eqref{eq:96} is UI.
Similarly to the derivation of \eqref{eq:UI2}, we can bound this term from above via
\begin{align*}
\stepcounter{equation}\tag{\theequation}\label{eq:UI}
&\int_{0}^{\bar T^n_L}   \of{M \sum_{k \in \mathcal{K}} c_k \bar Q_k^n(u)} du \\
&\stackrel{}{\leq}  M \sum_{k \in \mathcal{K}} c_k  \sum_{i \in \mathcal{I}^L} \of{\bar Q_k^n(\bar A_i^n) + \bar{\mathcal P}_k^n(\tilde B_i^n + \bar V_i^n)} \of{\tilde B_i^n + \bar V_i^n} \\
&= M \sum_{k \in \mathcal{K}} c_k  \sum_{i \in \mathcal{I}^L} \of{\bar Q_k^n(\bar A_i^n) \tilde B_i^n +  \bar Q_k^n(\bar A_i^n) \bar V_i^n + \bar{\mathcal P}_k^n(\tilde B_i^n + \bar V_i^n) (\tilde B_i^n + \bar V_i^n) }.
\end{align*}
%\begin{equation} \label{eq:UI}
%\begin{split}
%&\int_{0}^{\bar T^n_L}   \of{M \sum_{k \in \mathcal{K}} c_k \bar Q_k^n(u)} du \\
%&\stackrel{}{\leq}  M \sum_{k \in \mathcal{K}} c_k  \sum_{i \in \mathcal{I}^L} \of{\bar Q_k^n(\bar A_i^n) + \bar{\mathcal P}_k^n(\tilde B_i^n + \bar V_i^n)} \of{\tilde B_i^n + \bar V_i^n} \\
%&= M \sum_{k \in \mathcal{K}} c_k  \sum_{i \in \mathcal{I}^L} \of{\bar Q_k^n(\bar A_i^n) \tilde B_i^n +  \bar Q_k^n(\bar A_i^n) \bar V_i^n + \bar{\mathcal P}_k^n(\tilde B_i^n + \bar V_i^n) (\tilde B_i^n + \bar V_i^n) }.
%\end{split}
%\end{equation}
We next show that the sequence corresponding to each term in the right-hand side of \eqref{eq:UI} is UI.	
First, by Lemma \ref{lem:interchange}, $\bar Q_k^n( \bar A_i^n) \Arr q_{e,k}(a_i)$ in $\RR_+$ as $n \arr \infty$.
In addition, it follows from \eqref{eq:91} (setting $r_i = 1$) that
\begin{equation} \label{eq:B_i^n}
\begin{split}
\tilde B_i^n \Arr q_{e,p(i)}(a_i) \E\off{\Theta_{p(i)}} \quad \text{as } n \arr \infty,
\end{split}
\end{equation}
so that, by	Slutsky's theorem,
\begin{equation} \label{eq:87}
\bar Q_k^n( \bar A_i^n) \tilde B_i^n \Arr q_{e,k}(a_i) q_{e,p(i)}(a_i) \E\off{\Theta_{p(i)}} \quad \text{as } n \arr \infty.
\end{equation}
Now,
\begin{equation} \label{eq:92}
\begin{split}
\E\off{\bar Q_k^n( \bar A_i^n) \tilde B_i^n}
&= \E\off{	\E\off{\bar Q_k^n( \bar A_i^n) \frac{1}{n} \sum_{j=1}^{Q^n_{p(i)}( A_i^n)} \Theta_{p(i)}^{(j)}  \bigg| Q^n( A_i^n) } } \\
&= \E\off{ \bar Q_k^n(\bar A_i^n) \bar Q^n_{p(i)}(\bar A_i^n) }\E\off{\Theta_{p(i)}}  \\
&\arr q_{e,k}(a_i) q_{e,p(i)}(a_i) \E\off{\Theta_{p(i)}} \quad \text{as } n \arr \infty \quad \mbox{by Lemma \ref{LEM:ASYMPQ2} (ii).}
\end{split}
\end{equation}
It follows from \eqref{eq:87}, \eqref{eq:92}, and the fact that both $\bar Q_k^n( \bar A_i^n) \tilde B_i^n$ and $q_{e,k}(a_i) q_{e,p(i)}(a_i) \E\off{\Theta_{p(i)}}$ are non-negative, that the sequence $\{ \bar Q_k^n( \bar A_i^n) \tilde B_i^n : n \geq 1 \}$ is UI. %see, e.g., \cite[Theorem 5.4]{billingsley1968convergence}).
Second, $\bar Q_k^n(\bar A_i^n)$ and $\bar V_i^n$ being independent implies that
$$\E\off{\of{\bar Q_k^n(\bar A_i^n) \bar V_i^n}^2} =\E\off{\of{\bar Q_k^n(\bar A_i^n)}^2 } \E\off{\of{\bar V_i^n}^2}.$$
Because $\E\off{\of{\bar V_i^n}^2}  < \infty$ under Assumption \ref{assum2}, and $ \E\off{\of{\bar Q_k^n(\bar A_i^n)}^2} < \infty$ given $\E\off{S_k^2} < \infty$
by Lemma \ref{LEM:ASYMPQ2} (ii), the second moment of $\bar Q_k^n(\bar A_i^n) \bar V_i^n$ is finite,
implying that $\{\bar Q_k^n(\bar A_i^n) \bar V_i^n : n \geq 1\}$ is UI.

Lastly, for $\bar{\mathcal P}_k^n(\tilde B_i^n + \bar V_i^n) (\tilde B_i^n + \bar V_i^n)$,
note that $\tilde B_i^n \Arr q_{e,p(i)}(a_i) \E\off{\Theta_{p(i)}} $ by \eqref{eq:B_i^n}, and $\bar V_i^n \Arr s_i$ as $n \arr \infty$.
By the FWLLN for Poisson processes, we have
\begin{equation} \label{eq:93}
\bar{\mathcal P}_k^n(\tilde B_i^n + \bar V_i^n)
\Arr \lambda_k \of{q_{e,p(i)}(a_i) \E\off{\Theta_{p(i)}}   + s_i } \quad \text{as } n \arr \infty.
\end{equation}
By Slutsky's theorem, $\bar{\mathcal P}_k^n(\tilde B_i^n + s_i)  (\tilde B_i^n + \bar V_i^n) \Arr \lambda_k \of{q_{e,p(i)}(a_i) \E\off{\Theta_{p(i)}} + s_i}^2 $ as $n \arr \infty$.
Next, %the expectation of $\bar{\mathcal P}_k^n(\tilde B_i^n + \bar V_i^n)  (\tilde B_i^n + \bar V_i^n) $ satisfies
\begin{equation} \label{eq:94}
\begin{split}
\E\off{\bar{\mathcal P}_k(\tilde B_i^n + \bar V_i^n) (\tilde B_i^n + \bar V_i^n)  }
=& \E\off{\E\off{\bar{\mathcal P}_k(\tilde B_i^n + \bar V_i^n) (\tilde B_i^n + \bar V_i^n)   \big| \tilde B_i^n + \bar V_i^n } } \\
=& \E\off{\lambda_k (\tilde B_i^n + \bar V_i^n)^2  } \\
\arr& \lambda_k \of{q_{e,p(i)}(a_i) \E\off{\Theta_{p(i)}}  + s_i}^2 \quad \text{as } n \arr \infty,
\end{split}
\end{equation}
where the limit follows from Lemma \ref{LEM:ASYMPB} (i)--(ii) and Assumption \ref{assum2}.
Since both the pre-limit and limit in \eqref{eq:93} are non-negative, Theorem 5.4 in \cite{billingsley1968convergence},
\eqref{eq:93} and \eqref{eq:94} imply that $\{\bar{\mathcal P}_k^n(\tilde B_i^n + \bar V_i^n)  (\tilde B_i^n + \bar V_i^n) : n \geq 1 \}$ is UI, and in turn,
so is $\{\bar \Psi^n_{(L,\mathbf r)} : n \ge 1\}$.
\end{proof}

\section{Summary and Future Research} \label{secSummary}
We considered the optimal-control problem of polling systems with large switchover times.
Under the large-switchover-time scaling, we established that the binomial-exhaustive policy, with properly chosen control parameters, is asymptotically optimal.
Those optimal control parameters are computed by solving an FCP for a related deterministic relaxation, which is described via an HDS, and arises
as the fluid limit for a sequence of stochastic systems operating under the corresponding binomial-exhaustive policy.
For the important special case in which the basic table is cyclic and the cost function is separable convex and has at most a polynomial growth,
we showed that the exhaustive policy is asymptotically optimal.

The analytical tools in this paper can be useful in characterizing asymptotically optimal controls in other settings.
For example, the Stochastic Economic Lot Scheduling Problem (SELSP),
can be modeled as a polling system in which backlogged demand implies that the buffer content can be negative;
see, e.g., \cite{federgruen1996stochastic}.
Further, the stability region of the fluid model for polling systems is easier to characterize than that of the underlying stochastic system, and can
therefore be used to study the stability of stochastic polling systems under controls that do not adhere to the
conditions in \cite{fricker1994monotonicity}.

\section*{Acknowledgement}
We thank Professor Hanoch Levy for valuable conversations regarding the binomial-exhaustive policy,
which he proposed in \cite{levy1988optimization}.

\begin{appendix}

%\section{Solving the FCP} \label{ap:MoreFCP}
%
%
%In section \ref{pf:linearcyclic} of the appendix we prove Proposition \ref{LEM:LINEARCYCLIC} from the main paper.
%In Section \ref{ap:FCPExamples} we consider the class of non-cyclic basic tables in which one queue appears more than once, and show that the exhaustive policy
%is fluid optimal (solves the FCP) for this class as well, when the cost is separable convex.

\section{Proof of Proposition \ref{LEM:LINEARCYCLIC}} \label{pf:linearcyclic}
The proof of Proposition \ref{LEM:LINEARCYCLIC} involves approximating $\psi$ with piecewise linear functions.
Since $\psi$ is assumed to be separable convex, each of its components $\psi_k$ is an increasing convex function mapping $\RR_+$ into itself,
and can therefore be approximated over any compact interval by piecewise linear functions of the form
\begin{equation} \label{eq:95}
\begin{split}
p_k(x) &=
\begin{cases}
p_k^{(1)} x \quad &\text{if } \alpha_k^{(1)} \leq x < \alpha_k^{(2)} \\
p_k^{(2)} x \quad &\text{if } \alpha_k^{(2)} \leq x < \alpha_k^{(3)} \\
\vdots \\
p_k^{(N_k-1)} x \quad &\text{if } \alpha_k^{(N_k-1)} \leq x < \alpha_k^{(N_k)} \\
p_k^{(N_k)} x \quad &\text{if } \alpha_k^{(N_k)} \leq x. \\
\end{cases}
\end{split}
\end{equation}
where $0 \leq p_k^{(1)} < p_k^{(2)} \dots < p_k^{(N_k)}$ in $\RR_+$ and $0 = \alpha_k^{(1)} < \alpha_k^{(2)} \dots < \alpha_k^{(N_k)} $ in $\RR_+$.
We refer to $\{\af^{(\ell)}_k : 1 \le \ell \le N_k\}$ as the {\em irregular points} (of $p_k$),
since $p_k$ is differentiable everywhere except at those points, and to $\{p^{(\ell)}_k : 1 \le \ell \le N_k\}$ as the {\em coefficients} of $p_k$.

The next lemma is the key to proving Proposition \ref{LEM:LINEARCYCLIC}; its proof appears in Section \ref{apProofLemmaPWL} below.
We say that $f: \RR_+^{K} \ra \RR_+$ is separable and piecewise linear if for $x \in \RR_+^K$, $f(x) = \sum_{k=1}^{K} f_k(x_k)$
and each $f_k$ is a piecewise linear function, mapping $\RR_+$ into itself.
\begin{lemma} \label{lem:FCPPieceWiseLinearCost}
If the basic table is cyclic and $\psi$ is separable and piecewise linear  (in addition to being nondecreasing and continuous),
then $q_{exh}$ solves \eqref{eq:FluidOpt}.
\end{lemma}

\begin{proof} [Proof of Proposition \ref{LEM:LINEARCYCLIC}]
Fix $L \in \NN$. For any $L$-cycle PE-candidate $q_e^L$ with cycle length $\tau_L$, the trajectory of queue $k$ is bounded from below by $0$ and from above by
$M_k := \frac{1}{2} (1-\rho_k) \lambda_k \tau_L^2$, so that
\begin{equation*}
||\psi_k(q_{e,k}^L)||_{\tau_L} \leq \psi_k\of{M_k} , \quad k \in \mathcal K .
\end{equation*}
Hence, for any $\ep > 0$, there exist piecewise linear functions $p_k$ and $h_k$, both mapping $\RR_+$ to itself,
such that for all $y \in [0, \psi_k\of{M_k}]$,
\begin{equation*}
0 < p_k (y) - \psi_k (y) < \ep/K \quad \text{and} \quad 0 <  \psi_k (y) - h_k (y) < \ep/K   , \quad k \in \mathcal{K} .
\end{equation*}
Let $p, h: \RR_+^K \arr \RR_+$ be defined for $x \in \RR_+^K$ via $p(x) := \sum_{k \in \mathcal{K}} p_k(x_k)$ and $h(x) := \sum_{k \in \mathcal{K}} h_k(x_k)$.
It follows that
\begin{equation*}
0 < \frac{1}{\tau_L} \int_0^{\tau_L} p (q_{e}^L(u)) du  - \frac{1}{\tau_L} \int_0^{\tau_L} \psi (q_{e}^L(u)) du < \ep ,
\end{equation*}
and
\begin{equation*}
0 < \frac{1}{\tau_L} \int_0^{\tau_L} \psi (q_{e}^L(u)) du - \frac{1}{\tau_L} \int_0^{\tau_L} h (q_{e}^L(u)) du  < \ep .
\end{equation*}
Consider two $L$-cycle optimization problems, denoted by $P_p$ and $P_h$,
which replace the objective function $\psi$ in problem \eqref{eq:FluidOpt2} with $p$ and $h$, respectively.
Since $q_{exh}$ is optimal for both $P_p$ and $P_h$ by Lemma \ref{lem:FCPPieceWiseLinearCost}
and $\ep$ is arbitrary, $q_{exh}$ is a solution to $L$-cycle optimization problem \eqref{eq:FluidOpt2}.
As the latter statement holds for all $L \in \NN$, $q_{exh}$ is a solution to the global optimization problem \eqref{eq:FluidOpt}.
\end{proof}

\subsection{Proof of Lemma \ref{lem:FCPPieceWiseLinearCost}} \label{apProofLemmaPWL}
For each $k \in \mathcal{K}$, let $\psi_k$ be in the form of \eqref{eq:95},
with irregular points $\{\af^{(\ell)}_k : 1 \le \ell \le N_k\}$ and coefficients $\{\psi^{(\ell)}_k : 1 \le \ell \le N_k\}$, for some $N_k \in \NN$.
Fix $L \in \NN$. To show that $q_{exh}$ is a solution of the $L$-cycle optimization problem \eqref{eq:FluidOpt2},
we consider a relaxed problem, in which each queue is optimized without consideration of all other queues.
To this end, for each $k \in \mathcal K$,
we consider the following relaxation to \eqref{eq:FluidOpt2}:
\begin{equation} \label{eq:FluidOpt4}
\begin{split}
\min_{q^L \in \mathcal Q^L} \quad  &\frac{1}{\tau_L} \int_{0}^{\tau_L} \psi_k(q^L_{k}(u)) \, du.
\end{split}
\end{equation}
Note that we have $K$ different optimization problems of the form \eqref{eq:FluidOpt4}---one for each $k \in \mathcal K$.
For each of these $K$ optimization problems, let $q^L_{e,k}$ denote the $k$th component of a solution to the problem \eqref{eq:FluidOpt4}.
The closed curve $q^L_e := (q^L_{e,k}, k \in \mathcal K)$
necessarily gives a lower bound for the optimal objective value in the $L$-cycle optimization problem \eqref{eq:FluidOpt2}.
(However, $q^L_e$ needs not be an element of $\mathcal Q^L$, as it may not be a bona-fide $L$-cycle PE-candidate.)
The statement of the lemma will therefore follow by proving that $q_{exh}$, which is a feasible solution to \eqref{eq:FluidOpt2},
is a solution to \eqref{eq:FluidOpt4} for each $k \in \mathcal K$.

Fix $k \in \mathcal{K}$. For $\psi_k^{(0)} := 0$, let
\begin{equation*}
h_k^{(\ell)} := \psi_k^{(\ell)}  - \psi_k^{(\ell-1)} , \quad 1 \le \ell \le N_k.
\end{equation*}
The objective function in \eqref{eq:FluidOpt4} satisfies
\begin{equation*}
\begin{split}
\frac{1}{\tau_L} \int_{0}^{\tau_L} \psi_k(q^L_{k}(s)) \, ds
& = \frac{1}{\tau_L} \int_{0}^{\tau_L} \left[ \psi_k^{(1)} q^L_{k}(s) \mathbf 1_{\{q^L_{k}(s) \geq \alpha_k^{(1)}\}} + (\psi_k^{(2)} - \psi_k^{(1)}) q^L_{k}(s) \mathbf 1_{\{q^L_{k}(s) \geq \alpha_k^{(2)}\}} \right. \\
&\left. \quad + \cdots + (\psi_k^{(N_k)} - \psi_k^{(N_k-1)}) q^L_{k}(s) \mathbf 1_{\{q^L_{k}(s) \geq \alpha_k^{(N_k)}\}}   \right]  \, ds \\
& = \frac{1}{\tau_L} \int_{0}^{\tau_L} \left[ h_k^{(1)} q^L_{k}(s)\mathbf 1_{\{q^L_{k}(s) \geq \alpha_k^{(1)}\}}  + h_k^{(2)} q^L_{k}(s) \mathbf 1_{\{q^L_{k}(s) \geq \alpha_k^{(2)}\}} \right. \\
&\left. \quad + \cdots + h_k^{(N_k)} q^L_{k}(s) \mathbf 1_{\{q^L_{k}(s) \geq \alpha_k^{(N_k)}\}}   \right]  \, ds \\
& =  \frac{1}{\tau_L} \sum_{\ell = 1}^{N_k} h_k^{(\ell)} \int_{0}^{\tau_L}  q^L_{k}(s) \mathbf 1_{\{q^L_{k}(s) \geq \alpha_k^{(\ell)}\}} ds .
\end{split}
\end{equation*}

Then for
\begin{equation*}
A^{(\ell)}\of{q^L_{k}} := \int_{0}^{\tau_L}  q^L_{k}(s) \mathbf 1_{\{q^L_{k}(s) \geq \alpha_k^{(\ell)}\}} ds , \quad 1 \le \ell \le N_k,
\end{equation*}
\eqref{eq:FluidOpt4} is equivalent to
\begin{equation} \label{eq:FluidOpt6}
\begin{split}
\min_{q^L \in \mathcal Q^L} \quad  &\frac{1}{\tau_L} \sum_{\ell = 1}^{N_k}  h_k^{(\ell)} A^{(\ell)}\of{q^L_{k}}  \\
s.t. \quad
&A^{(\ell)}\of{q^L_{k}} = \int_{0}^{\tau_L}  q^L_{k}(s) \mathbf 1_{\{q^L_{k}(s) \geq \alpha_k^{(\ell)}\}} ds , \quad 1 \le \ell \le N_k. \\
\end{split}
\end{equation}
Now, $A^{(\ell)}(q^L_{k})$ can be further partitioned into $L$ sub-areas over each table cycle. In particular, let
\begin{equation*}
a^{(\ell, m)}\of{q^L_{k}} := \int_{u^{(m-1)}}^{u^{(m)}}  q^L_{k}(s) \mathbf 1_{\{q^L_{k}(s) \geq \alpha_k^{(\ell)}\}} ds , \quad 1 \le m \le L, \quad 1 \le \ell \le N_k,
\end{equation*}
where $u^{(m-1)}$ denotes the beginning epoch of the $m$th {\em table} cycle.
(To facilitate the notation henceforth, the superscript $(m)$ in $u^{(m)}$ is an index for table cycles of $q^L$ over the cycle length $[0, \tau_L]$.
This is different from the convention elsewhere in the paper, e.g., in \eqref{fluid implicit}, where $(m)$ was indexing server cycles.)
We can then write
\begin{equation*}
A^{(\ell)}\of{q^L_{k}} = \sum_{m=1}^L a^{(\ell,m)}\of{q^L_{k}}  , \quad 1 \le \ell \le N_k,
\end{equation*}
so that \eqref{eq:FluidOpt6} can be equivalently written as
\begin{equation} \label{eq:FluidOpt8}
\begin{split}
\min_{q^L \in \mathcal Q^L} \quad  &\frac{1}{\tau_L}  \sum_{\ell = 1}^{N_k}  h_k^{(\ell)}  \sum_{m=1}^L a^{(\ell,m)}\of{q^L_{k}}   \\
s.t. \quad
&a^{(\ell, m)}\of{q^L_{k}} = \int_{u^{(m-1)}}^{u^{(m)}}  q^L_{k}(s) \mathbf 1_{\{q^L_{k}(s) \geq \alpha_k^{(\ell)}\}} ds ,
\quad 1 \le m \le L, \quad 1 \le \ell \le N_k. \\
\end{split}
\end{equation}

For
\begin{equation*}
\tau^{(\ell, m)}(q_{k}^L) := \int_{u^{(m-1)}}^{u^{(m)}}  \mathbf 1_{\{q^L_{k}(s) \geq \alpha_k^{(\ell)}\}} ds, \quad 1 \le m \le L, \quad 1 \le \ell \le N_k,
\end{equation*}
we have that
\begin{equation} \label{eq:AreaBound}
a^{(\ell, m)}\of{q^L_{k}} \geq \frac{1}{2} (1-\rho_k) \lambda_k \of{\tau^{(\ell, m)}(q_{k}^L) }^2 ,
\end{equation}
and
\begin{equation} \label{eq:LengthBound}
\sum_{m=1}^L \tau^{(\ell, m)}(q_{k}^L) \geq M^{(\ell)} ,
\end{equation}
where
\begin{equation*}
M^{(\ell)} := \max{\offf{\tau_L - L \of{\frac{\alpha_k^{(\ell)} }{\lambda_k} + \frac{\alpha_k^{(\ell)} }{\mu_k - \lambda_k}}, \, 0}}.
\end{equation*}
Adding the inequalities in \eqref{eq:AreaBound} and \eqref{eq:LengthBound} to the constraints of problem \eqref{eq:FluidOpt8} does not change its feasible region,
and yields the following equivalent formulation
\begin{align*}
\stepcounter{equation}\tag{\theequation}\label{eq:FluidOpt9}
\min_{q^L \in \mathcal Q^L} \quad  &\frac{1}{\tau_L}  \sum_{\ell = 1}^{N_k}  h_k^{(\ell)}  \sum_{m=1}^L a^{(\ell,m)}\of{q^L_{k}}   \\
s.t. \quad
&a^{(\ell, m)}\of{q^L_{k}} = \int_{u^{(m-1)}}^{u^{(m)}}  q^L_{k}(s) \mathbf 1_{\{q^L_{k}(s) \geq \alpha_k^{(\ell)}\}} ds , \\
&\tau^{(\ell, m)}(q_{k}^L) = \int_{u^{(m-1)}}^{u^{(m)}}  \mathbf 1_{\{q^L_{k}(s) \geq \alpha_k^{(\ell)}\}} ds , \\
&a^{(\ell, m)}\of{q^L_{k}} \geq \frac{1}{2} (1-\rho_k) \lambda_k \of{\tau^{(\ell, m)}(q_{k}^L) }^2 , \quad 1 \le m \le L, \quad 1 \le \ell \le N_k  \\
&\sum_{m=1}^L \tau^{(\ell, m)}(q_{k}^L) \geq M^{(\ell)} , \quad 1 \le \ell \le N_k .
\end{align*}
%\begin{equation} \label{eq:FluidOpt9}
%\begin{split}
%\min_{q^L \in \mathcal Q^L} \quad  &\frac{1}{\tau_L}  \sum_{\ell = 1}^{N_k}  h_k^{(\ell)}  \sum_{m=1}^L a^{(\ell,m)}\of{q^L_{k}}   \\
%s.t. \quad
%&a^{(\ell, m)}\of{q^L_{k}} = \int_{u^{(m-1)}}^{u^{(m)}}  q^L_{k}(s) \mathbf 1_{\{q^L_{k}(s) \geq \alpha_k^{(\ell)}\}} ds , \\
%&\tau^{(\ell, m)}(q_{k}^L) = \int_{u^{(m-1)}}^{u^{(m)}}  \mathbf 1_{\{q^L_{k}(s) \geq \alpha_k^{(\ell)}\}} ds , \\
%&a^{(\ell, m)}\of{q^L_{k}} \geq \frac{1}{2} (1-\rho_k) \lambda_k \of{\tau^{(\ell, m)}(q_{k}^L) }^2 , \quad 1 \le m \le L, \quad 1 \le \ell \le N_k  \\
%&\sum_{m=1}^L \tau^{(\ell, m)}(q_{k}^L) \geq M^{(\ell)} , \quad 1 \le \ell \le N_k .
%\end{split}
%\end{equation}
Next, let
$$\boldsymbol \tau := (\tau^{(\ell, m)}, 1 \le m \le L, 1 \le \ell \le N_k) \qandq \boldsymbol a := (a^{(\ell, m)}, 1 \le m \le L, 1 \le \ell \le N_kk).$$
We consider the following relaxed problem by dropping the first three constraints in problem \eqref{eq:FluidOpt9}
\begin{equation} \label{eq:FluidOpt10}
\begin{split}
\min_{ \boldsymbol \tau, \, \boldsymbol a } \quad  &\frac{1}{\tau_L}   \sum_{\ell = 1}^{N_k}  h_k^{(\ell)} \sum_{m=1}^L a^{(\ell,m)} \\
s.t. \quad
&a^{(\ell, m)} \geq \frac{1}{2} (1-\rho_k) \lambda_k \of{\tau^{(\ell, m)} }^2 , \quad1 \le m \le L, \quad 1 \le \ell \le N_k \\
&\sum_{m=1}^L \tau^{(\ell, m)} \geq M^{(\ell)} , \quad 1 \le \ell \le N_k . \\
\end{split}
\end{equation}
It follows from observation (and can be verified by solving the Karush-Kuhn Tucker equations)
that the solution to problem \eqref{eq:FluidOpt10}, denoted by $(\boldsymbol \tau_*, \boldsymbol a_*)$, has elements
\begin{equation*}
\begin{split}
\tau_*^{(\ell, m)} =  M^{(\ell)}  / L  \qandq
a_*^{(\ell, m)} = \frac{1}{2} (1-\rho_k) \lambda_k \of{\tau_*^{(\ell, m)} }^2 , \quad  1 \le m \le L, \quad 1 \le \ell \le N_k.
\end{split}
\end{equation*}

Note that problem \eqref{eq:FluidOpt10} is a relaxation of problem \eqref{eq:FluidOpt9} because for any feasible solution
$(\boldsymbol \tau, \boldsymbol a)$ to \eqref{eq:FluidOpt10}, there does not necessarily exist a corresponding $q^L \in \mathcal Q^L_k$ such that,
for all $1 \le m \le L$ and $1 \le \ell \le N_k$,
\begin{equation} \label{eq:RelaxedConstraints}
\begin{split}
&a^{(\ell, m)} = \int_{u^{(m-1)}}^{u^{(m)}}  q^L_{k}(s) \mathbf 1_{\{q^L_{k}(s) \geq \alpha_k^{(\ell)}\}} ds \qandq
\tau^{(\ell, m)} = \int_{u^{(m-1)}}^{u^{(m)}}  \mathbf 1_{\{q^L_{k}(s) \geq \alpha_k^{(\ell)}\}} ds.
\end{split}
\end{equation}
Hence, if there exists a $q^L_{e} \in \mathcal Q^L_k$ such that \eqref{eq:RelaxedConstraints} holds for $(\boldsymbol \tau_*, \boldsymbol a_*)$, then $q^L_{e}$
is a solution to \eqref{eq:FluidOpt9}.

Let $q_{exh,k}$ denote the trajectory of queue $k$ in $q_{exh}$. (Note that $q_{exh} \in \mathcal Q^L_k$.)
It can be verified that \eqref{eq:RelaxedConstraints} indeed holds for $(q_{exh,k}, \boldsymbol \tau_*, \boldsymbol a_*)$, namely,
for all $1 \le m \le L$ and $1 \le \ell \le N_k$,
\begin{equation*}
\begin{split}
a_*^{(\ell, m)} = \int_{u^{(m-1)}}^{u^{(m)}}  q_{exh, k}(s) \mathbf 1_{\{q_{exh,k}(s) \geq \alpha_k^{(\ell)}\}} ds \qandq
\tau_*^{(\ell, m)} = \int_{u^{(m-1)}}^{u^{(m)}}  \mathbf 1_{\{q_{exh,k}(s) \geq \alpha_k^{(\ell)}\}} ds.
\end{split}
\end{equation*}
Thus, $q_{exh}$ solves \eqref{eq:FluidOpt9}, and therefore also the equivalent problem \eqref{eq:FluidOpt4}. In turn,
$q_{exh}$  is a solution to the $L$-cycle optimization problem \eqref{eq:FluidOpt2}.
Since the arguments hold for each $L \in \NN$, $q_{exh}$ is optimal to \eqref{eq:FluidOpt}. \hfill \qed

\end{appendix}

\bibliographystyle{ormsv080_abv}
\bibliography{bibfile}

\end{document}